\documentclass[11pt]{amsart}

\usepackage{graphicx}
\usepackage{latexsym}    
\usepackage{amssymb}   
\usepackage{amsmath}    
\usepackage{amsbsy}
\usepackage{amsthm}
\usepackage{amsgen}
\usepackage{amsfonts}
\usepackage{mathtools}
\usepackage{array}
\usepackage[all]{xy}    
\usepackage{tikz}
\usetikzlibrary{decorations.pathmorphing}
\usepackage{hyperref}
\usepackage{color}
\usepackage{verbatim}
\usepackage{url}
\usepackage{enumerate}
\usepackage{enumitem}

\tikzset{snake it/.style={decorate, decoration=snake}}

\usepackage[normalem]{ulem}

\newtheorem{thm}{Theorem}[section]
\newtheorem{cor}[thm]{Corollary}
\newtheorem{lem}[thm]{Lemma}
\newtheorem{prop}[thm]{Proposition}

\newtheorem{theorem}{Theorem}

\newtheorem{corollary}[theorem]{Corollary}

\theoremstyle{definition}
\newtheorem{defn}[thm]{Definition}
\newtheorem*{defn*}{Definition}

\newtheorem*{conj*}{Conjecture}
\newtheorem{example}[thm]{Example}
\newtheorem{rem}[thm]{Remark}
\newtheorem*{obs}{Key Observation}

\newtheorem*{ack}{Acknowledgements}

\newenvironment{entry}%
{\begin{list}{}{%
\setlength{\labelwidth}{25pt}%
\setlength{\leftmargin}%
{\labelwidth+\labelsep}}}%
{\end{list}}

\newcommand{\DesLabel}[1]%
{\raisebox{-20pt}[1em][0pt]{%
\makebox[\labelwidth][l]%
{\parbox[t]{\labelwidth}{%
\hspace{4pt}\textbf{($\mathcal{#1}$)}}}}}
\newenvironment{Des}%
{\begin{entry}}%
{\end{entry}}

\numberwithin{equation}{section}

\newcommand{\scal}[1]{\langle #1 \rangle}

\newcommand{\pd}[1]{{\frac{\partial}{\partial #1}}}
\newcommand{\tpd}[1]{{\frac{\partial}{\partial #1}}}
\newcommand{\PD}[2]{{\frac{\partial #1}{\partial #2}}}

\newcommand{\stack}[2]{\genfrac{}{}{0pt}{0}{#1}{#2}}

\DeclareMathOperator{\Diff}{Diff}

\DeclareMathOperator{\im}{Im}
\DeclareMathOperator{\Ker}{Ker}
\DeclareMathOperator{\Span}{span}

\DeclareMathOperator{\End}{End}
\DeclareMathOperator{\ch}{ch}
\DeclareMathOperator{\Hom}{Hom}
\DeclareMathOperator{\Pf}{Pf}
\newcommand{\Ad}{\operatorname{Ad}}
\newcommand\tr{\operatorname{tr}}
\newcommand\Vol{{\operatorname{vol}}}
\newcommand\sign{\operatorname{sign}}

\newcommand{\SO}{\mathrm{SO}}
\renewcommand{\O}{\mathrm{O}}
\newcommand{\Spin}{\mathrm{Spin}}
\newcommand{\SU}{\mathrm{SU}}
\newcommand{\U}{\mathrm{U}}
\newcommand{\Sp}{\mathrm{Sp}}
\newcommand{\GL}{\mathrm{GL}}
\newcommand{\Pin}{\mathrm{Pin}}
\newcommand{\Pjn}{\mathrm{Pjn}}
\newcommand{\sph}{\mathbf{S}}
\newcommand{\RP}{\mathbf{RP}}
\newcommand{\HP}{\mathbf{HP}}

\newcommand{\Id}{\mathsf{1}}

\newcommand{\Q}{\mathbb{Q}}

\renewcommand{\leq}{\leqslant}
\renewcommand{\geq}{\geqslant}

\def\ol{\overline}
\def\ul{\underline}
\def\wt{\widetilde}
\def\x{\times}
\def\ox{\otimes}

\def\lra{\longrightarrow}

\def\hra{\hookrightarrow}

\def\In{\subseteq}

\def\R{\mathbb{R}}
\def\Z{\mathbb{Z}}
\def\C{\mathbb{C}}
\def\N{\mathbb{N}}
\def\DD{\mathbf{D}}
\def\D{\mathfrak{D}}
\def\HH{\mathbb{H}}
\newcommand{\B}{\mathfrak{B}}

\def\<{\langle}
\def\>{\rangle}

\def\mc{\mathcal}
\def\mf{\mathfrak}

\def\V{\mc{V}}
\def\H{\mc{H}}
\def\eps{\epsilon}
\def\ve{\varepsilon}
\def\vphi{\varphi}
\def\vtheta{\vartheta}
\def\vg{\check g}
 
\def\pr{\mathrm{pr}}
\def\wh{\widehat}
\def\wt{\widetilde}
\def\bq{/\!\!/}
\def\Mab{M^7_{\ul{a}, \ul{b}}}
\def\Pab{P^{10}_{\ul{a}, \ul{b}}}
\def\hPab{\widehat P^{13}_{\ul{a}, \ul{b}}}
\def\Pina{\mathrm{Pin}(2)_{\ul{a}}}
\def\Pjnb{\mathrm{Pjn}(2)_{\ul{b}}}
\def\Bab{B^4_{\ul{a}, \ul{b}}}

\def\id{\mathrm{id}}
\newcommand{\so}{\mf{so}}

\def\bpm{\begin{pmatrix}}
\def\epm{\end{pmatrix}}
\def\bvm{\begin{vmatrix}}
\def\evm{\end{vmatrix}}
\def\bsm{\left(\begin{smallmatrix}}
\def\esm{\end{smallmatrix}\right)}
\def\beq{\begin{equation}}
\def\eeq{\end{equation}}

\begin{document}



\title[Highly connected $7$-manifolds and non-negative curvature]{Highly connected $7$-manifolds and non-negative sectional curvature}



\author[S.\ Goette]{S.\ Goette$^\dagger$}
\address[Goette]{
Mathematisches Institut, Universit\"at Freiburg, Germany.}
\email{sebastian.goette@math.uni-freiburg.de}


\author[M.\ Kerin]{M.\ Kerin$^\dagger$$^*$}
\address[Kerin]{School of Mathematics, Statistics and Applied Mathematics, NUI Galway, Ireland.}
\email{martin.kerin@nuigalway.ie}
\thanks{$^\dagger$ Received partial support from DFG Priority Program \emph{Geometry at Infinity}} 
\thanks{$^*$ Received support from SFB 878: \emph{Groups, Geometry \& Actions} at WWU M\"unster.}


\author[K.\ Shankar]{K.\ Shankar$^*$}
\address[Shankar]{Department of Mathematics, University of Oklahoma, U.S.A.}
\email{Krishnan.Shankar-1@ou.edu}
\thanks{}

\date{\today}

\dedicatory{It is our pleasure to dedicate this article to Karsten Grove, Wolfgang Meyer and Wolfgang Ziller on their respective $70^{th}\!$, $80^{th}\!$ and $65^{th}\!$ birthdays.}




\begin{abstract}
In this article, a six-parameter family of highly connected $7$-manifolds which admit an $\SO(3)$-invariant metric of non-negative sectional curvature is constructed and the Eells-Kuiper invariant of each is computed.  In particular, it follows that all exotic spheres in dimension $7$ admit  an $\SO(3)$-invariant metric of non-negative curvature.
\end{abstract}

\maketitle





A manifold $M$ of dimension $2n+1$ or $2n$ is called \emph{highly connected} if it is $(n-1)$-connected, that is, if the homotopy groups $\pi_i(M)$ are trivial for all $i \leq n-1$.  As the topology of such manifolds is relatively simple, they have received much attention: see, for example, \cite{Ba, Cr, CE, CN, KiSh, Sm2, KM, Wa1, Wa2, DWi}.  In fact, it was Milnor's quest to understand such manifolds which led to the discovery of $7$-dimensional manifolds which are homeomorphic, but not diffeomorphic, to the standard sphere $\sph^7$ \cite{Mi1, Mi3}.  Just as in the case of $\sph^7$, these \emph{exotic spheres} occur as the total spaces of $\sph^3$-bundles over $\sph^4$.  By a combination of the work of Milnor \cite{Mi2} (cf.\ \cite{KM}), Smale \cite{Sm1} and Eells and Kuiper \cite{EK}, it was subsequently shown that there are $28$ possible oriented differentiable structures on the $7$-dimensional (topological) sphere, $16$ of which are \emph{Milnor spheres}, that is, obtained as $\sph^3$-bundles over $\sph^4$.  If one forgets the orientation, there are $15$ possible diffeomorphism types, $11$ occurring as Milnor spheres.

\begin{theorem}
\label{T:thmA}
All exotic $7$-spheres admit an $\SO(3)$-invariant Riemannian metric of non-negative sectional curvature.
\end{theorem}

In \cite{GM}, Gromoll and Meyer showed that one of the exotic Milnor spheres can be written as a biquotient and, hence, admits a Riemannian metric with non-negative sectional curvature.  Furthermore, by exploiting the biquotient structure, it was demonstrated in \cite{EsKe, FW} that the Gromoll-Meyer sphere can be equipped with a metric of non-negative curvature such that all sectional curvatures are positive on an open, dense set, while a further deformation of this metric to globally positive curvature has been proposed in \cite{PW}.  Unfortunately, the Gromoll-Meyer sphere is the only exotic sphere in any dimension which can be written as a biquotient \cite{KZ, To}.  Nevertheless, the remaining Milnor spheres were shown to admit non-negative curvature by Grove and Ziller \cite{GZ}, as a consequence of their investigation of cohomogeneity-one manifolds.  On the other hand, the fact that all exotic $7$-spheres admit a metric of positive Ricci curvature was established by Wraith \cite{Wr}, while Searle and Wilhelm \cite{SW} recently demonstrated that each admits a metric having, simultaneously, positive Ricci curvature and almost-non-negative sectional curvature.

Despite not being biquotients, Dur\'an, P\"uttmann and Rigas \cite{DPR} (see also \cite{FW1}) have shown that all exotic spheres in dimension $7$ can be constructed in a similar way to the Gromoll-Meyer sphere.  They tried, without success, to equip the four non-Milnor exotic spheres with a metric of non-negative curvature.  In the present work, we achieve this via a different construction.

\begin{theorem}
\label{T:thmB}
For all triples $\ul a = (a_1, a_2, a_3)$, $\ul b = (b_1, b_2, b_3) \in \Z^3$ of integers congruent to $1$ mod $4$ and satisfying $\gcd(a_1, a_2 \pm a_3) = \gcd(b_1, b_2 \pm b_3) = 1$, there is a $2$-connected, $7$-dimensional manifold $\Mab$ which admits an $\SO(3)$-invariant metric of non-negative sectional curvature and for which $H^4(\Mab; \Z) = \Z_{|n|}$, whenever 
$$
n = \frac{1}{8} \det \bpm  a_1^2 & b_1^2 \\ a_2^2 - a_3^2 & b_2^2 - b_3^2 \epm \neq 0.
$$
If $n = 0$, then $H^3(\Mab; \Z) = H^4(\Mab; \Z) = \Z$.
\end{theorem}

The manifolds $\Mab$ with $a_1 = b_1 = 1$ are diffeomorphic to those studied by Grove and Ziller \cite{GZ} and consist of all $\sph^3$-bundles over $\sph^4$.   By choosing the parameters appropriately, one obtains the Milnor spheres.  Grove and Ziller constructed metrics of non-negative curvature by first showing that there is a four-parameter family of non-negatively curved, $10$-dimensional, cohomogeneity-one manifolds consisting of all principal $(\sph^3 \x \sph^3)$-bundles over $\sph^4$, and then taking the associated $\sph^3$-bundles over $\sph^4$ with their induced metrics.  

By observing that, in the Grove-Ziller case, the associated-bundle construction is equivalent to taking a quotient by a free $\sph^3$ action, it is natural to look for a more general collection of examples.  As it turns out, there is a larger six-parameter family of non-negatively curved, $10$-dimensional cohomogeneity-one manifolds $P^{10}_{\ul{a}, \ul{b}}$ which, under the $\gcd$ conditions of Theorem \ref{T:thmB}, admit a free, isometric action by $\sph^3$.  The manifolds $\Mab$ are precisely the quotients of the $P^{10}_{\ul{a}, \ul{b}}$ by this action, and are each equipped with the induced metric.  In general, the $\Mab$ are not $\sph^3$-bundles over $\sph^4$ in any obvious way.  For appropriate choices of the parameters $\ul a$, $\ul b \in \Z^3$, it is clear that one obtains $7$-dimensional manifolds which are homotopy spheres.

The major difficulty in this project is to determine the diffeomorphism type of the manifolds $\Mab$.  By the work of Crowley \cite{Cr}, it suffices to compute the Eells-Kuiper invariant $\mu(\Mab)$ \cite{EK} and $q$-invariant \cite{Cr} of these spaces.  In particular, for the homotopy spheres, the Eells-Kuiper invariant determines the diffeomorphism type.  In Theorem \ref{T:EKinv}, the general formula for $\mu(\Mab)$ below has been computed via a modification of the methods of \cite{Gojems}, in which the first named author determined the diffeomorphism type of a recently discovered example with positive curvature, see \cite{De, GVZ}.  For $q, p_1, p_2, p_3 \in \Z$ with $\gcd(q, p_i) = 1$ for all $i = 1,2,3$, let 
$$
\mc D(q; p_1, p_2, p_3) = \frac{1}{2^6 \cdot 7 \cdot q^2} \sum_{l=1}^{|q|-1} \ \ \sum_{\mathclap{\substack{(i,j,k) = \\ \circlearrowright (1,2,3)}}} p_i \! \left( 
\frac{14 \cos(\frac{p_i \pi l}{q}) + \cos(\frac{p_j \pi l}{q})\cos(\frac{p_k \pi l}{q})}
{\sin^2(\frac{p_i \pi l}{q})\sin(\frac{p_j \pi l}{q})\sin(\frac{p_k \pi l}{q})}
\right) 
$$
denote the corresponding \emph{generalised Dedekind sum}.  In particular, if $q = 1$, then $\mc D(q; p_1, p_2, p_3) = 0$.  Moreover, $\mc D(q; p_1, p_2, p_3)$ is invariant under permutations of the $p_i$ and satisfies $\mc D(q; p_1, p_2, - p_3) = - \mc D(q; p_1, p_2, p_3)$.  Note also that the generalised Dedekind sum of \cite[Definition 3.6]{Gojems} corresponds to $\mc D(p; 4,q,q)$, with $p,q \in \Z$ odd and relatively prime.

\begin{theorem}
\label{T:thmC}
Let $\Mab$ and $n \neq 0$ be as in Theorem \ref{T:thmB}, and set
$$
m = \frac{1}{8 a_1^2  b_1^2} \det \bpm  a_1^2 & b_1^2 \\ a_2^2 + a_3^2 + 8 & b_2^2 + b_3^2 + 8 \epm.
$$
Then the Eells-Kuiper invariant of $\Mab$ is given by
$$
\mu(\Mab) = \frac{|n| - a_1^2\, b_1^2\, m^2 }{2^5 \cdot 7 \cdot n} - D(\ul a) + D(\ul b)  \mod 1 \ \in \Q/\Z,
$$
where $D(\ul a) = \mc D(a_1; 4, a_2+a_3, a_2-a_3)$ and $D(\ul b) = \mc D(b_1; 4, b_2+b_3, b_2-b_3)$ are generalised Dedekind sums as defined above.  \end{theorem}

For some simple subfamilies it is easy to compute the generalised Dedekind sums $D(\ul a)$ and $D(\ul b)$ (see Example \ref{Ex:DedSum}), leading to a closed form for $\mu(\Mab)$ in these cases.  Recall that, following \cite{EK}, the non-Milnor exotic spheres have $28 \cdot \mu(\Mab) \in \{2, 5, 9, 12, 16, 19, 23, 26\}$.

\begin{corollary}
\label{C:nonMilnor}
Suppose $\ul a = (-3, 12 k - 3, 12 l + 1)$, $\ul b = (1, 4r + 1, 4 s + 1) \in \Z^3$.  Then 
$$
\mu(\Mab) =  \frac{|n| - a_1^2\, b_1^2\, m^2}{2^5 \cdot 7 \cdot n} - \frac{4l+1}{28} \in \Q/\Z.
$$
Moreover, the subfamily given by $(k,l,r,s) = (0, 0, r, r)$ has $H^4(\Mab; \Z) = 0$, hence consists of homotopy spheres, and the oriented diffeomorphism types of the non-Milnor exotic spheres are attained, for example, at $r \in \{-3, -1, 1, 2, 4, 8, 11, 15\}$.
\end{corollary}

The paper is organised as follows:  
In Section \ref{S:Prelim}, there can be found reviews of cohomogeneity-one manifolds, orbifolds and orbi-bundles, the Eells-Kuiper invariant and adiabatic limits.  Section \ref{S:CCC} begins with a review of the construction of Grove and Ziller, followed by the definition of the new non-negatively curved manifolds $\Mab$, before ending with the computation of the cohomology ring of $\Mab$.  In Section \ref{S:EK} a new metric is defined on $\Mab$ to facilitate the computation of the Eells-Kuiper invariant of $\Mab$.  With respect to this metric, various Chern-Weil characteristic forms and numbers are computed, as well as the individual terms in the adiabatic limit of the Eells-Kuiper invariant.

\ack{The authors wish to thank Diarmuid Crowley, Anand Dessai, Karsten Grove, Marco Radeschi and Wolfgang Ziller for important comments, for helpful and stimulating discussions, and for their interest in this work.  Particular thanks are extended to Burkhard Wilking for his suggestions and support, and for facilitating this collaboration with invitations to visit M\"unster.  K.\ Shankar wishes to thank the Mathematics Institute in M\"unster for its hospitality during his sabbatical in 2015-16, during which time this project was initiated.  Some of this work was completed while M.\ Kerin was a Research Member at MSRI under the Program on Differential Geometry and he wishes to thank the institute for its hospitality.  Finally, the authors would like to express their gratitude to the referees, for their careful reading of the paper and comments which led to several improvements.
}


\section{Preliminaries}
\label{S:Prelim}


\subsection{Cohomogeneity-one manifolds and principal bundles} \hspace*{1mm}\\
\label{SS:Cohom1}

Because of their importance in the construction of the manifolds $\Mab$, this subsection is devoted to a brief review of cohomogeneity-one manifolds.  A more elaborate discussion can be found in, for example, \cite{GZ}.

Let $G$ be a compact Lie group acting smoothly on a closed, connected, smooth manifold $M$ via $G \x M \to M,\ (g,p) \mapsto g \cdot p$.  For each $p \in M$, the \emph{isotropy group} at $p$ is the subgroup $G_p = \{g \in G \mid g \cdot p = p\} \In G$, and the \emph{orbit} through $p$ is the submanifold $G \cdot p = \{g \cdot p \in M \mid g \in G \} \In M$.  Since $G$ acts transitively on $G \cdot p$, there is a diffeomorphism $G \cdot p \cong G/G_p$, and there is  a foliation of $M$ by $G$-orbits. 

The action $G \x M \to M$ is said to be of \emph{cohomogeneity one} if there is an orbit of codimension one or, equivalently, if $\dim(M/G) = 1$.  In such a case, the manifold $M$ is called a \emph{cohomogeneity-one ($G$-)manifold}.  If, in addition, $\pi_1(M)$ is assumed to be finite, then the orbit space $M/G$ can be identified with a closed interval.  By fixing an appropriately normalised $G$-invariant metric on $M$, it may be assumed that $M/G = [-1,1]$.  Let $\pi: M \to M/G = [-1,1]$ denote the quotient map.  The orbits $\pi^{-1}(t)$, $t \in (-1,1)$, are called \emph{principal orbits} and the orbits $\pi^{-1}(\pm 1)$ are called \emph{singular orbits}.

Choose a point $p_0 \in \pi^{-1}(0)$ and consider a geodesic $c:\R \to M$ orthogonal to all the orbits, such that $c(0) = p_0$ and $\pi \circ c|_{[-1,1]} = \id_{[-1,1]}$.  Then, for every $t \in (-1,1)$, one has $G_{c(t)} = G_{p_0} \In G$, and this \emph{principal isotropy group} will be denoted by $H \In G$.  If $p_\pm = c(\pm 1) \in M$, denote the \emph{singular isotropy groups} $G_{p_\pm}$ by $K_\pm$ respectively.

By the slice theorem, $M$ can be decomposed as the union of two disk bundles, over the singular orbits $G/K_- = \pi^{-1}(- 1)$ and $G/K_+ = \pi^{-1}(+ 1)$ respectively, which are glued along their common boundary $G/H = \pi^{-1}(0)$:
$$
M = (G \x_{K_-} \DD^{l_-}) \cup_{G/H} (G \x_{K_+} \DD^{l_+}) \, .
$$
In particular, since the principal orbit $G/H$ is the boundary of both disk bundles, it follows that $K_\pm/H = \sph^{l_\pm-1}$, where $l_\pm$ are the respective codimensions of $G/K_\pm$ in $M$.

Conversely, given any chain $H \In K_\pm \In G$, with $K_\pm/H = \sph^{d_\pm}$, one can construct a cohomogeneity-one $G$-manifold $M$ with codimension $d_\pm + 1$ singular orbits.  For this reason, a cohomogeneity-one manifold is conveniently represented by its group diagram:
$$
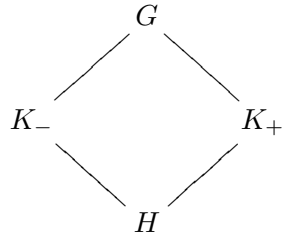
\xymatrix{
& G & \\
K_- \ar@{-}[ur] & & K_+ \ar@{-}[ul] \\
&  H \ar@{-}[ur] \ar@{-}[ul] & 
}
$$

In \cite{GZ}, the authors determined a sufficient condition for a cohomogeneity-one manifold to admit non-negative curvature.

\begin{thm}[\cite{GZ}]
\label{T:GZ}
Let $G$ be a compact Lie group acting on a manifold $M$ with cohomogeneity one.  If the singular orbits are of codimension $2$, then $M$ admits a $G$-invariant metric of non-negative sectional curvature.
\end{thm}

Given a cohomogeneity-one $G$-manifold $M$, let $j_\pm : K_\pm \to G$ and $i_\pm : H \to K_\pm$ denote the respective inclusion maps.  Suppose that $L$ is a compact Lie group and that there are homomorphisms $\vphi_\pm : K_\pm \to L$ such that $\vphi_- \circ i_- =  \vphi_+ \circ i_+$.  Then, by \cite[Prop.\ 1.6]{GZ}, one can construct a cohomogeneity-one $(G \x L)$-manifold $P$ with group diagram
$$
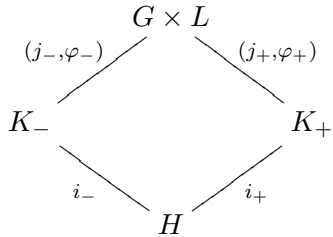
\xymatrix{
& G \x L & \\
K_- \ar@{-}[ur]^{(j_-, \vphi_-)} & & K_+ \ar@{-}[ul]_{(j_+, \vphi_+)} \\
&  H \ar@{-}[ur]_{i_+} \ar@{-}[ul]^{i_-} &
}
$$
such that the subaction by $\{e\} \x L \In G \x L$ is free and induces a principal $L$-bundle $L \to P \to M$.  In particular, under this construction the co\-dimensions of the singular orbits in $M$ and $P$ are equal.


\subsection{The Eells-Kuiper invariant} \hspace*{1mm}\\
\label{SS:EK-inv}

When Milnor discovered exotic spheres in \cite{Mi1}, he used an invariant based upon the Hirzebruch Signature Theorem (see \cite{Hi}) to establish that the $7$-manifolds he had constructed could not be diffeomorphic to $\sph^7$.  Soon afterwards, Eells and Kuiper \cite{EK} found an invariant, based upon the integrality of the $\hat A$-genus for spin manifolds \cite[Cor.\ 3.2]{BH}, which completely determines the diffeomorphism type of $7$-dimensional homotopy spheres.  For simplicity, the following review of the Eells-Kuiper invariant will focus on dimensions $7$ and $8$.

Suppose $X$ is a closed, smooth, $8$-dimensional manifold which is, in addition, oriented and spin, that is, the first and second Stiefel-Whitney classes $w_1(X) \in H^1(X; \Z_2)$ and $w_2(X) \in H^2(X; \Z_2)$ vanish.  Let $p_1(X) \in H^4(X; \Q)$ and $p_2(X) \in H^8(X; \Q)$ denote the rational Pontrjagin classes of (the tangent bundle of) $X$, and let $[X] \in H_8(X; \Z)$ denote the fundamental class of $X$.  Finally, let $\sigma(X)$ denote the signature of the quadratic form $\alpha \mapsto \alpha^2$ on $H^4(X; \Q)$.  From the Signature Theorem \cite{Hi} and Corollary 3.2 of \cite{BH}, it is known that both the signature
$$
\sigma(X) = \frac{1}{45}(-p_1(X)^2 + 7\, p_2(X))[X]
$$
and the $\hat A$-genus
$$
\hat A(X) = \frac{1}{2^7 \cdot 45} (7\, p_1(X)^2 - 4\, p_2(X)) [X]
$$
are integers.  By taking an appropriate linear combination, one can easily deduce that
\beq
\label{E:Ahat}
\hat A(X) = \frac{1}{2^7 \cdot 7} (p_1(X)^2[X] - 4\, \sigma(X)) \in \Z.
\eeq

Suppose now that $M$ is a $7$-dimensional, closed, oriented, smooth, $2$-connected manifold with $H^4(M; \Z)$ finite.  Notice that, in particular, $M$ is spin.  Since the spin cobordism group in dimension $7$ is trivial, one can always find a compact, oriented, smooth, $8$-dimensional, spin \emph{coboundary} $W$, that is, a manifold with boundary $\partial W = M$.

From the long exact sequence in cohomology for the pair $(W, M)$, one obtains an isomorphism
$$
j: H^4(W,M; \Q) \stackrel{\cong}{\lra} H^4(W; \Q). 
$$
Therefore, the rational Pontrjagin class $p_1(W) \in  H^4(W; \Q)$ can be pulled back to define a Pontrjagin class $j^{-1}(p_1(W)) \in  H^4(W, M; \Q)$ on $(W,M)$.  Moreover, there is a well-defined fundamental class $[W,M] \in H_8(W,M; \Q)$ for the pair $(W,M)$, and one can define the signature $\sigma(W,M)$ to be the signature of the quadratic form $\alpha \mapsto \alpha^2$ on $H^4(W,M; \Q)$.  By analogy with the expression for the $\hat A$-genus in \eqref{E:Ahat} above, this motivates the following definition.

\begin{defn}[\cite{EK}]
\label{D:EKinv}
Let $M$ be a $7$-dimensional, closed, oriented, smooth, $2$-connected manifold with $H^4(M; \Z)$ finite, and $W$ be a compact, oriented, smooth, $8$-dimensional, spin coboundary.  Then the \emph{Eells-Kuiper invariant} of $M$ is given by
\beq
\label{E:EKinv}
\mu(M) = \frac{1}{2^7 \cdot 7} \left((j^{-1}p_1(W))^2[W,M] - 4 \, \sigma(W,M)\right) \hspace{-3mm} \mod 1 \ \in \Q/\Z.
\eeq
\end{defn}

In particular, the Eells-Kuiper invariant measures the defect of the right-hand side from being the $\hat A$-genus of a closed, spin manifold, and has the following properties:
\begin{enumerate}
\item $\mu(M)$ is independent of the choice of coboundary $W$.\vspace{1mm}
\item $\mu(M)$ respects orientation, i.e. $\mu(-M) = -\mu(M)$.\vspace{1mm}
\item $\mu(M)$ is additive, i.e. $\mu(M_1 \# M_2) = \mu(M_1) + \mu(M_2)$.
\end{enumerate}

Although it is quite simple to define, the Eells-Kuiper invariant is difficult to compute in practice.  One approach is to appeal to the generalisation by Atiyah, Patodi and Singer of the Atiyah-Singer Index Theorem to manifolds with boundary \cite{APS}.  

In brief, equip $M$ with a Riemannian metric $g_M$, with Levi-Civita connection $\nabla^{TM}$, and extend $g_M$ to a Riemannian metric on $W$ which is product near the boundary.  Let $\D$ and $\B$ be the \emph{spin-Dirac operator} and \emph{odd signature operator} on $(M, g_M)$ respectively, that is, $\D$ is the usual Dirac operator on the spinor bundle of $(M, g_M)$ and $\B$ is the restriction of the operator $\pm( * d - d *)$ to differential forms on $M$ of even degree, where $*$ denotes the Hodge-$*$ operator.  Then, by applying the Atiyah-Patodi-Singer Index Theorem to both $\D$ and $\B$ and following the scheme laid out in \cite[Prop.\ 2.1]{KS} (cf.\ \cite{Do1}), one obtains
\begin{align}
\label{E:EKeta}
\mu(M) = \frac{h + \eta}{2}(\D) &+ \frac{1}{2^5 \cdot 7} \, \eta(\B) \\
\nonumber
& - \frac{1}{2^7 \cdot 7} \int_M p_1(TM, \nabla^{TM}) \wedge \hat p_1(TM, \nabla^{TM}) \ \in \Q/\Z.
\end{align}
The terms in this formula require some explanation.  The first term, $h(\D)$, is simply $\dim \Ker(\D)$, the dimension of the space of harmonic spinors.  The terms under the integral are differential forms, namely, 
$$
p_1(TM, \nabla^{TM}) = \frac{1}{8\pi^2} \, \tr ((\Omega^{TM})^2)
$$ 
is the Pontrjagin $4$-form on $M$ obtained from the curvature $2$-form $\Omega^{TM}$ via Chern-Weil theory, while $\hat p_1(TM, \nabla^{TM})$ is a $3$-form on $M$ such that
$$
d (\hat p_1(TM, \nabla^{TM})) = p_1(TM, \nabla^{TM}).
$$  
Such a $3$-form exists, since $H^4_{dR}(M) \cong H^4(M; \R) = 0$ by assumption.  

The terms $\eta(\D)$ and $\eta(\B)$ are the $\eta$-invariants of the operators $\D$ and $\B$.  Recall that, if $\mathsf D$ is a Dirac operator (that is, a first-order, self-adjoint differential operator, such that $\mathsf D^2$ is a Laplacian), then its eigenvalues are real numbers and the $\eta$-invariant of $\mathsf D$ is defined by 
$$
\eta(\mathsf D) = \eta_{\mathsf D}(0),\  \textrm{ where }  \eta_{\mathsf D}(z) = \sum_{\lambda \neq 0} \frac{\sign(\lambda)}{|\lambda|^z}, \ \  z \in \C,
$$
the sum being over the non-trivial eigenvalues of $\mathsf D$ (counting multiplicities).  Therefore, $\eta(\mathsf D)$ measures the asymmetry of the eigenvalues of $\mathsf D$ about $0$.  The invariant $\eta(\mathsf D)$ can also be thought of as the defect in the corresponding Atiyah-Singer Index Theorem due to $W$ not being a closed manifold.

The primary benefit of the formula \eqref{E:EKeta} for $\mu(M)$ is that the right-hand side is written entirely in terms of the geometry of $M$, that is, the coboundary $W$ no longer plays a role.  Thus, for an appropriate choice of metric $g_M$, it is reasonable to expect that \eqref{E:EKeta} can be used to compute $\mu(M)$.


\subsection{Orbifolds, orbi-bundles and invariants} \hspace*{1mm}\\
\label{SS:InertOrb}

As orbifolds will play a significant role in the rest of the article, it is useful to recall some definitions and notation (cf.\ \cite{Gojems}).

\begin{defn} 
\label{D:orb}
Let $G$ be a compact Lie group acting effectively on $\R^n$.  An \emph{$n$-dimensional smooth $G$-orbifold} is a second-countable, Hausdorff space $B$ such that:
\begin{enumerate}
\item For each point $b \in B$ there exists a neighbourhood $U \In B$ of $b$, an open subset $V \In \R^n$ invariant under the effective action $\rho : \Gamma \hra G \to \GL(n, \R)$ of a finite group $\Gamma$, and a homeomorphism
$$
\psi : \rho(\Gamma) \backslash V \to U \ \ \text{ with }\ \psi(0) = b.
$$
The homeomorphism $\psi$ is called an \emph{orbifold chart}, $\Gamma$ the \emph{isotropy group} of $b \in B$ and $\rho$ the \emph{isotropy representation} at $b$.  Let $\wt \psi : V \to U$ denote the composition of $\psi$ with the projection $V \to \rho(\Gamma) \backslash V$.

\item Let $b \in U \In B$ and $\psi : \rho(\Gamma) \backslash V \to U$ be as above.  For $b' \in U$, let $\psi' : \rho'(\Gamma') \backslash V' \to U'$ denote the corresponding orbifold chart.  Then there exists a smooth, open embedding $\vphi : (\wt \psi')^{-1}(U \cap U') \to V$ and a group homomorphism $\vtheta : \Gamma' \to \Gamma$ such that, for all $\gamma' \in \Gamma'$,
$$
\vphi \circ \rho'(\gamma') = \rho(\vtheta(\gamma')) \circ \vphi 
$$ 
and, for all $v' \in  (\tilde \psi')^{-1}(U \cap U') \In V'$,
$$
\wt \psi(\vphi(v')) = \wt \psi'(v').
$$
The map $\vphi$ is called a \emph{coordinate change} and $\vtheta$ an \emph{intertwining homomorphism}.
\end{enumerate}
If $G \In \O(n)$, then the term \emph{$n$-orbifold} will be used for brevity.  If $G \In \SO(n)$ and all coordinate changes are orientation preserving, then the $n$-orbifold $B$ is called \emph{oriented}.
\end{defn}

Notice that if $\vphi$ is a coordinate change as above with intertwining homomorphism $\vtheta$, then, for each $\gamma \in \Gamma$, the map $\rho(\gamma) \circ \vphi$ is another coordinate change with intertwining homomorphism $\gamma \cdot \vtheta$, where $(\gamma \cdot \vtheta)(\gamma') = \gamma \vtheta(\gamma')\gamma^{-1}$.  Moreover, the assumption that isotropy representations are effective ensures that intertwining homomorphisms are unique.

\begin{defn}
\label{D:orbibun}
Let $B$ be a smooth $G$-orbifold and $F$ a smooth manifold.  An \emph{orbi-bundle with fibre $F$} is a map $\pi$ from a topological space $M$ to $B$ such that:
\begin{enumerate}
\item For each $b \in B$, there exists an orbifold chart $\psi : \rho(\Gamma) \backslash V \to U \In B$ around $b$, a fibre-preserving, smooth action $\wh \rho$ of $\Gamma$ on $V \x F$ such that the projection $\pr_1: V \x F \to V$ is $\Gamma$-equivariant, and a homeomorphism $\wh \psi :  \wh\rho(\Gamma) \backslash (V \x F) \to \pi^{-1}(U)$ such that the diagram 
$$
\xymatrix{
V \x F \ar[r] \ar[d]_{\pr_1} & \wh \rho(\Gamma) \backslash (V \x F) \ar[r]^-{\wh \psi} \ar[d] & \pi^{-1}(U) \ar[d]^{\pi} \\
V \ar[r] & \rho(\Gamma) \backslash V \ar[r]^-{\psi} \ar[r] & U
}
$$
commutes. 
\item Let $\psi : \rho(\Gamma) \backslash V \to U$ and $\psi' : \rho'(\Gamma') \backslash V' \to U'$ be orbifold charts in $B$ as in Definition \ref{D:orb}, with associated coordinate change $\vphi$ and intertwining homomorphism $\vtheta$.  Let $\wh \rho$, $\wh \rho'$ be the corresponding actions and $\wh\psi$, $\wh \psi'$ the corresponding homeomorphisms as above.  Finally, let $q: V \x F \to \wh \rho(\Gamma) \backslash (V \x F)$ and $q': V' \x F \to \wh \rho'(\Gamma') \backslash (V' \x F)$ denote the quotient maps.  Then there is a smooth, open embedding $\wh \vphi : (\wh\psi' \circ q')^{-1}(U \cap U') \to V \x F$ (a coordinate change) such that, for all $\gamma' \in \Gamma'$,
$$
\wh \vphi \circ \wh\rho'(\gamma') = \wh\rho(\vtheta(\gamma')) \circ \wh\vphi 
$$
and, for all $(v', f) \in (\wh\psi' \circ q')^{-1}(U \cap U') \In V' \x F$,
$$
(\wh \psi \circ q) (\wh \vphi (v', f)) = (\wh \psi' \circ q')(v',f). 
$$
\end{enumerate}
If all of the fibre-preserving actions $\wh\rho$ are free, then the space $M$ carries the structure of a smooth manifold.  In this case, the map $\pi: M \to B$ is called a \emph{Seifert fibration}.

If the fibre $F$ is a vector space and if, in addition, all actions $\wh \rho$ and all coordinate changes $\wh \vphi$ are linear, then $\pi: M \to B$ is called a \emph{vector orbi-bundle}.

If $F$ is a Lie group $G$ and all actions $\wh \rho$ and all coordinate changes $\wh \vphi$ commute with the right action of $G$ on $F$, then $G$ acts on $M$ and $\pi: M \to B$ is called a \emph{principal $G$-orbi-bundle}.  
\end{defn}

Note that, given a principal $G$-orbi-bundle and a right action of $G$ on a manifold $F$, one can construct an \emph{associated orbi-bundle} with fibre $F$ and structure group $G$ with the properties above.
 
As discussed in \cite[Rem.\ 1.3]{Gojems}, a Seifert fibration can, equivalently, be described as a regular Riemannian foliation of $M$ with compact leaves.  The leaf space $B$ naturally has the structure of a smooth orbifold, while the generic leaves (which form an open, dense set in $M$) are each diffeomorphic to some fixed smooth manifold $F$.  The exceptional leaves are each finitely covered by $F$, and the projection map $\pi : M \to B$ has the properties listed in Definition \ref{D:orbibun}.

\begin{rem}
\label{R:Seifert}
If a compact Lie group $G$ acts almost freely on a manifold $M$ such that the sub-action of a closed subgroup $H \In G$ is free, then the quotient $M/G$ naturally inherits the structure of a smooth orbifold, the quotient $M/H$ is a smooth manifold, and the projection $\pi : M/H \to M/G$ is a Seifert fibration with fibre $G/H$.
\end{rem}

Using the local definitions above, an orbifold $B$ possesses a natural tangent orbi-bundle $TB \to B$ and can always be equipped with an (orbifold) Riemannian metric.

Furthermore, as all leaves of a Seifert fibration $\pi : M \to B$ are manifolds of a fixed dimension, it makes sense to talk about the vertical sub-bundle $\V$ of the tangent bundle $TM$, that is, the vector bundle given by vectors tangent to the leaves.  If $M$ is equipped with a Riemannian metric $g_M$, the horizontal sub-bundle $\H$ is defined as the bundle of all vectors orthogonal to the leaves.  Given a vector $w \in T_p M$, the vertical and horizontal components of $w$ will be denoted by $w^\V$ and $w^\H$ respectively.

At each point $p \in M$, the differential $d\pi_p|_{\H_p} : \H_p \to T_{\pi(p)} B$ is an isomorphism.  If $d\pi_p|_{\H_p}$ is, in addition, an isometry at each $p \in M$ with respect to the metrics $g_M$ and $g_B$, then, by a slight abuse of terminology, one may refer to $\pi : (M, g_M) \to (B, g_B)$ as a Riemannian submersion.

For vector orbi-bundles, it is possible to use the language of Definition \ref{D:orbibun} to define Whitney sums, tensor products, dual bundles and exterior products.  Similarly, spin vector orbi-bundles can be defined in a natural way analogous to that for vector bundles.  

There is also a natural notion of \emph{Dirac orbi-bundle} over a Riemannian orbifold that is analogous to the notion of a \emph{Dirac bundle} over a Riemannian manifold $(M,g_M)$, which consists of a complex vector bundle $E \to M$ equipped with a Hermitian metric $g_E$ and compatible connection $\nabla^E$, as well as a Clifford action $c:TM \to \End(E)$ that is skew-symmetric with respect to $g^E$ and satisfies a Leibniz rule with respect to $\nabla^E$ and the Levi-Civita connection $\nabla^{TM}$ of $(M, g_M)$.

Just as for manifolds, Chern-Weil theory can be applied to vector orbi-bundles.  Suppose that $E \to B$ is a vector orbi-bundle with connection $1$-form $\omega^E$ and curvature $2$-form $\Omega^E = d \omega^E + \omega^E \wedge \omega^E$.  Let $\nabla^E$ and $R^E = (\nabla^E)^2$ denote the induced connection and curvature, respectively.  Recall from Chern-Weil that the \emph{first Pontrjagin form} and \emph{Euler form} of $(E, \nabla^E)$ are defined by
\begin{align}
\begin{split}
\label{E:CWforms}
p_1(E, \nabla^E) &= \frac{1}{8 \pi^2} \tr((\Omega^E)^2) = \frac{1}{8 \pi^2} \tr((R^E)^2) \,, \\
e(E, \nabla^E) &= \frac{1}{4 \pi^2} \Pf(\Omega^E) = \frac{1}{4 \pi^2} \Pf(R^E) \,,
\end{split}
\end{align}
respectively, where $\Pf$ denotes the Pfaffian, that is, $\Pf = \det^{1/2}$.

\medskip
Associated to the base of a Seifert fibration $\pi : M \to B$ there is a further orbifold $\Lambda B$, which will be important in the computation of the Eells-Kuiper invariant.  

\begin{defn}  The \emph{inertia orbifold} $\Lambda B$ of an orbifold $B$ is the orbifold consisting of points $(b, [\gamma])$, where $b \in B$ and $[\gamma]$ denotes the $\Gamma$-conjugacy class of an element $\gamma$ of the isotropy group $\Gamma$ of $b$.
\end{defn}

In general, the inertia orbifold consists of several components.  In particular, the component of $\Lambda B$ corresponding to the identity element of each isotropy group is simply a copy of $B$ itself.  Other components are often called \emph{twisted sectors}.

The orbifold charts for $\Lambda B$ are obtained from those of $B$.  Suppose $\psi : \rho(\Gamma) \backslash V \to U$ is an orbifold chart around $b = \psi(0) \in B$.  For each $\gamma \in \Gamma$, let $V^\gamma$ denote the fixed-point set of the action of $\gamma$ on $V$ and let $Z_\Gamma (\gamma) \In \Gamma$ denote the centraliser of $\gamma$.  Then $Z_\Gamma (\gamma)$ acts on $V^\gamma$ via the restriction of $\rho$, although this action need not be effective.  The ineffective kernel of this action is a finite subgroup of $Z_\Gamma (\gamma)$ of order 
\beq
\label{E:multiplicity}
m(\gamma) = \#\{\sigma \in Z_\Gamma (\gamma) \mid \rho(\sigma)|_{V^\gamma} = \id_{V^\gamma}\}.
\eeq
In this way, $m(\gamma)$ defines a locally constant function on $\Lambda B$ and is called the \emph{multiplicity} of $(b,[\gamma]) \in \Lambda B$.  An orbifold chart for $\Lambda B$ around the point $(b, [\gamma])$ is given by the homeomorphism
$$
\psi_{[\gamma]} : Z_\Gamma (\gamma) \backslash V^\gamma \to \wt \psi(V^\gamma) \x \{[\gamma]\} \In \Lambda B.
$$
Note that the orbit space $Z_\Gamma (\gamma) \backslash V^\gamma$ is the same as that obtained by considering the effective action on $V^\gamma$ of the quotient of $Z_\Gamma(\gamma)$ by its ineffective kernel.

Suppose from now on that $B$ is an oriented Riemannian orbifold (as will be the case in the applications to follow).  Let $\mc N_\gamma \to V^\gamma$ denote the normal bundle to $V^\gamma$ in $V$.  Since $B$ is oriented, $\mc N_\gamma$ has even rank, say $2 k_\gamma$.  Since the action of $\gamma$ is effective on $V$, but trivial on $V^\gamma$, it follows that $\gamma$ must act effectively on the fibres of $\mc N_\gamma$ via (abusing notation) an element $\gamma \in \SO(2 k_\gamma)$.  Let $\tilde \gamma \in \Spin(2 k_\gamma)$ denote a lift of $\gamma$ under the natural projection $\Spin(2 k_\gamma) \to \SO(2 k_\gamma)$.  If the orbifold $B$ is also spin, such a lift is part of the orbifold spin structure.  Otherwise, the lift $\tilde \gamma$ is determined uniquely up to sign.  Consequently, the inertia orbifold $\Lambda B$ has a natural double cover 
$$
\wt{\Lambda B} = \{(b, [\tilde \gamma]) \mid \tilde \gamma \text{ is a lift of } \gamma\} \to \Lambda B,
$$
where orbifold charts for $\wt{\Lambda B}$ are constructed as for $\Lambda B$ and given by
$$
\psi_{[\tilde \gamma]} : Z_\Gamma (\gamma) \backslash V^\gamma \to \wt \psi(V^\gamma) \x \{[\tilde \gamma]\} \In \wt{\Lambda B}.
$$

In generalisations of index theorems to the orbifold setting, one has to take into account the local structure of the orbifold, that is, the action of the isotropy groups.  This is achieved by defining equivariant forms.  Recall that, for a Hermitian vector bundle $E$ with connection $\nabla^E$ and equipped with a parallel, fibre-preserving automorphism $g$, the equivariant Chern character is classically defined as
$$
\ch_g(E, \nabla^E) = \tr \left(g \, \exp\! \left( -\frac{\Omega^E}{2\pi i} \right) \right). 
$$

By the discussion above, the normal bundle $\mc N_\gamma \to V^\gamma$ is spin, hence has a principal $\Spin(2 k_\gamma)$-bundle $\Spin(\mc N_\gamma) \to V^\gamma$ associated to it.  There is a unique (complex) $\Spin(2 k_\gamma)$-representation $S$ of (complex) dimension $2^{k_\gamma}$, the \emph{spinor module}, which decomposes into irreducible, inequivalent $\Spin(2 k_\gamma)$-representations $S^\pm$ of dimension $2^{k_\gamma - 1}$ such that $S = S^+ \oplus S^-$.  Thus, there is a complex (local) \emph{spinor bundle} $\mc{S}(\mc N_\gamma) \to V^\gamma$, where 
$$
\mc{S}(\mc N_\gamma) = \Spin(\mc N_\gamma) \x_{\Spin(2 k_\gamma)} S
$$ 
and, given a local orientation of $\mc N_\gamma$, a natural splitting $\mc{S}(\mc N_\gamma) = \mc{S}^+(\mc N_\gamma) \oplus \mc{S}^-(\mc N_\gamma)$ of the spinor bundle, where $\mc{S}^\pm(\mc N_\gamma) = \Spin(\mc N_\gamma) \x_{\Spin(2 k_\gamma)} S^\pm$.

The Levi-Civita connection on $V$ induces a connection $\nabla^{\mc N_\gamma}$ on $\mc N_\gamma$, hence a connection $\nabla^{\mc{S}(\mc N_\gamma)}$ on the spinor bundle.  For a  choice of lift $\tilde \gamma$ and compatible orientations on $V^\gamma$ and $\mc N_\gamma$, it follows from \cite[Sec.\ 6.4]{BGV} that 
the equivariant Chern character for $(\mc{S}(\mc N_\gamma), \nabla^{\mc{S}(\mc N_\gamma)})$ is given by
\begin{align}
\label{E:ChChar}
\ch_{\tilde \gamma}(\mc{S}(\mc N_\gamma), \nabla^{\mc{S}(\mc N_\gamma)}) &= 
\ch_{\tilde \gamma}(\mc{S}^+(\mc N_\gamma) - \mc{S}^-(\mc N_\gamma), \nabla^{\mc{S}(\mc N_\gamma)}) \\
\nonumber
&= \pm i^{k_\gamma} \det \left(\id_{\mc N_\gamma} - \gamma \exp\left(-\frac{(\nabla^{\mc N_\gamma})^2}{2\pi i}  \right)\right)^{\!\frac{1}{2}}.
\end{align}
Notice that, since $1$ is not an eigenvalue of the action of $\gamma$ on $\mc N_\gamma$, the identity \eqref{E:ChChar} yields that the form $\ch_{\tilde \gamma}(\mc{S}^+(\mc N_\gamma) - \mc{S}^-(\mc N_\gamma), \nabla^{\mc{S}(\mc N_\gamma)})$ on $V^\gamma$ is invertible.  The \emph{equivariant $\hat{A}$-form} on $V^\gamma$ is now defined as
\beq
\label{E:eqAhat}
\hat A_{\tilde \gamma}(TV, \nabla^{TV}) = (-1)^{k_\gamma} 
\frac{\hat A(TV^\gamma, \nabla^{T V^\gamma})}{\ch_{\tilde \gamma}(\mc{S}^+(\mc N_\gamma) - \mc{S}^-(\mc N_\gamma), \nabla^{\mc{S}(\mc N_\gamma)})}\, ,
\eeq
where the $\hat A$-form for $(TV^\gamma, \nabla^{TV^\gamma})$ is given by
$$
\hat A(TV^\gamma, \nabla^{TV^\gamma}) = 
\det \left(\frac{\frac{1}{4\pi i} \Omega^{TV^\gamma}}{\sinh\left( \frac{1}{4\pi i} \Omega^{TV^\gamma} \right)}\right)^{\!\!\frac{1}{2}} .
$$

\begin{rem} The equivariant form $\hat A_{\tilde \gamma}(TV, \nabla^{TV})$ has the following properties:
\begin{enumerate}
\item $\hat A_{\tilde \gamma}(TV, \nabla^{TV})$ depends only on the conjugacy class of $\tilde \gamma$.
\item Choosing the other lift $-\tilde \gamma$ instead of $\tilde \gamma$ leads to a change of sign in $\hat A_{\tilde \gamma}(TV, \nabla^{TV})$.
\item If the orientation of $TV|_{V^\gamma}$ is fixed, then changing the orientation of $V^\gamma$ leads to a change in the orientation of $\mc N_\gamma$ and, hence, the subbundles $\mc{S}^+(\mc N_\gamma)$ and $\mc{S}^-(\mc N_\gamma)$ being swapped.  This in turn yields a sign change in $\hat A_{\tilde \gamma}(TV, \nabla^{TV})$.  On the other hand, the integral of $\hat A_{\tilde \gamma}(TV, \nabla^{TV})$ over the corresponding stratum of $\Lambda B$ depends only on the orientation of $V$, not on that of $V^\gamma$.
\end{enumerate}
\end{rem}

The oriented orbifold $B$ admits an open cover by orbifold charts compatible with its orientation.  The induced open cover of the inertia orbifold $\Lambda B$ by orbifold charts is compatible with the induced orientation on $\Lambda B$.  Hence, the local equivariant forms $\hat A_{\tilde \gamma}(TV, \nabla^{TV})$ can be used to construct a well-defined form $\hat A_{\Lambda B} (TB, \nabla^{TB})$ on $\Lambda B$ such that
\beq
\label{E:orbAhat}
(\psi_{[\gamma]} \circ \bar q)^*\hat A_{\Lambda B} (TB, \nabla^{TB}) = \frac{1}{m(\gamma)} \hat A_{\tilde \gamma}(TV, \nabla^{TV}),
\eeq
where $\bar q : V^\gamma \to Z_\Gamma(\gamma) \backslash V^\gamma$ denotes the quotient map and $m(\gamma)$ is the multiplicity of the point $(\psi(0), [\gamma]) \in \Lambda B$.

In a similar way, one can define a generalisation $\hat L_{\Lambda B}(TB, \nabla^{TB})$ of the $\hat L$-class (i.e.\ the rescaled $L$-class) on the inertia orbifold $\Lambda B$, where the $\hat L$-form for $(TV^\gamma, \nabla^{TV^\gamma})$ is given as usual by
\begin{align}
\label{E:orbLhat}
\hat L(T V^\gamma, \nabla^{TV^\gamma}) &= \hat A(TV^\gamma, \nabla^{TV^\gamma}) \ch(\mc S(V^\gamma), \nabla^{\mc S(V^\gamma)}) \\
\nonumber 
&= 2^{(\dim V^\gamma - \deg_\gamma)/2} L(T V^\gamma, \nabla^{TV^\gamma}),
\end{align}
where $\deg_\gamma : \Omega^*(V^\gamma) \to \N \cup \{0\}$ denotes the map taking a form $\xi \in \Omega^*(V^\gamma)$ to its degree $\deg(\xi) \in \N \cup \{0\}$, often called the number operator.

In particular,  on the component $B \In \Lambda B$ the forms $\hat A_{\Lambda B} (TB, \nabla^{TB})$ and $\hat L_{\Lambda B}(TB, \nabla^{TB})$ on $\Lambda B$ coincide with the usual version of the forms $\hat A (TB, \nabla^{TB})$ and $\hat L(TB, \nabla^{TB})$ on $B$, defined via charts for $B$.


\subsection{Adiabatic limits} \hspace*{1mm}\\
\label{SS:AdiLims}

The computation of the Eells-Kuiper invariant $\mu(M)$ of a $7$-manifold $M$ via \eqref{E:EKeta} sometimes becomes more tractable if $M$ is the total space of a fibre bundle.  As established in \cite{Gojems}, the same is true in the more general setting of Seifert fibrations.

Given the intended application of these methods to the manifolds $\Mab$, assume from now on that the following condition is satisfied:
\begin{Des}
\item[A]
The $7$-dimensional Riemannian manifold $(M, g_M)$ is $2$-connnected, smooth, closed and oriented with $H^4(M;\Z)$ finite, and there is a Riemannian submersion $\pi : (M, g_M) \to (B, g_B)$ onto a $4$-dimensional Riemannian orbifold $(B,g_B)$ with (generic) fibre $F=\sph^3$. 
\end{Des}

By blowing up the base $(B, g_B)$ by a factor $\ve^{-2}$, $\ve > 0$, one obtains a family of metrics $g_{M,\ve}$ on $M$ with the same horizontal sub-bundle $\H$ and given by
\beq
\label{E:geps}
g_{M,\ve}|_\V = g_M|_\V  \ \ \textnormal{and} \ \  g_{M,\ve}|_\H = \ve^{-2} g_M|_\H  .
\eeq
Of course, up to a global rescaling, this is equivalent to the often-used trick of shrinking the fibres by a factor $\ve^2$.  The limit of any geometric object on $(M, g_{M,\ve})$ as $\ve \to 0$ is called the \emph{adiabatic limit}.  In particular,  it is natural to investigate the adiabatic limit of geometric invariants; for example, the $\eta$-invariants of a family of Dirac operators $\mathsf D_{M,\ve}$ compatible with the metrics $g_{M,\ve}$.

In order to understand the adiabatic limit of the formula \eqref{E:EKeta} for the Eells-Kuiper invariant, it is necessary to establish quite a bit of notation.  A more complete and more general treatment can be found in \cite{Gojems}.

Let $e_1, \dots, e_4$ and $f_1, \dots, f_3$ denote local orthonormal frames for $TB$ and the vertical sub-bundle $\V$ respectively.  Let $\tilde v \in \H$ denote the horizonal lift of a vector field $v \in TB$.  A local orthonormal frame for the metric $g_{M,\ve}$ defined in \eqref{E:geps} is, therefore, given by
\beq
e_{1,\ve} = f_1, \ \dots, \ e_{3,\ve} = f_3, \ \ e_{4,\ve} = \ve \tilde e_1, \ \dots, \ e_{7,\ve} = \ve \tilde e_4.
\eeq

According to Definition 1.6 of \cite{Gojems}, an \emph{adiabatic family of Dirac bundles} for the Riemannian submersion $\pi : (M, g_M) \to (B, g_B)$ in ($\mc A$) consists of a Hermitian vector bundle $(E, g_E)$ over $(M, g_M)$, a Clifford multiplication $c: TM \to \End(E)$, and a family $(\nabla^{E, \ve})_{\ve \geq 0}$ of connections such that:
\begin{enumerate}
\item For all $\ve > 0$, the quadruple $(E, \nabla^{E, \ve}, g_E, c_\ve)$ is a Dirac bundle on $(M, g_{M, \ve})$, where the Clifford multiplication $c_\ve$ is given by $c_\ve(e_{i, \ve}) = c(e_{i, 1})$.
\item 
\label{Cond2}
The connection $\nabla^{E, \ve}$ is analytic in $\ve$ around $\ve = 0$.
\item The kernels of the fibrewise Dirac operators 
$$
\mathsf D_{\sph^3} = \sum_{i = 1}^3 c(f_i) \nabla^{E, 0}_{f_i}
$$
acting on $E|_{\pi^{-1}(b)}$, $b \in B$, form a vector orbi-bundle $\mc{K}_{\mathsf D} \to B$.
\end{enumerate}
The associated family $(\mathsf D_{M, \ve})_{\ve > 0}$, with
$$
\mathsf D_{M, \ve} = \sum_{i = 1}^7 c_\ve (e_{i, \ve}) \nabla^{E, \ve}_{e_{i, \ve}}
$$
is called an \emph{adiabatic family of Dirac operators} for $\pi$.

By Condition \eqref{Cond2} above, there exists an $\ve_0 > 0$ such that the dimension of the kernel of $\mathsf D_{M, \ve}$ is constant for all $\ve \in (0, \ve_0)$.  Furthermore, by Theorem 1.5 of \cite{Dai} (cf. \cite[Sec.\ 2.g]{Gojems}), there are finitely many \emph{very small eigenvalues} of $\mathsf D_{M, \ve}$, that is, eigenvalues $\lambda_\nu(\ve)$, counted with multiplicity, such that $\lambda_\nu(\ve) = O(\ve^2)$ and $\lambda_\nu(\ve) \neq 0$, for all $\ve \in (0, \ve_0)$.

Associated to the fibrewise Dirac operators $\mathsf D_{\sph^3}$ there is an \emph{orbifold $\eta$-form}, denoted $\eta_{\Lambda B}(\mathsf D_{\sph^3})$ by an abuse of notation.  Although $\eta_{\Lambda B}(\mathsf D_{\sph^3})$ can be defined in a similar way to the forms $\hat A_{\Lambda B}$ and $\hat L_{\Lambda B}$ and, consequently, depends only on the conjugacy class and sign of the (lifted) isotropy elements, the usual analytic definition will not be used in the computations to follow, hence will be omitted.  Instead, assume that the fibres of the Riemannian submersion $\pi : (M, g_M) \to (B, g_B)$ in ($\mc A$) are totally geodesic and recall that the isotropy groups act freely on the $\sph^3$ fibres of the Seifert fibration $M \to B$.  Moreover, let $W \to B$ denote the rank-$4$ vector orbi-bundle associated to $M \to B$.  

The horizontal sub-bundle $\H$ on $M$ determines a unique fibre-bundle connection $1$-form $\omega^\pi \in \Hom(TM, \V)$, which acts as the identity on the vertical sub-bundle $\V$ and vanishes on $\H$.  Let $\nabla^W$ be the connection induced on $W$ by $\omega^\pi$, and let $R^W$ be its curvature. 

As before, let $\psi_{[\gamma]}$ be an orbifold chart around $(b, [\gamma]) \in \Lambda B$ and $\bar q : V^\gamma \to Z_\Gamma(\gamma) \backslash V^\gamma$ be the quotient map.  Finally, let $R^W_\gamma$ denote the restriction of the curvature $R^W$ of $\omega^\pi$ to the bundle given by the restriction of $TM$ to $\pi^{-1}((\psi_{[\gamma]} \circ \bar q)(V^\gamma)) \In M$.

Then, by exploiting \cite[Thm.\ 1.14, Thm.\ 3.9]{Go} and \cite[Thm.\ 1.11]{Gojems}, for the case of the spin-Dirac operator $\D$, define $\eta_{\Lambda B}(\D_{\sph^3})$ to be the form on $\Lambda B$ such that
\beq
\label{E:orbetaD}
(\psi_{[\gamma]} \circ \bar q)^* \eta_{\Lambda B}(\D_{\sph^3}) = 
\begin{cases}
-\frac{1}{2^7 \cdot 3 \cdot 5} (\wt \psi)^* e(W, \nabla^W), & \gamma = \id,\\[1mm]
\frac{1}{2} \eta_{\gamma \exp(-R^W_\gamma/2\pi i)}(\D_{\sph^3}), & \gamma \neq \id,
\end{cases}
\eeq
and, for the case of the odd signature operator $\B$, define $\eta_{\Lambda B}(\B_{\sph^3})$ to be the form on $\Lambda B$ such that
\beq
\label{E:orbetaB}
(\psi_{[\gamma]} \circ \bar q)^* \eta_{\Lambda B}(\B_{\sph^3}) = 
\begin{cases}
-\frac{1}{2^2 \cdot 3 \cdot 5} (\wt \psi)^* e(W, \nabla^W), & \gamma = \id,\\[1mm]
\frac{1}{2} \eta_{\gamma \exp(-R^W_\gamma/2\pi i)}(\B_{\sph^3}), & \gamma \neq \id.
\end{cases}
\eeq
In particular, the form $\eta_{\gamma \exp(-R^W_\gamma/2\pi i)}(\mathsf D_{\sph^3})$ is the classical equivariant $\eta$-form as defined by Donnelly \cite{Do2} and, in the current special situation of totally geodesic $\sph^3$ fibres, formulae to compute this form for the operators $\D_{\sph^3}$ and $\B_{\sph^3}$ can be found in \cite{HR} and \cite[Proof of Prop.\ 2.12]{APS} respectively.

In the expression given in Theorem 0.1 of \cite{Gojems} for the limit of an adiabatic family of Dirac operators, there is a term involving the $\eta$-invariant of a self-adjoint operator $\mathsf D_B^\textnormal{eff}$, called the effective horizontal operator.  By definition, $\mathsf D_B^\textnormal{eff}$ is trivial whenever the fibrewise Dirac operators $\mathsf D_{\sph^3}$ are invertible.   Moreover, if the fibrewise Dirac operators $\mathsf D_{\sph^3}$ are invertible, then $\mathsf D_{M, \ve}$ is invertible for sufficiently small $\ve > 0$ and, hence, there are no very small eigenvalues.  In the case of the spin-Dirac operator $\D$, the Weizenb\"ock formula for its fibrewise Dirac operator $\D_{\sph^3}$ ensures invertibility whenever the fibres have positive scalar curvature. 

On the other hand, since the orbifold $B$ is $4$-dimensional, Corollary 1.10 of \cite{Gojems} ensures that for the odd signature operator $\B$ one has $\eta(\B_B^\textnormal{eff}) = 0$.  Additionally, by the work  of Dai \cite{Dai} (cf. \cite{MM}), the very small eigenvalues of $\B_{M, \ve}$ are related to the higher differentials of a natural differentiable Leray-Serre spectral sequence $(E_r, d_r)$ for $M \to B$.  Indeed, given that $M$ is a $2$-connected $7$-manifold, it follows that $\sum_\nu \sign (\lambda_\nu(\ve)) = \tau_\ve$, where $\tau_\ve \in \Z$ is the signature of the quadratic form on $E_4^{0,3}$ given by
\beq
\label{E:QuadForm}
\< \alpha, \beta \> = (\alpha \cdot d_4 \beta)[M].
\eeq

In light of these remarks, it is now possible to write down the adiabatic limits of the families $\D_{M,\ve}$ and $\B_{M,\ve}$.  Let $\nabla^{TB}$ be the Levi-Civita connection for the orbifold $(B, g_B)$.

\begin{thm}[{\cite[Thm 0.1, Cor. 1.10]{Gojems}}]
\label{T:AdiaLimD}
If $(M, g_M)$ is a Riemannian $7$-manifold satisfying Condition ($\mc A$), such that the fibres of the Riemannian submersion $\pi : (M, g_M) \to (B, g_B)$ are totally geodesic and have positive scalar curvature, then
\begin{align}
\lim_{\ve \to 0} \eta(\D_{M,\ve}) &= \int_{\Lambda B} \hat A_{\Lambda B}(TB, \nabla^{TB})\, 2 \, \eta_{\Lambda B}(\D_{\sph^3}). \\
\lim_{\ve \to 0} \eta(\B_{M,\ve}) &= \int_{\Lambda B} \hat L_{\Lambda B}(TB, \nabla^{TB})\, 2 \, \eta_{\Lambda B}(\B_{\sph^3}) + \lim_{\ve \to 0} \tau_\ve.
\end{align}
\end{thm}

In Section 2.a of \cite{Gojems}, it has been shown that the adiabatic limit of the family $\nabla^{TM, \ve}$ of Levi-Civita connections of the metrics $g_{M, \ve}$ is given by 
$$
\lim_{\ve \to 0} \nabla^{TM, \ve} = \nabla^\V \oplus \pi^* \nabla^{TB},
$$
where $\nabla^\V$ denotes the connection on the vertical bundle $\V$ induced by $\nabla^{TM}$.  Using this, it can be shown that the adiabatic limit of the Pontrjagin forms $p_1(TM, \nabla^{TM, \ve})$ is given by
\beq
\label{E:AdLimP1}
\lim_{\ve \to 0} p_1(TM, \nabla^{TM, \ve}) = p_1(\V, \nabla^\V) + \pi^* p_1(TB, \nabla^{TB}),
\eeq
where the Pontrjagin forms on the right-hand side of \eqref{E:AdLimP1} are those obtained from the respective curvature $2$-forms of the bundles.  As $H^4_{dR}(M) = 0$, there are $3$-forms $\hat p_1(\V, \nabla^\V)$ and $\hat p_1(\pi^* TB, \nabla^{TB})$ on $M$ such that
\begin{align*}
d \hat p_1(\V, \nabla^\V) &= p_1(\V, \nabla^\V), \\
d \hat p_1( \pi^* TB, \nabla^{TB}) &= p_1(\pi^* TB, \nabla^{TB}) = \pi^* p_1(TB, \nabla^{TB}).
\end{align*} 
From the variation formula for Chern-Weil classes, it then follows that
\begin{align}
\label{E:AdLimP1Int}
\lim_{\ve \to 0} \int_M p_1(TM, \nabla^{TM, \ve}) \wedge \, & \hat p_1(TM, \nabla^{TM, \ve})  \\
\nonumber 
&= \int_M \left(p_1(\V, \nabla^\V) + \pi^* p_1(TB, \nabla^{TB})\right) \\
\nonumber
&\hspace*{20mm} \wedge \left(\hat p_1(\V, \nabla^\V) + \hat p_1(\pi^* TB, \nabla^{TB}) \right).
\end{align}

Finally, if the fibres of $\pi : (M, g_M) \to (B, g_B)$ are totally geodesic and have positive scalar curvature, recall that $(M, g_{M, \ve})$ has positive scalar curvature for $\ve$ sufficiently small.  Hence, the Weizenb\"ock formula applied to $\D_{M, \ve}$ ensures that $h(\D_{M, \ve}) = \dim \Ker(\D_{M, \ve}) = 0$ for $\ve$ sufficiently small.  This fact, together with Theorem \ref{T:AdiaLimD} and \eqref{E:AdLimP1Int}, yields another expression for the Eells-Kuiper invariant.

\begin{cor}
\label{C:AdLimEK}
If $(M, g_M)$ is a Riemannian $7$-manifold satisfying Condition ($\mc A$), such that the fibres of the Riemannian submersion $\pi : (M, g_M) \to (B, g_B)$ are totally geodesic and have positive scalar curvature, then
\begin{align*}
\mu(M) = \ & 
\frac{1}{2} \int_{\Lambda B} \hat A_{\Lambda B}(TB, \nabla^{TB})\, 2 \, \eta_{\Lambda B}(\D_{\sph^3}) \\
&\ + \frac{1}{2^5 \! \cdot \! 7} \int_{\Lambda B} \hat L_{\Lambda B}(TB, \nabla^{TB})\, 2 \, \eta_{\Lambda B}(\B_{\sph^3}) \\
&\ + \frac{1}{2^5 \! \cdot \! 7} \lim_{\ve \to 0} \tau_\ve \\
&\ - \frac{1}{2^7 \! \cdot \! 7}  \int_M \left(p_1(\V, \nabla^\V) + \pi^* p_1(TB, \nabla^{TB})\right) \\
&\hspace*{25mm} \wedge \left(\hat p_1(\V, \nabla^\V) + \hat p_1(\pi^* TB, \nabla^{TB}) \right) \in \Q/\Z.
\end{align*}
\end{cor}

The formula in Corollary \ref{C:AdLimEK} will be used in Section \ref{S:EK} to compute the Eells-Kuiper invariant given in Theorem \ref{T:thmC}.


\section{Construction, curvature and cohomology}
\label{S:CCC}


\subsection{The Grove-Ziller construction} \hspace*{1mm}\\
\label{SS:GZ}

In order to construct the manifolds $\Mab$, recall first the method employed by Grove and Ziller \cite{GZ} in their construction of metrics with non-negative sectional curvature on all $\sph^3$-bundles over $\sph^4$.  There is an effective action of $\SO(3)$ on $\sph^4$ of cohomogeneity one, such that the double cover $\sph^3$ of $\SO(3)$ acts ($\Z_2$-)ineffectively on $\sph^4$ with cohomogeneity one and group diagram:
$$
\xymatrix{
& \sph^3 & \\
\Pin(2) \ar@{-}[ur] & & \Pjn(2) \ar@{-}[ul] \\
&  Q \ar@{-}[ur] \ar@{-}[ul] & 
}
$$
where $\sph^3$ is taken to be the group $\Sp(1)$ of unit quaternions and
\begin{align*}
Q &= \{\pm 1, \pm i, \pm j, \pm k\}, \\
\Pin(2) &= \{e^{i \theta} \mid \theta \in \R\} \cup \{e^{i \theta} j \mid \theta \in \R\},  \\
\Pjn(2) &= \{e^{j \theta} \mid \theta \in \R\} \cup \{i \, e^{j \theta} \mid \theta \in \R\}.
\end{align*}
The notation $\Pjn(2)$ is intended to be suggestive since, clearly, the groups $\Pin(2)$ and $\Pjn(2)$ are isomorphic, the only difference being that the roles of $i$ and $j$ are switched.  In particular, the singular orbits $\sph^3/\Pin(2)$ and $\sph^3/\Pjn(2)$ are both diffeomorphic to $\RP^2 = \SO(3)/\O(2)$ and are of codimension $2$ in $\sph^4$.

Now, for $a_2, a_3, b_2, b_3 \in \Z$ with $a_i, b_i \equiv 1$ mod $4$, $i = 2,3$, consider the homomorphisms
$$
\vphi_- : \Pin(2) \to \sph^3 \x \sph^3 \ \textrm{ and } \ \vphi_+ : \Pjn(2) \to \sph^3 \x \sph^3
$$
with images
\begin{align*}
\im(\vphi_-) &= \{(e^{i a_2 \theta}, e^{i a_3 \theta}) \mid \theta \in \R\} \cup \{(e^{i a_2 \theta} j, e^{i a_3 \theta} j) \mid \theta \in \R\},  \\
\im(\vphi_+) &= \{(e^{j b_2 \theta}, e^{j b_3 \theta}) \mid \theta \in \R\} \cup \{(i\, e^{j b_2 \theta}, i\, e^{j b_3 \theta}) \mid \theta \in \R\}
\end{align*}
in $\sph^3 \x \sph^3$ respectively.  Let $\ul a = (1, a_2, a_3)$ and $\ul b = (1, b_2, b_3)$.  Then, as described in Subsection \ref{SS:Cohom1}, the homomorphisms $\vphi_\pm$ give rise to a manifold $\Pab$ admitting a ($\Z_2$-ineffective) cohomogeneity-one action by $\sph^3 \x \sph^3 \x \sph^3$ with group diagram
\beq
\label{E:Pab}
\xymatrix{
& \sph^3 \x \sph^3 \x \sph^3 & \\
\Pina \ar@{-}[ur] & & \Pjnb \ar@{-}[ul] \\
&  \Delta Q \ar@{-}[ur] \ar@{-}[ul] & 
}
\eeq
where the principal isotropy group $\Delta Q$ denotes the diagonal embedding of $Q$ into $\sph^3 \x \sph^3 \x \sph^3$, and the singular isotropy groups are given by
\begin{align*}
\Pina &= \{(e^{i \theta}, e^{i a_2 \theta}, e^{i a_3 \theta}) \mid \theta \in \R\} \cup \{(e^{i \theta} j, e^{i a_2 \theta} j, e^{i a_3 \theta} j) \mid \theta \in \R\},  \\
\Pjnb &= \{(e^{j \theta}, e^{j b_2 \theta}, e^{j b_3 \theta}) \mid \theta \in \R\} \cup \{(i \, e^{j \theta}, i\, e^{j b_2 \theta}, i\, e^{j b_3 \theta}) \mid \theta \in \R\}.
\end{align*}
Note that the restriction $a_i, b_i \equiv 1$ mod $4$ is to ensure only that $\Delta Q$ is a subgroup of both $\Pina$ and $\Pjnb$.  Furthermore, since the singular orbits of the cohomogeneity-one action on $\Pab$ are of codimension $2$, it follows from Theorem \ref{T:GZ} that each $\Pab$ admits an $(\sph^3 \x \sph^3 \x \sph^3)$-invariant metric $g_{GZ}$ of non-negative sectional curvature.
 
By construction, the action of the subgroup $\{1\} \x \sph^3 \x \sph^3$ on $\Pab$ is free, meaning that $\Pab$ is the total space of a principal $(\sph^3 \x \sph^3)$-bundle over $\sph^4$.  Grove and Ziller \cite{GZ} showed that all principal $(\sph^3 \x \sph^3)$-bundles over $\sph^4$ are attained in this way.

Since every $\sph^3$-bundle over $\sph^4$ arises as an associated bundle to a principal $(\sph^3 \x \sph^3)$-bundle, the above construction also yields a metric with non-negative curvature on all $\sph^3$-bundles over $\sph^4$.  Indeed, by the associated bundle construction, an $\sph^3$-bundle over $\sph^4$ can be written as 
\beq
\label{E:assoc}
\Pab \x_{\sph^3 \x \sph^3} \sph^3 = \Pab \x_{\sph^3 \x \sph^3} ((\sph^3 \x \sph^3)/\Delta \sph^3) \, .
\eeq
If $\Pab$ is equipped with the Grove-Ziller metric $g_{GZ}$ as above and $\sph^3$ with the round metric, then the product metric on $\Pab \x \sph^3$ is non-negatively curved.  By the Gray-O'Neill formula for Riemannian submersions, the quotient map 
$$
\Pab \x \sph^3 \to \Pab \x_{\sph^3 \x \sph^3} \sph^3
$$ 
induces a metric of non-negative curvature on $\Pab \x_{\sph^3 \x \sph^3} \sph^3$, as claimed.  In particular, Grove and Ziller \cite{GZ} conclude that all Milnor spheres admit a metric with non-negative sectional curvature.

The point of departure from the Grove-Ziller construction just discussed comes from the following 

\begin{obs}
The subgroup $\{1\} \x \Delta \sph^3 \In \{1\} \x \sph^3 \x \sph^3$ acts freely on (the left of) $\Pab$ such that the quotient $ (\{1\} \x \Delta \sph^3) \backslash \Pab$ is diffeomorphic to the corresponding $\sph^3$-bundle over $\sph^4$, namely, $\Pab \x_{\sph^3 \x \sph^3} \sph^3$.
\end{obs}

This observation, also noted in Section 5 of \cite{GZ}, follows from \eqref{E:assoc} and the simple, often-used fact that, if $G$ is a Lie group acting on itself by left multiplication and on an arbitrary manifold $P$ via a left action $\vphi : G \x P \to P \, ; \, (g,p) \mapsto g \cdot p$, then there is a diffeomorphism $P \x_G G \to P$ induced by the smooth map $P \x G \to P\, ;$ $(p,g) \mapsto \vphi(g^{-1}, p) = g^{-1}\! \cdot p$, where $G$ acts diagonally on $P \x G$.  Now, if $H \In G$ is any closed subgroup, the action of $H$ on $G$ via right multiplication commutes with the diagonal action of $G$ on $P \x G$, hence induces a diffeomorphism 
$$
P \x_G (G/H) \to (P \x_G G)/H \to H \backslash P \, .
$$


\subsection{The manifolds $\Mab$} \hspace*{1mm}\\
\label{SS:Mab}

In light of the final observation above and the suggestive notation for $\ul a$ and $\ul b$ used in the description of $\Pab$, it is natural to investigate the action of the group $\{1\} \x \Delta \sph^3 \In \sph^3 \x \sph^3 \x \sph^3$ in a more general setting, namely, when the first entry of either or both of the triples $\ul a$ and $\ul b$ is different to $1$.  For the sake of notation, from now on let $G = \sph^3 \x \sph^3 \x \sph^3$.

For $\ul a = (a_1, a_2, a_3)$, $\ul b = (b_1, b_2, b_3) \in \Z^3$, with $a_i, b_i \equiv 1$ mod $4$ and $\gcd(a_1, a_2, a_3) = \gcd(b_1, b_2, b_3) = 1$, define $\Pab$ to be the cohomogeneity-one $G$-manifold with group diagram \eqref{E:Pab}, where the singular isotropy groups are now given by
\begin{align*}
\Pina &= \{(e^{i a_1\theta}, e^{i a_2 \theta}, e^{i a_3 \theta}) \mid \theta \in \R\} \cup \{(e^{i a_1 \theta} j, e^{i a_2 \theta} j, e^{i a_3 \theta} j) \mid \theta \in \R\},  \\
\Pjnb &= \{(e^{j b_1 \theta}, e^{j b_2 \theta}, e^{j b_3 \theta}) \mid \theta \in \R\} \cup \{(i \, e^{j b_1 \theta}, i\, e^{j b_2 \theta}, i\, e^{j b_3 \theta}) \mid \theta \in \R\}.
\end{align*}
The $\gcd$ conditions on the triples $\ul a$ and $\ul b$ ensure simply that the homomorphisms $\Pin(2) \to \Pina$ and $\Pjn(2) \to \Pjnb$ into $G$ are injective.  

\begin{lem}
\label{L:free}
The subgroup $\{1\} \x \Delta \sph^3 \In G$ acts freely on $\Pab$ if and only if 
\beq
\label{E:free}
\gcd(a_1, a_2 \pm a_3) = 1 \ \textrm{ and } \ \gcd(b_1, b_2 \pm b_3) = 1.
\eeq
\end{lem}

\begin{proof}
First, suppose that the action is free and that one of the $\gcd$ conditions does not hold, say $\gcd(a_1, a_2 - a_3) = d \neq 1$.  This implies that $e^{2\pi i a_1/d} = 1$ and $e^{2\pi i (a_2 - a_3)/d} = 1$ and, furthermore, since $\gcd(a_1, a_2, a_3) = 1$, that $d$ divides neither $a_2$ nor $a_3$.  

As the action of $\{1\} \x \Delta \sph^3$ on an $G$-orbit in $\Pab$ is via
$$
(1,q,q) \cdot [q_1, q_2, q_3] = [q_1, q\, q_2, q\, q_3],
$$
it follows that the non-trivial element $\ul q = (1, e^{2\pi i a_2/d}, e^{2\pi i a_2/d}) \in \{1\} \x \Delta \sph^3$ fixes the point $[1,1,1] \in G/\Pina \In \Pab$.  This contradicts the freeness assumption, hence $d = 1$.  The arguments for the other $\gcd$ conditions in \eqref{E:free} are similar.

On the other hand, suppose now that $\gcd(a_1, a_2 \pm a_3) = 1$ and $\gcd(b_1, b_2 \pm b_3) = 1$.  Since the action of $\{1\} \x \Delta \sph^3$ on a principal orbit $G/\Delta Q$ is clearly free, it is sufficient to establish freeness of the action on the singular orbits.  

Suppose that $(1,q,q) \in \{1\} \x \Delta \sph^3$ lies in the isotropy subgroup of some $[q_1, q_2, q_3] \in G/\Pina$, that is, that
$$
(1,q,q) \cdot [q_1, q_2, q_3] = [q_1, q\, q_2, q\, q_3] = [q_1, q_2, q_3].
$$ 
Therefore, there is some $\ul \alpha = (\alpha_1, \alpha_2, \alpha_3) \in \Pina$ such that
\beq
\label{E:alpha}
(q_1, q\, q_2, q\, q_3) = (q_1 \, \alpha_1, q_2 \, \alpha_2, q_3 \, \alpha_3).
\eeq
The identity $q_1 = q_1 \, \alpha_1$ implies that $\alpha_1 = 1$, hence that $\ul \alpha$ lies in the identity component of $\Pina$.  In other words, there is some $\theta \in \R$ such that $\ul \alpha = (e^{i a_1\theta}, e^{i a_2 \theta}, e^{i a_3 \theta}) = (1, e^{i a_2 \theta}, e^{i a_3 \theta})$.  To conclude that the action on the orbit $G/\Pina$ is free, it suffices now to show that $e^{i \theta} = 1$, because it can then be deduced from \eqref{E:alpha} that $q = 1$.

From \eqref{E:alpha} it is clear that $q_2^{-1} q_3 = \alpha_2^{-1} q_2^{-1} q_3 \, \alpha_3$.  As $q_2^{-1} q_3 \in \sph^3$, there exist $x, y \in \C$ with $|x|^2 + |y|^2 = 1$ such that $q_2^{-1} q_3 = x + y j$.  Hence,
$$
x + y j = \alpha_2^{-1} (x + y \, j) \alpha_3 = e^{i(a_3 - a_2)\theta} x + e^{i(a_2 + a_3)\theta} y j.
$$
Since $x$ and $y$ cannot vanish simultaneously, it follows that either $e^{i(a_2 - a_3)\theta} = 1$ or $e^{i(a_2 + a_3)\theta} = 1$.  Together with $e^{i a_1 \theta} = 1$ and $\gcd(a_1, a_2 \pm a_3) = 1$, one concludes that $e^{i \theta} = 1$, as claimed.

The argument for freeness along the other singular orbit is analogous.
\end{proof}

By Lemma \ref{L:free}, whenever $\gcd(a_1, a_2 \pm a_3) = 1$ and $\gcd(b_1, b_2 \pm b_3) = 1$, there is a smooth, $7$-dimensional manifold $\Mab$ defined via
$$
\Mab = (\{1\} \x \Delta \sph^3) \backslash \Pab \, .
$$

\begin{lem}
\label{L:nonneg}
The manifold $\Mab$ admits an $\SO(3)$-invariant Riemannian metric of non-negative sectional curvature.
\end{lem}

\begin{proof}
By Theorem \ref{T:GZ}, the cohomogeneity-one manifolds $\Pab$ admit a $G$-invariant metric $g_{GZ}$ of non-negative curvature.  Since the free action by $\{1\} \x \Delta \sph^3 \In G$ is isometric, it follows from the Gray-O'Neill formula for Riemannian submersions that $g_{GZ}$ induces a metric $\vg$ of non-negative curvature on the quotient $\Mab$.  Moreover, the action of the subgroup $\sph^3 \x \{(1, 1)\} \In G$ on $\Pab$ is isometric and commutes with the action of $\{1\} \x \Delta \sph^3$, hence descends to an isometric $\sph^3$ action on $(\Mab, \vg)$ with ineffective kernel $\{(\pm 1, 1, 1)\}$.  The effective action on $(\Mab, \vg)$ is, therefore, an $\SO(3)$ action.
\end{proof}

For the topological computations to follow, it is important to remark that, by construction and just as for a cohomogeneity-one manifold, there is a codimension-one singular Riemannian foliation of $\Mab$ by biquotients, such that the leaf space is $[-1,1]$ and $\Mab$ decomposes as a union of two-dimensional disk bundles over the two singular leaves, glued along their common boundary, a regular leaf.  This follows easily from the Slice Theorem applied to $\Pab$.  Indeed, the action of $\{1\} \x \Delta \sph^3$ preserves the $G$-orbits of $\Pab$, and the image of an orbit $G/U$, $U \in \{\Delta Q, \Pina, \Pjnb\}$, is a leaf given by
\beq
\label{E:BiqDiff}
(\{1\} \x \Delta \sph^3) \backslash G / U \cong (\sph^3 \x \sph^3) \bq U \, ,
\eeq
where this diffeomorphism is induced by 
$$
(q_1 \, u_1, q_2 \, u_2, q_3 \, u_3) \mapsto (q_1 \, u_1, u_2^{-1} q_2^{-1} q_3 \, u_3),
$$ 
for $(q_1, q_2, q_3) \in G$ and $(u_1, u_2, u_3) \in U$. Viewing $\Mab$ in this way, the $\gcd$ conditions \eqref{E:free} required in the definition are simply the conditions ensuring that each of the biquotient actions on $\sph^3 \x \sph^3$ is free.

In contrast to the Grove-Ziller situation, where $a_1 = b_1 = 1$ and the manifold $\Mab$ is naturally the total space of a fibre bundle, the quotient $(\{1\} \x \sph^3 \x \sph^3) \backslash \Pab$ is not a manifold, in general.

\begin{lem}
\label{L:orbi}
The action by the subgroup $\{1\} \x \sph^3 \x \sph^3 \In G$ has trivial isotropy at points on principal orbits, while the isotropy group at a point $[q_1, q_2, q_3]$ on a singular orbit, that is, on either $G/\Pina$ or $G/\Pjnb$, is given by
\begin{align*}
\Z_{|a_1|} &\cong \{(1, q_2 \, \xi^{a_2} \bar q_2, q_3 \, \xi^{a_3} \bar q_3) \mid \xi \in \sph^1_i \In \Pin(2), \xi^{a_1} = 1 \},\\
\textrm{ or }\ \Z_{|b_1|} &\cong \{(1, q_2 \, \xi^{b_2} \bar q_2, q_3 \, \xi^{b_3} \bar q_3) \mid \xi \in \sph^1_j \In \Pjn(2), \xi^{b_1} = 1 \}
\end{align*}
respectively, where $\sph^1_i = \{e^{i \theta} \mid \theta \in \R\}$ and $\sph^1_j = \{e^{j \theta} \mid \theta \in \R\}$.  

Hence, for $(a_1, b_1) \neq (1,1)$ the quotient 
$$
\Bab = (\{1\} \x \sph^3 \x \sph^3) \backslash \Pab
$$ 
is a $4$-dimensional, smooth orbifold.  In this case, the singular set consists of at most two copies, $\RP^2_\pm$, of $\RP^2$ for which the normal bundles in $\Bab$ have cone angles $2\pi/|a_1|$ and $2\pi/|b_1|$ respectively.  Moreover, the projection $\pi : \Mab \to \Bab$ is a Seifert fibration with fibre $\sph^3 \cong (\sph^3 \x \sph^3)/\Delta \sph^3$.

Orbifold charts can be chosen on $\Bab$ such that the action of the corresponding isotropy group at a point of $\RP^2_\pm$ is trivial tangent to $\RP^2_\pm$, equivalent to the action generated by multiplication by $e^{8\pi i/a_1}$ normal to $\RP^2_-$, and generated by multiplication by $e^{8\pi i/b_1}$ normal to $\RP^2_+$.

The actions of the isotropy groups on the fibre $\sph^3$ are equivalent to
$$
(\xi, q) \mapsto \xi^{a_2} q \, \xi^{-a_3} \ \ \text{ and } \ \ (\xi, q) \mapsto \xi^{b_2} q \, \xi^{-b_3},
$$
where $q \in \sph^3$ and $\xi \in \Z_{|a_1|}$, $\xi \in \Z_{|b_1|}$ respectively.

\end{lem}

\begin{proof}
As in the proof of Lemma \ref{L:free}, it suffices to restrict attention to the $\{1\} \x \sph^3 \x \sph^3$ action on the $G$-orbits in $\Pab$.  It is a simple exercise to show that the action on the principal orbits is free.

Consider the action of $(1,x_2,x_3) \in \{1\} \x \sph^3 \x \sph^3$ on the singular orbit $G/\Pina$.  If
$$
(1,x_2, x_3) \cdot [q_1, q_2, q_3] =  [q_1, x_2 \, q_2, x_3 \, q_3] =  [q_1, q_2, q_3], 
$$
then there is some $\ul \alpha = (\alpha_1, \alpha_2, \alpha_3) \in \Pina$ such that 
$$
(q_1, x_2 \, q_2, x_3 \, q_3) =  (q_1 \, \alpha_1, q_2 \, \alpha_2, q_3 \, \alpha_3).
$$
Therefore, $\alpha_1 = 1$, $x_2 = q_2 \, \alpha_2 \, \bar q_2$ and $x_3 = q_3 \, \alpha_3 \, \bar q_3$.  The condition $\alpha_1 = 1$ ensures that $\ul \alpha$ is in the identity component of $\Pina$, hence $1 = \alpha_1 = e^{i a_1 \theta}$.  In other words, $e^{i \theta} \in \sph^1_i$ is an $a_1^\textrm{th}$ root of unity.  The description of the isotropy group at $[q_1, q_2, q_3]$ now follows from the definition of $\Pina$ and the fact that all entries in $\ul a$ are $1$ mod $4$.

An entirely analogous argument yields the isotropy group at points on $G/\Pjnb$.

When $(a_1, b_1) \neq (1,1)$, the non-trivial isotropy groups of the $\{1\} \x \sph^3 \x \sph^3$ action are finite, hence the action is almost free.  The quotient of a smooth manifold by an almost-free action is always a smooth orbifold.

The singular set of the orbifold $\Bab$ consists of the image of those points of $\Pab$ at which there is non-trivial isotropy, namely, the image of the singular orbits $G/\Pina$ and $G/\Pjnb$.

Since the $\{1\} \x \sph^3 \x \sph^3$ action on $G$ commutes with the action of $\Pina$, it follows that
\begin{align*}
(\{1\} \x \sph^3 \x \sph^3) \backslash (G/\Pina) &\cong ((\{1\} \x \sph^3 \x \sph^3) \backslash G) / \Pina \\
 &\cong \sph^3 / \Pin(2)_{a_1} \, ,
\end{align*}
where $\Pin(2)_{a_1} = \{e^{i a_1 \theta} \mid \theta \in \R \} \cup \{e^{i a_1 \theta} j \mid \theta \in \R \}$.  The action of $\Pin(2)_{a_1}$ on $\sph^3$ is free up to the ineffective kernel $\Gamma \In \{\xi \in \sph^1_i \In \Pin(2) \mid \xi^{a_1} = 1\} \cong \Z_{|a_1|}$.  Since $\sph^1_i / \Gamma \cong \sph^1$, the group $\Pin(2)_{a_1} / \Gamma \cong \Pin(2)$ acts freely on $\sph^3$ with quotient $\RP^2_-$.

Notice that the $(\{1\} \x \sph^3 \x \sph^3)$-orbit of a point $[q_1, q_2, q_3] \in G/\Pina$ contains the point $[q_1, 1, 1]$.  For $\ve > 0$ sufficiently small, the intersection of the set $\{[x, 1, 1] \in G/\Pina \mid x \in \sph^3\} \In G/\Pina$ with the $\ve$-ball $B_\ve([q_1, 1, 1]) \In \Pab$ projects diffeomorphically onto a chart of $\RP^2_-$.   As the isotropy group $\Z_{|a_1|}$ at $[q_1, 1, 1]$ fixes all nearby points of the form $[q, 1, 1] \in G/\Pina$, the isotropy representation can be non-trivial only on the normal $2$-disk to $G/\Pina \In \Pab$ at $[q_1, 1, 1]$.  However, the action of $\{1\} \x \sph^3 \x \sph^3$ on $\Pab$ is free away from the singular orbits, hence the $\Z_{|a_1|}$ action on the normal $\ve$-circle at $[q_1, 1, 1]$ must be free.  Therefore, the normal space at any point in $\RP^2_- \In \Bab$ must have cone angle $2\pi/|a_1|$, as it is the quotient of the normal $2$-disk to $G/\Pina$ by the $\Z_{|a_1|}$ isotropy action.  

On the other hand, since the restriction to the unit circle of the $\Pina$ slice action on the normal $2$-disk at $[q_1, 1, 1] \in G/\Pina$ has isotropy $\Delta Q$, it must be equivalent to the action
$$
\Pin(2) \x \DD^2 \to \DD^2\, ;\, (\alpha, z) \mapsto 
\begin{cases}
e^{4i\theta} z, & \alpha = e^{i\theta} \\
e^{4i\theta} \bar z, & \alpha = e^{i\theta}j
\end{cases}
$$
of $\Pin(2)$ on the standard $2$-disk $\DD^2 \In \C$.  Therefore, the free $\Z_{|a_1|}$-isotropy action on the normal disk at $[q_1, 1, 1] \in G/\Pina$ is generated by multiplication by $e^{8 \pi i/a_1}$, as claimed.

The action of $\Z_{|a_1|}$ on the fibre $\sph^3$ follows from the description \eqref{E:BiqDiff} of the leaves of $\Mab$ as biquotients.

Similar arguments deliver the corresponding conclusions for the normal bundle to the image $\RP^2_+ \In \Bab$ of $G/\Pjnb$.

The fact that $\pi : \Mab \to \Bab$ is a Seifert fibration with fibre $\sph^3$ follows immediately from Remark \ref{R:Seifert}.
\end{proof}

Note, in particular, that the orbifold $\Bab$ inherits an (ineffective) action of $\sph^3$ of cohomogeneity one with principal isotropy subgroup $Q$ and singular isotropy groups $\Pin(2)_{a_1}$ and $\Pjn(2)_{b_1}$.

\begin{cor}[{\cite[Prop.\ 4.1]{Gojems}}]
\label{C:InerOrb}
The inertia orbifold  $\Lambda B$ associated to $\Bab$ is described by a disjoint union
$$
\Lambda B = \Bab \sqcup \left(\sph^2_- \x \left\{1, \dots, \frac{|a_1|-1}{2}\right\} \right) \sqcup \left(\sph^2_+ \x \left\{1, \dots, \frac{|b_1|-1}{2}\right\} \right),
$$
where $\sph^2_\pm$ denotes the orientable double cover of $\RP^2_\pm$ respectively.  If $b \in \RP^2_\pm \In \Bab$ and $\gamma_\pm$ denotes the generator of the isotropy group at $b$, then the pre-images of $b$ in the twisted sector $\sph^2_\pm \x \{s\}$ are given by the two points $(b, [\gamma_\pm^{\ell}])  \in \Lambda B$, where $\ell \in \left\{1, \dots, \frac{|c|-1}{2}\right\}$ with $\ell \equiv \pm s \mod c$, for $c = a_1$ or $c= b_1$ respectively.

Moreover, the twisted sectors $\sph^2_\pm \x \{s\}$ have multiplicity $m(\gamma_-^s) = |a_1|$ and $m(\gamma_+^s) = |b_1|$ respectively.
\end{cor}

\begin{proof}  The four-dimensional component of $\Lambda B$ consists, by definition, of all points fixed by the the identity in $\{1\} \x \sph^3 \x \sph^3$.

For the remaining components, it suffices to find a curve in $\sph^2_\pm \x \{s\}$ joining $(b, [\gamma_\pm^s])$ with $(b, [ (\gamma_\pm^s)^{-1}] )$, that is, a loop in $\RP^2_\pm$ along which the generator of the isotropy group at $b$ changes from $\gamma_\pm$ to $\gamma_\pm^{-1}$.  Indeed, this is the only non-trivial change that can occur.  Consequently, there can be only $\frac{|c| - 1}{2}$ additional components for each of $c = a_1$ and $c = b_1$.  Without loss of generality, assume $c = a_1$.

Consider the curve $g : [0, \frac{\pi}{2}] \to \sph^3$ given by $g(t) = e^{j t}$, with endpoints $g(0) = 1$, $g(\frac{\pi}{2}) = j \in \Pin(2)$.  Let $\tilde b :  [0, \frac{\pi}{2}] \to G/\Pina$ be the loop in $G/\Pina$ defined by $\tilde b(t) = [q_1 g(t), g(t), g(t)]$, where $b(\frac{\pi}{2}) = [q_1 j, j, j] = [q_1, 1, 1] = \tilde b(0)$.  Since $g(t) \not\in \Pin(2)$ for $t \in (0, \frac{\pi}{2})$, the projection of $\tilde b$ to $\RP^2_-$ is also a non-trivial loop.  By Lemma \ref{L:orbi}, passing around this loop yields a path $\gamma_-(t) = g(t) \gamma_- \ol{g(t)}$ of generators of the isotropy groups at $\tilde b(t)$, with endpoints $\gamma_-(0) = \gamma_-$ and $\gamma_-(\frac{\pi}{2}) = \gamma_-^{-1}$. 

Finally, the multiplicity statements follow directly from the definition \eqref{E:multiplicity}, together with the facts that the isotropy groups are abelian and, via Lemma \ref{L:orbi}, act trivially on local charts of $\RP^2_\pm$.
\end{proof}


\subsection{The cohomology of \texorpdfstring{$\Mab$}{M}} \hspace*{1mm}\\
\label{SS:cohomology}

Unfortunately, Lemma \ref{L:orbi} implies that the manifold $\Mab$ is, in general, not the total space of an $\sph^3$-bundle over $\sph^4$ in any obvious way, if at all.  In \cite[Prop.\ 3.3]{GZ}, being associated to a principal bundle over $\sph^4$, with total space of cohomogeneity one, was an important part of the authors' cohomology computations.  On the other hand, in \cite[Sec.\ 13]{GWZ} the cohomology rings for a particular family of $7$-dimensional cohomogeneity-one manifolds was computed.  Although these manifolds are foliated by homogeneous spaces instead of biquotients, they strongly resemble the manifolds $\Mab$.  In order to compute the cohomology ring of the manifolds $\Mab$, ideas from both \cite{GWZ} and \cite{GZ} will be used, although it is necessary to work quite a bit harder.

It will be useful in the sequel to consider the following manifold:  Given $\Pab$, let $\hPab$ be the cohomogeneity-one $(G \x \sph^3)$-manifold given (as in Subsection \ref{SS:Cohom1}) by the homomorphisms $\vphi_-:\Pin(2) \to \sph^3 \, ; \, \alpha \mapsto \alpha$, and $\vphi_+:\Pjn(2) \to \sph^3\, ; \, \beta \mapsto \beta$.  In other words, and in analogy with the description of $\Pab$, the singular isotropy groups ($\cong \Pin(2)$) for $\hPab$ are described by the $4$-tuples $(a_1, a_2, a_3, 1)$ and $(b_1, b_2, b_3, 1)$ respectively.  

It is clear that, by construction, the action of the subgroup $\{(1,1,1)\} \x \sph^3 \In G \x \sph^3$ is free, inducing a principal $\sph^3$-bundle 
\begin{align}
\label{E:S3bund}
\sph^3 &\to \hPab \to \Pab \, . \\
\intertext{Notice, however, that the action of $G \x \{1\} \In G \x \sph^3$ on $\hPab$ is also free, and that the quotient is $\sph^4$.  Therefore, $\hPab$ is, in addition, the total space of a principal bundle}
\label{E:Gbund}
G &\to \hPab \to \sph^4.
\end{align}

\begin{lem}
\label{L:2conn}
The manifolds $\Pab$ and $\Mab$ are $2$-connected. 
\end{lem}

\begin{proof}
From the long exact homotopy sequence for the bundle \eqref{E:S3bund}, $\Pab$ is $2$-connected if and only if $\hPab$ is.  However, $\hPab$ is $2$-connected because of \eqref{E:Gbund}, since both $G$ and $\sph^4$ are $2$-connected.

As $\Pab$ is a principal $\sph^3$-bundle over $\Mab$, it follows from the corresponding long exact homotopy sequence that $\Mab$ is $2$-connected. 
\end{proof}

In the argument to compute the cohomology of $\Mab$, it will be important to understand the cohomology of the singular (biquotient) leaves $(\sph^3 \x \sph^3) \bq \Pina$ and $(\sph^3 \x \sph^3) \bq \Pjnb$.

\begin{lem}
\label{L:singbiq}
If $X = (\sph^3 \x \sph^3) \bq K$, $K \in \{\Pina, \Pjnb\}$ is a singular leaf of $\Mab$, then
$X$ has the same cohomology groups as $\sph^3 \x \RP^2$, that is, 
$$
H^j(X ; \Z) = 
\begin{cases}
\Z, & j = 0,3, \\
0, & j = 1,4,\\
\Z_2, & j = 2, 5.
\end{cases}
$$
Moreover, if $\wt X = (\sph^3 \x \sph^3) \bq K^o$, where $K^o \cong \sph^1$ is the identity component of $K$, let $\rho: \wt X \to X$ be the projection given by taking the quotient by the free action of $K/K^o \cong \Z_2$.  Then the two-fold covering $\rho$ induces an isomorphism $\rho^* : H^3(X; \Z) \to H^3( \wt X; \Z)$.
\end{lem}

\begin{proof}
Consider again the cohomogeneity-one $(G \x \sph^3)$-manifold $\hPab$.  By construction, a singular orbit $Y$ of $\hPab$ is a principal $(\sph^3 \x \sph^3)$-bundle over a singular leaf $X = (\sph^3 \x \sph^3) \bq K$ of $\Mab$, where the principal $\sph^3 \x \sph^3$ action is that of the subgroup $\{1\} \x \Delta \sph^3 \x \sph^3 \In G \x \sph^3$.  On the other hand, $Y$ is also a principal $G$-bundle over a copy of $\RP^2 \In \sph^4$.  Therefore, since the classifying space $B_G = (\HP^\infty)^3$ for principal $G$-bundles is $3$-connected, $Y$ is trivial as a principal $G$-bundle, that is, $Y$ is $G$-equivariantly diffeomorphic to $G \x \RP^2$.  

Associated to the principal bundle $\sph^3 \x \sph^3 \to Y \stackrel{\sigma}{\lra} X$, there is a homotopy fibration 
\beq
\label{E:homfib}
Y \stackrel{\sigma}{\lra} X \to B_{\sph^3 \x \sph^3} =  \HP^\infty \x \HP^\infty \, .
\eeq
The identities $H^0(X; \Z) = \Z$, $H^1(X; \Z) = 0$ and $H^2(X; \Z) = \Z_2$ follow immediately from the corresponding Serre spectral sequence $(E_r, d_r)$.  From the differential 
$$
d_4 : E_4^{0,3} = H^3(Y; \Z) = \Z^3 \to E_4^{4,0} = H^4(B_{\sph^3 \x \sph^3}; \Z) = \Z^2
$$
on the $E_4$ page, one obtains
\begin{align*}
H^3(X; \Z) &\cong \Ker (d_4 : E_4^{0,3} \to E_4^{4,0}) \, ,\\
H^4(X; \Z) &\cong H^4(B_{\sph^3 \x \sph^3}; \Z))/ \im (d_4 : E_4^{0,3} \to E_4^{4,0})\, .
\end{align*}
As $d_4 : E_4^{0,3} \to E_4^{4,0}$ cannot be injective, it follows that $H^3(X; \Z) = \Z^{b_3(X)}$, for $b_3(X) \in \{1,2,3\}$. 

On the other hand, consider the two-fold covering 
$$
\rho: \wt X = (\sph^3 \x \sph^3) \bq K^o \to X = (\sph^3 \x \sph^3) \bq K \, .
$$
Since $K^o \cong \sph^1$, the Gysin sequence for the fibration 
$$
K^o \to \sph^3 \x \sph^3 \to \wt X
$$ 
yields $H^3(\wt X; \Z) = \Z$. In fact, although it is not important here, $\wt X$ is always diffeomorphic to $\sph^3 \x \sph^2$ (see \cite{DV, GGK}).  Therefore, from Smith Theory one obtains 
$$
H^3(X ; \Q) \cong H^3(\wt X; \Q)^{\Z_2} \cong \Q^{\Z_2}.
$$
This clearly implies that $b_3(X) \leq 1$, hence, that $H^3(X; \Z) = \Z$.

Moreover, since $\Ker(d_4 : E_4^{0,3} \to E_4^{4,0}) \cong \Z$, it is apparent that there are generators $x_1, x_2, x_3 \in H^3(Y; \Z) = \Z^3$ and $\alpha_1, \alpha_2 \in H^4(B_{\sph^3 \x \sph^3}; \Z) = \Z^2$ such that 
\begin{align}
\nonumber 
d_4(x_1) &= r_1 \, \alpha_1 + s_1 \, \alpha_2 \, , \\
\label{E:d4}
d_4(x_2) &= r_2 \, \alpha_1 + s_2 \, \alpha_2 \, , \\
\nonumber
d_4(x_3) &= 0,
\end{align}
for some $r_1, r_2, s_1, s_2 \in \Z$ with $C = \det \bsm r_1 & r_2 \\ s_1 & s_2 \esm \neq 0$.  In particular, since $H^5(X; \Z) = \Ker(d_4 : E_4^{0,5} \to E_4^{4,2})$, it is easy to deduce from \eqref{E:d4} that $H^5(X;\Z) = \Z_2$.

The fact that $H^4(X; \Z) = 0$ will follow from the surjectivity of the differential $d_4 : E_4^{0,3} \to E_4^{4,0}$, which is equivalent to the identity $C = \pm 1$.  As $X$ is a five-dimensional manifold, it is clear that $H^8(X; \Z) = 0$.  In particular, no terms on the diagonal $E_4^{k,l}$, $k + l = 8$, of the spectral sequence for \eqref{E:homfib} can survive to the $E_\infty$ page.  Therefore, as all other differentials with range $E_4^{8,0}$ are trivial, the differential $d_4 : E_4^{4,3} \to E_4^{8,0}$ is necessarily surjective.  This is the case if and only if the $\gcd$ of the determinants of all $(3 \x 3)$-minors of a matrix representation of $d_4 : E_4^{4,3} \to E_4^{8,0}$ is $1$.  This latter equivalence can be proven by reducing such a matrix to Smith normal form via integral row operations.  With respect to the bases $\{x_i \alpha_j \mid 1 \leq i \leq 3, 1 \leq j \leq 2\}$ and $\{ \alpha_1^2, \alpha_1 \alpha_2, \alpha_2^2\}$ for $E_4^{4,3}$ and $E_4^{8,0}$ respectively, the matrix representation of $d_4 : E_4^{4,3} \to E_4^{8,0}$ is
$$
\bpm
r_1 & 0 & r_2 & 0 & 0 & 0 \\
s_1 & r_1 & s_2 & r_2 & 0 & 0\\
0 & s_1 & 0 & s_2 & 0 &0
\epm .
$$
Then the $\gcd$ of the determinants of all $(3 \x 3)$-minors is divisible by $C$, from which it follows that $C = \pm 1$, as desired.

It remains to show that $\rho^* : H^3(X; \Z) \to H^3( \wt X; \Z)$ is an isomorphism.  To this end, recall that, by definition, there is an injective homomorphism $K \to G \x \sph^3$, such that $Y = (G \x \sph^3)/K$.  Define, therefore, $\wh\rho : \wt Y \to Y$ to be the two-fold covering of $Y$ induced by the free action of $\Z_2 \cong K/K^o$ on $\wt Y  = (G \x \sph^3)/K^o$.  By the same arguments as for $Y$, it follows that $\wt Y$ is $G$-equivariantly diffeomorphic to $G \x \sph^2$.  Moreover, since the actions of $G$ and $K/K^o$ commute, there is a commutative diagram of homotopy fibrations
$$
\xymatrix{
\wt Y \ar[r] \ar[d]_{\wh \rho} & \sph^2 \ar[r] \ar[d] & B_G \ar[d]^{B_\id} \\
Y \ar[r]  & \RP^2 \ar[r] & B_G
}
$$
Let $(\wt{\mc E}_r, \wt \delta_r)$ and $(\mc E_r, \delta_r)$ denote the Serre spectral sequences for the upper and lower homotopy fibrations respectively.  It is clear from these spectral sequences that the differentials 
\begin{align*}
\delta_4 : \mc E_4^{0,3} = H^3(Y; \Z) &\to \mc E_4^{4,0} = H^4(B_G; \Z) , \\
\wt \delta_4 : \wt{\mc E}_4^{0,3} = H^3(\wt Y; \Z) &\to \wt{\mc E}_4^{4,0} = H^4(B_G; \Z)
\end{align*}
are isomorphisms.  By naturality, one has 
$$
\wt \delta_4 \circ \wh \rho^* = (B_\id)^* \circ \delta_4 : \mc E_4^{0,3} \to \wt{\mc E}_4^{4,0},
$$
where $(B_\id)^* : H^*(B_G\; \Z) \to H^* (B_G; \Z)$ is the isomorphism induced by the identity $\id: G \to G$.  Hence,
\beq
\label{E:iso}
\wh \rho^* = \wt\delta_4^{-1} \circ (B_\id)^* \circ \delta_4 : H^3(Y; \Z) = \mc E_4^{0,3} \to \wt{\mc E}_4^{0,3} = H^3(\wt Y; \Z)
\eeq
is an isomorphism.

Furthermore, note that there is a principal $(\sph^3 \x \sph^3)$-bundle $\wt \sigma : \wt Y \to \wt X$, hence a homotopy fibration $\wt Y \stackrel{\wt \sigma}{\lra} \wt X \to B_{\sph^3 \x \sph^3}$, such that the following diagram commutes:
\beq
\label{E:CoverIso}
\xymatrix{
\wt Y \ar[r]^{\wt \sigma} \ar[d]_{\wh \rho} & \wt X \ar[r] \ar[d]_\rho & B_{\sph^3 \x \sph^3} \ar[d]^{B_\id} \\
Y \ar[r]^{\sigma}  & X \ar[r] & B_{\sph^3 \x \sph^3}
}
\eeq
Now, from the argument to determine $H^3(X; \Z)$, it is clear that the edge homomorphism $\sigma^* : H^3(X; \Z) = \Z \to H^3(Y; \Z) = \Z^3$ from the spectral sequence $(E_r, d_r)$ for the lower homotopy fibration is injective and maps the generator $x$ of $H^3(X; \Z)$ to a generator of $H^3(Y; \Z)$.

On the other hand, an identical argument for the Serre spectral sequence of the upper homotopy fibration in \eqref{E:CoverIso} shows that the edge homomorphism $\wt \sigma^* : H^3(\wt X; \Z) = \Z \to H^3(\wt Y; \Z) = \Z^3$ is injective and maps the generator $\wt x$ of $H^3(\wt X; \Z)$ to a generator of $H^3(\wt Y; \Z)$.

Finally, suppose that $\rho^* : H^3(X; \Z) \to H^3( \wt X; \Z)$ is given by $\rho^*(x) = \lambda \, \wt x$, for some $\lambda \in \Z$.  By the commutativity of \eqref{E:CoverIso}, 
$$
\lambda \, \wt \sigma^*(\wt x) = \wt \sigma^*(\rho^*(x)) = \wh\rho^*(\sigma^*(x)) .
$$
Together with \eqref{E:iso}, this implies that $\lambda \, \wt\sigma^*(\wt x)$ is a generator of $H^3(\wt Y; \Z)$.  However, $\wt \sigma^*(\wt x)$ is itself a generator and $\wt \sigma^*$ is injective, hence $\lambda = \pm 1$.  Therefore, $\rho^*$ is an isomorphism, as asserted.
\end{proof}

It is also possible to understand the topology of the regular (biquotient) leaves $(\sph^3 \x \sph^3)\bq \Delta Q$.

\begin{lem}
\label{L:regbiq}
The regular leaf $(\sph^3 \x \sph^3)\bq \Delta Q$ of $\Mab$ is diffeomorphic to $(\sph^3/Q) \x \sph^3$ and has cohomology groups
$$
H^j((\sph^3 \x \sph^3)\bq \Delta Q ; \Z) = 
\begin{cases}
\Z, & j = 0,6, \\
0, & j = 1,4,\\
\Z_2 \oplus \Z_2, & j = 2, 5,\\
\Z \oplus \Z, & j = 3.
\end{cases}
$$
Moreover, the homomorphism $\tau^* : H^3((\sph^3 \x \sph^3)\bq \Delta Q; \Z) \to H^3(\sph^3 \x \sph^3; \Z)$ induced by the eight-fold covering $\tau : \sph^3 \x \sph^3 \to (\sph^3 \x \sph^3)\bq \Delta Q$ is injective with image a lattice of index $8$ in $H^3(\sph^3 \x \sph^3; \Z)$.  Indeed, there is a basis $\{x_1, x_2\}$ of $H^3(\sph^3 \x \sph^3; \Z)$, such that $\im(\tau^*)$ is generated by $8 x_1$ and $x_2$.
\end{lem}

\begin{proof}
Recall that the regular leaves of $\Mab$ are quotients of principal orbits of $\Pab$, that is,
$$
(\{1\} \x \Delta \sph^3) \backslash G / \Delta Q \cong (\sph^3 \x \sph^3)\bq \Delta Q \, .
$$
The subaction by $\{1\} \x \sph^3 \x \sph^3 \In G$ on a principal orbit $G/\Delta Q$ of $\Pab$ is free with quotient $\sph^3/Q$, hence yields a principal bundle
$$
\sph^3 \x \sph^3 \to G/\Delta Q \to \sph^3/Q \, .
$$
As the classifying space $B_{\sph^3 \x \sph^3}$ is $3$-connected, it follows that $G/\Delta Q$ is trivial as a principal $(\sph^3 \x \sph^3)$-bundle.  In other words, the orbit $G/ \Delta Q$ is $(\{1\} \x \sph^3 \x \sph^3)$-equivariantly diffeomorphic to $(\sph^3/Q) \x (\sph^3 \x \sph^3)$.  Therefore, the free subaction of $\{1\} \x \Delta \sph^3 \In \{1 \} \x \sph^3 \x \sph^3$ yields a diffeomorphism
$$
(\sph^3 \x \sph^3)\bq \Delta Q \cong (\sph^3/Q) \x (\Delta \sph^3 \backslash (\sph^3 \x \sph^3)) \cong (\sph^3/Q) \x \sph^3.
$$

The classifying space $B_Q$ for $Q$ has $\pi_1(B_Q) = Q$.  As the (free) action of $Q$ on $\sph^3$ is orientation preserving, it follows that the induced action on $H^*(\sph^3; \Z)$ is trivial.  Therefore, there is a Gysin sequence for the homotopy fibration $\sph^3 \to \sph^3/Q \to B_Q$.  Now, from \cite[p.\ 59]{At}, it is known that
$$
H^j(B_Q; \Z) =
\begin{cases}
\Z, & j = 0,\\
\Z_2 \oplus \Z_2, & j \equiv 2 \!\!\! \mod 4,\\
\Z_8, & j > 0, \ j \equiv 0 \!\!\! \mod 4,\\
0, & \textrm{otherwise},
\end{cases}
$$
where the periodicity is generated by taking cup products with the generator in degree $4$.  From the Gysin sequence for $\sph^3 \to \sph^3/Q \to B_Q$, the cohomology groups of $\sph^3/Q$ are computed to be
$$
H^j(\sph^3/Q; \Z) =
\begin{cases}
\Z, & j = 0,3,\\
0, & j = 1,\\
\Z_2 \oplus \Z_2, & j = 2.
\end{cases}
$$
The cohomology groups of $(\sph^3 \x \sph^3)\bq \Delta Q$ can now be computed from the K\"unneth formula applied to the product $(\sph^3/Q) \x \sph^3$.

Finally, consider the Serre spectral sequence $(E_r, d_r)$ for the homotopy fibration $\sph^3 \x \sph^3 \stackrel{\tau}{\to} (\sph^3 \x \sph^3) \bq \Delta Q \to B_Q$.  Since $H^4((\sph^3 \x \sph^3) \bq \Delta Q; \Z) = 0$, the differential 
$$
d_4 : E_4^{0,3} = H^3(\sph^3 \x \sph^3; \Z) = \Z \oplus \Z \to E_4^{4,0} = H^4(B_Q; \Z) = \Z_8  
$$
must be surjective.  Hence, there is a basis $\{x_1, x_2\}$ of $H^3(\sph^3 \x \sph^3; \Z)$ such that $d_4(x_1)$ is a generator of $H^4((\sph^3 \x \sph^3) \bq\Delta Q; \Z)$ and $d_4(x_2) = 0$.  Clearly, this implies that the kernel of $d_4$ is generated by $8 x_1$ and $x_2$.  Consequently, the edge homomorphism $\tau^* : H^3((\sph^3 \x \sph^3)\bq \Delta Q; \Z) \to H^3(\sph^3 \x \sph^3; \Z)$ is injective with image generated by $8 x_1$ and $x_2$.
\end{proof}

\begin{thm}
\label{T:cohom} 
For $\ul a = (a_1, a_2, a_3), \ul b = (b_1, b_2, b_3) \in \Z^3$, with $a_i, b_i \equiv 1$ mod $4$ and satisfying \eqref{E:free}, define
$$
n = n(\ul a, \ul b) = \frac{1}{8} \, \det \bpm a_1^2 & b_1^2 \\ a_2^2 - a_3^2 & b_2^2 - b_3^2 \epm.
$$
If $n \neq 0$, then $H^4(\Mab; \Z)$ is cyclic of order $|n|$.  In contrast, if $n = 0$, then $H^3(\Mab; \Z) \cong H^4(\Mab; \Z) \cong \Z$.
\end{thm}

\begin{proof}
Since, by Lemma \ref{L:2conn}, $\Mab$ is $2$-connected, the Hurewicz and Universal Coefficients Theorems, together with Poincar\'e Duality, imply that the only interesting cohomology groups are $H^3(\Mab; \Z)$ and $H^4(\Mab; \Z)$, with $H^3(\Mab; \Z)$ being isomorphic to the free part of $H^4(\Mab; \Z)$.  In order to compute these, a Mayer-Vietoris argument will be used.  

Recall, from the discussion following Lemma \ref{L:nonneg}, that $\Mab$ decomposes as the union of two $2$-disk bundles $M_-$ and  $M_+$, over the singular (biquotient) leaves $(\sph^3 \x \sph^3) \bq \Pina$ and $(\sph^3 \x \sph^3) \bq \Pjnb$ respectively, which are glued along their common  (biquotient) boundary $(\sph^3 \x \sph^3) \bq \Delta Q$.  In particular, there are circle bundles
\begin{align*}
\sph^1 = \Pina/\Delta Q &\lra (\sph^3 \x \sph^3) \bq \Delta Q \stackrel{\pi_-}{\lra} (\sph^3 \x \sph^3) \bq \Pina \,, \\[1mm]
\sph^1 = \Pjnb/\Delta Q &\lra (\sph^3 \x \sph^3) \bq \Delta Q \stackrel{\pi_+}{\lra} (\sph^3 \x \sph^3) \bq \Pjnb \, ,
\end{align*}
obtained via the identifications
\begin{align}
\begin{split}
\label{E:circBun}
\partial M_- &= (\sph^3 \x \sph^3) \x_{\Pina} (\Pina/\Delta Q) \cong (\sph^3 \x \sph^3) \bq \Delta Q \,, \\[1mm]
\partial M_+ &= (\sph^3 \x \sph^3) \x_{\Pjnb} (\Pjnb/\Delta Q) \cong (\sph^3 \x \sph^3) \bq \Delta Q  \, .
\end{split}
\end{align}

Since the circle-bundle projection maps $\pi_\pm$ respect deformation retractions of $M_-$, $M_+$ and $M_- \cap M_+$ onto the respective leaves, the relevant part of the Mayer-Vietoris sequence (with integer coefficients) becomes
\begin{align}
\nonumber
0 \lra H^3(\Mab) &\lra H^3((\sph^3 \x \sph^3) \bq \Pina) \oplus H^3((\sph^3 \x \sph^3) \bq \Pjnb) \\
\label{E:MV}
&\xrightarrow{\pi_-^* - \pi_+^*} H^3((\sph^3 \x \sph^3) \bq \Delta Q) 
\lra H^{4}(\Mab) \lra 0 \,,
\end{align}
where Lemmas \ref{L:singbiq} and \ref{L:regbiq} have been applied.  In particular, $H^4(\Mab; \Z)$ is given by the cokernel of the homomorphism $\pi_-^* - \pi_+^* : \Z \oplus \Z \to \Z \oplus \Z$.

Following the proofs of \cite[Prop.\ 3.3]{GZ} and \cite[Thm.\ 13.1]{GWZ}, let $X = (\sph^3 \x \sph^3) \bq K$ be a singular leaf (as in Lemma \ref{L:singbiq}), $\rho : \wt X =  (\sph^3 \x \sph^3) \bq K^o \to X$ its two-fold cover and $\pi \in \{\pi_-, \pi_+\}$ the corresponding circle-bundle projection map $\pi:  (\sph^3 \x \sph^3) \bq \Delta Q \to X$.  Then there is a commutative diagram
$$
\xymatrix{
\sph^3 \x \sph^3 \ar[r]^(0.6){\psi} \ar[d]_{\tau} & \wt X \ar[d]^\rho \\
 (\sph^3 \x \sph^3) \bq \Delta Q \ar[r]^(0.7){\pi} & X
}
$$
given by the respective projection maps.  The induced diagram in integral cohomology is
\beq
\label{E:CommDiag}
\xymatrix{
H^3(\sph^3 \x \sph^3; \Z)  & H^3(\wt X; \Z) \ar[l]_(0.4){\psi^*}  \\
H^3((\sph^3 \x \sph^3) \bq \Delta Q; \Z)  \ar[u]^{\tau^*} & H^3(X; \Z) \ar[l]_(0.32){\pi^*} \ar[u]_{\rho^*}
}
\eeq
Using the procedure laid out in \cite{Es}, one can compute the homomorphism $\psi^*$ explicitly.  Let $\ul c = \{c_1, c_2, c_3) \in \{\ul a, \ul b\}$ be the triple describing the isomorphism $\sph^1 \to K^o \In G$.  As $\wt X$ is a biquotient and $K^o \cong \sph^1$, there is a smooth map 
$$
f: \sph^1  \to (\sph^3 \x \sph^3)^2 \, ;\, z \mapsto ((1,z^{c_2}), (z^{c_1}, z^{c_3}))
$$
defining the free action of $K^o$ on $\sph^3 \x \sph^3$, that is, $z \cdot (q_1, q_2) = (q_1 \bar z^{c_1}, z^{c_2} q_2 \bar z^{c_3})$.  

If $T = \sph^1 \x \sph^1$ is the standard maximal torus of $\sph^3 \x \sph^3$, then $T^2$ is a maximal torus of $(\sph^3 \x \sph^3)^2$ and $\im(f) \In T^2$.  Let $H^*(B_{\sph^1}; \Z) = \Z[u]$ and  $H^*(B_{T}; \Z) = \Z[t_1, t_2]$.  Then $H^*(B_{\sph^3 \x \sph^3}; \Z) = H^*(B_{T}; \Z)^W = \Z[\bar y_1, \bar y_2]$, where $W$ is the Weyl group of $\sph^3 \x \sph^3$ and $\bar y_i = t_i^2$, $i = 1,2$.  From the Serre spectral sequence $(E_r, d_r)$ for the universal principal bundle $\sph^3 \x \sph^3 \to E_{\sph^3 \x \sph^3} \to B_{\sph^3 \x \sph^3}$, generators $y_1, y_2 \in H^3 (\sph^3 \x \sph^3; \Z)$ can be chosen such that $d_4(y_i) = \bar y_i$, $i = 1,2$. Moreover, from the K\"unneth formula, it follows that $H^*(B_{(\sph^3 \x \sph^3)^2}; \Z) = \Z[\bar y_1 \ox 1, \bar y_2 \ox 1, 1 \ox \bar y_1, 1\ox \bar y_2]$.

Consider the following commutative diagram of (homotopy) fibrations:
$$
\xymatrix{
\sph^1 \ar[r] \ar[d]_f & \sph^3 \x \sph^3 \ar[r]^(0.35)\psi \ar[d]_{=} & \wt X = (\sph^3 \x \sph^3) \bq K^o \ar[r]^(0.7){\beta_1} \ar[d] & B_{\sph^1} \ar[d]^{B_f}\\
(\sph^3 \x \sph^3)^2 \ar[r] & \sph^3 \x \sph^3 \ar[r] & B_{\Delta(\sph^3 \x \sph^3)} \ar[r]^{\beta_2} & B_{(\sph^3 \x \sph^3)^2}
}
$$
In the Serre spectral sequence $(\bar E_r, \bar d_r)$ for the homotopy fibration $\beta_2$, the differential $\bar d_4 : \bar E_4^{0,3} = H^3(\sph^3 \x \sph^3 ; \Z) \to \bar E_4^{4,0} = H^4(B_{(\sph^3 \x \sph^3)^2}; \Z)$ is given by $\bar d_4(y_i) = \bar y_i \ox 1 - 1 \ox \bar y_i$, $i = 1,2$.  By naturality, the differential $\delta_4 : \mc E_4^{0,3} = H^3(\sph^3 \x \sph^3 ; \Z) \to \mc E_4^{4,0} = H^4(B_{\sph^1}; \Z)$ in the Serre spectral sequence $(\mc E_r, \delta_r)$ for the homotopy fibration $\beta_1$ is given by $\delta_4(y_i) = (B_f)^*(\bar y_i \ox 1 - 1 \ox \bar y_i)$ for $i = 1,2$.  

On the other hand, from the methods in \cite{Es} and the definition of $f : \sph^1 \to (\sph^3 \x \sph^3)^2$, the homomorphism $(B_f)^* : H^4(B_{(\sph^3 \x \sph^3)^2}; \Z) \to H^4(B_{\sph^1}; \Z)$ can be shown to be given by 
\begin{align*}
(B_f)^*(\bar y_1 \ox 1) &= 0, \hspace*{8.5mm} (B_f)^*(\bar y_2 \ox 1) = c_2^2 \, u^2, \\
(B_f)^*(1 \ox \bar y_1) &= c_1^2 \, u^2, \ \ (B_f)^*(1 \ox \bar y_2) = c_3^2 \, u^2.
\end{align*}

Therefore, $\delta_4(y_1) = - c_1^2 \, u^2$ and $\delta_4(y_2) = (c_2^2 - c_3^2) \, u^2$.  By the freeness conditions \eqref{E:free}, the coefficients of $u^2$ are relatively prime.  Hence, $\Ker \delta_4 \In H^3(\sph^3 \x \sph^3 ; \Z)$ is generated by $(c_2^2 - c_3^2)\, y_1 + c_1^2\, y_2$.  Since $\psi^* : H^3(\wt X; \Z) \to H^3(\sph^3 \x \sph^3; \Z)$ is an edge homomorphism for $(\mc E_r, \delta_r)$, it follows that there is a generator $\wt x$ of $H^3(\wt X; \Z) = \Z$ such that
\beq
\label{E:psi}
\psi^*(\wt x) = (c_2^2 - c_3^2)\, y_1 + c_1^2\, y_2 \, .
\eeq

Recall that, by Lemma \ref{L:singbiq}, $\rho^*: H^3(X; \Z) \to H^4(\wt X; \Z)$ is an isomorphism.  Thus, there is a generator $x$ of $H^3(X; \Z) = \Z$ such that $\rho^*(x) = \wt x$.  Now, since the diagram \eqref{E:CommDiag} is commutative, $\psi^*(\wt x)$ lies in the image of $\tau^* : H^3((\sph^3 \x \sph^3)/\Delta Q ; \Z) \to H^3(\sph^3 \x \sph^3 ; \Z)$.  However, $\tau^*$ is independent of the choice of triple $\ul c \in \Z^3$.  Hence, by considering the triples $\ul c = (1,1,1)$ and $\ul c = (1,-3,1)$, respectively, it is clear that $y_2$ and $8y_1 + y_2$ lie in the image of $\tau^*$ (compare the proof of \cite[Thm.\ 13.1]{GWZ}).  From this, together with Lemma \ref{L:regbiq}, it can be concluded that $8 y_1$ and $y_2$ are generators of $\im(\tau^*) \In H^3(\sph^3 \x \sph^3 ; \Z)$, and that there is a basis $\{v_1, v_2\}$ of $H^3((\sph^3 \x \sph^3)/\Delta Q ; \Z)$ such that $\tau^*(v_1) = 8y_1$ and $\tau^*(v_2) = y_2$.

It is now possible to compute the homomorphism $\pi^* : H^3(X; \Z) \to H^3((\sph^3 \x \sph^3)/\Delta Q ; \Z)$.  By \eqref{E:CommDiag} and \eqref{E:psi},
\begin{align*}
\tau^*(\pi^*(x)) &= \psi^*(\rho^*(x)) \\
&= \psi^*(\wt x)\\
&= (c_2^2 - c_3^2)\, y_1 + c_1^2\, y_2\\
&= \frac{c_2^2 - c_3^2}{8}\, 8 y_1 + c_1^2\, y_2\\
&= \tau^*\left( \frac{1}{8}(c_2^2 - c_3^2)\, v_1 + c_1^2\, v_2\right)\, .
\end{align*}
Since $\tau^*$ is injective, by Lemma \ref{L:regbiq}, it follows that
\beq
\label{E:pi}
\pi^*(x) = \frac{1}{8}(c_2^2 - c_3^2)\, v_1 + c_1^2\, v_2 \,.
\eeq

Note that, since $\tau^*$ and the basis $\{y_1, y_2\}$ of $H^3(\sph^3 \x \sph^3; \Z)$ are independent of the choice of singular leaf $X$, the basis $\{v_1, v_2\}$ of $H^3((\sph^3 \x \sph^3)/\Delta Q ; \Z)$ is independent of the choice of $X$.  Therefore, there are generators $x_{\ul a}$ and $x_{\ul b}$ of $H^3((\sph^3 \x \sph^3) \bq \Pina)$ and $H^3((\sph^3 \x \sph^3) \bq \Pjnb)$ respectively, such that \eqref{E:pi} can be applied to each of the singular leaves and the homomorphism
$$
H^3((\sph^3 \x \sph^3) \bq \Pina) \oplus H^3((\sph^3 \x \sph^3) \bq \Pjnb) \xrightarrow{\pi_-^* - \pi_+^*} H^3((\sph^3 \x \sph^3) \bq \Delta Q)
$$
is given by
\begin{align*}
(\pi_-^* - \pi_+^*)(x_{\ul a}) &= \pi_-^*(x_{\ul a}) = \frac{1}{8}(a_2^2 - a_3^2)\, v_1 + a_1^2\, v_2 \, , \\[1mm]
(\pi_-^* - \pi_+^*)(x_{\ul b}) &= -\pi_+^*(x_{\ul b}) = \frac{1}{8}(b_3^2 - b_2^2)\, v_1 - b_1^2\, v_2 \,.
\end{align*}
In order to compute the cokernel of $\pi_-^* - \pi_+^*$, note that the freeness conditions \eqref{E:free} ensure that $(\pi_-^* - \pi_+^*)(x_{\ul a})$ is a generator of $H^3((\sph^3 \x \sph^3)/\Delta Q ; \Z)$ and that there exist $r,s \in \Z$ such that $r \, a_1^2 + s \, (a_2^2 - a_3^2) = 1$.  Then a new basis for  $H^3((\sph^3 \x \sph^3)/\Delta Q ; \Z)$ is given by $w_1 = (\pi_-^* - \pi_+^*)(x_{\ul a})$ and $w_2 = r \, v_1 - 8s \, v_2$.  With respect to the basis $\{w_1, w_2\}$, $(\pi_-^* - \pi_+^*)(x_{\ul b})$ has the form
$$
(\pi_-^* - \pi_+^*)(x_{\ul b}) = (s \, (b_3^2 - b_2^2) - r\, b_1^2)\, w_1 - n \, w_2 \, ,
$$
where $n = n(\ul a, \ul b)$ is as defined in the statement of the theorem.  Therefore, if $n \neq 0$, 
\begin{align*}
H^4(\Mab; \Z) &\cong H^3((\sph^3 \x \sph^3)/\Delta Q ; \Z) / \im(\pi_-^* - \pi_+^*) \\
&\cong (\scal{w_1} \oplus \scal{w_2})/ \scal{w_1, (s \, (b_3^2 - b_2^2) - r\, b_1^2)\, w_1 - n \, w_2} \\
&\cong \scal{w_2} / \scal{n \, w_2} \\
&\cong \Z_{|n|} \, ,
\end{align*}
as desired.  Finally, it is clear that $H^3(\Mab; \Z) \cong H^4(\Mab; \Z) \cong \Z$ whenever $n=0$.
\end{proof}

\begin{proof}[Proof of Theorem \ref{T:thmB}]
Clearly, the statement follows immediately from Theorem \ref{T:cohom}, together with Lemma \ref{L:nonneg}.
\end{proof}

\begin{cor}
\label{C:spheres}
If $n = \pm 1$, then $\Mab$ is homeomorphic to $\sph^7$.  In particular, this is the case whenever $\ul a = (k, -3, 1)$ and $\ul b = (1, l, l)$, with $k, l \equiv 1$ mod $4$.
\end{cor}

\begin{proof}
Since $n = \pm 1$, $\Mab$ has the cohomology ring of a sphere, hence is a homology sphere.  As any closed, orientable manifold admits a degree $1$ map onto $\sph^n$ (by collapsing the complement of a disk to a point), there is a map inducing an isomorphism on homology.  It follows now from the homology version of the Whitehead Theorem that $\Mab$ is a homotopy sphere.  By \cite{Sm1}, $\Mab$ is then homeomorphic to $\sph^7$.
\end{proof}


\section{The Eells-Kuiper invariant of \texorpdfstring{$\Mab$}{M}}
\label{S:EK}


\subsection{A smooth metric on \texorpdfstring{$\Mab$}{M}}\hspace*{1mm}\\
\label{SS:InerOrb}

In order to compute the Eells-Kuiper invariant of $\Mab$ by applying Corollary \ref{C:AdLimEK}, it is necessary to define a suitable metric on $\Mab$.  This metric can be written down explicitly and is not the same as the metric of non-negative sectional curvature obtained in Lemma \ref{L:nonneg}.

Recall that the manifold $\Mab$ decomposes as the union of two-dimensional disk bundles $M_-$ and $M_+$ over the biquotients $(\sph^3 \x \sph^3) \bq \Pina$ and  $(\sph^3 \x \sph^3) \bq \Pjnb$ respectively, which are glued along their common boundary, the biquotient $(\sph^3 \x \sph^3) \bq \Delta Q$.  In particular, there is an action of $\Pina \cong \Pin(2)$ on the disk $\DD^2_\ve :=\{z \in \C \mid |z| < 1 + \ve\}$, $\ve > 0$, such that the disk bundle over $(\sph^3 \x \sph^3) \bq \Pina$ is given by
$$
\DD^2_\ve \to M_- = (\sph^3 \x \sph^3) \x_{\Pina} \DD^2_\ve \to (\sph^3 \x \sph^3) \bq \Pina \,.
$$

As seen in Lemma \ref{L:orbi}, by making use of the identifications given in \eqref{E:BiqDiff} (as in \eqref{E:circBun}), it turns out that the action on $\DD^2_\ve$ is nothing more than the slice representation for the isotropy group $\Pina$ of the cohomogeneity-one manifold $\Pab$\,.  As such, this action is determined by the (ineffective) transitive action of $\Pina$ on the boundary circle $\sph^1 \cong \Pina/\Delta Q$ of the normal disk to the singular orbit $G/\Pina \In \Pab$, and is equivalent to the action
\beq
\label{E:SliceAct}
\Pin(2) \x \sph^1 \to \sph^1\, ;\, (\alpha, z) \mapsto 
\begin{cases}
e^{4i\theta} z, & \alpha = e^{i\theta} \\
e^{4i\theta} \bar z, & \alpha = e^{i\theta}j
\end{cases}
\eeq
of $\Pin(2)$ on the unit circle in $\C$ with isotropy subgroup $Q$ at $1 \in \C$.  Clearly, there is an analogous action of $\Pjnb$ on $\DD^2_\ve$ which yields an analogous description of $M_+$.

Furthermore, the (equivariant) diffeomorphism 
$$
\DD^2_\ve\backslash \{0\} \to \sph^1 \x (0,1+\ve) \,;\, z \mapsto (z/|z|, |z|)
$$ 
and the transitive action \eqref{E:SliceAct} induce a diffeomorphism
\begin{align}
\begin{split}
\label{E:LeafProd-}
\Phi_- : (\sph^3 \x \sph^3) \x_{\Pina} (\DD^2_\ve \backslash\{0\}) &\to (\sph^3 \x \sph^3) \x_{\Pina} (\sph^1 \x (0,1+\ve)) \\
&\to (\sph^3 \x \sph^3) \bq \Delta Q \x (-1, \ve) \, ,
\end{split}
\end{align} 
given by mapping a point $[q_1, q_2, |z|] \in (\sph^3 \x \sph^3) \x_{\Pina} (\DD^2_\ve \backslash\{0\})$ to the point $([q_1, q_2],|z|-1) \in (\sph^3 \x \sph^3) \bq \Delta Q \x (-1, \ve)$.  

Similarly, there is a diffeomorphism
\begin{align}
\begin{split}
\label{E:LeafProd+}
\Phi_+ : (\sph^3 \x \sph^3) \x_{\Pjnb} (\DD^2_\ve \backslash\{0\}) &\to (\sph^3 \x \sph^3) \bq \Delta Q \x (-\ve,  1) \\
[q_1, q_2, |z|] &\mapsto ([q_1, q_2], 1-|z|) \,.
\end{split}
\end{align} 

Assume now that $\ve \in (0,\frac{1}{4})$ and let $\tau : \Mab \to [-1,1]$ be the projection onto the leaf space of the codimension-one foliation by biquotients, such that
\begin{align}
\begin{split}
\label{E:tau}
\tau|_{M_-}([q_1, q_2, r]) &= r -1 \,, \\[1mm]
\tau|_{M_+}([q_1, q_2, r]) &= 1 - r \,.
\end{split}
\end{align}
Then $\tau^{-1}([-1,\ve)) = M_-$, $\tau^{-1}((-\ve, 1]) = M_+$ and $\tau^{-1}(-\ve, \ve) = M_- \cap M_+$, where $M_-$ and $M_+$ are glued along neighbourhoods of their boundaries via the diffeomorphism
\begin{align}
\begin{split}
\label{E:glue}
\Phi_+^{-1} \circ \Phi_- : \tau^{-1}(-\ve, \ve) &\to \tau^{-1}(-\ve, \ve) \\
[q_1, q_2, r] &\mapsto [q_1, q_2, 2 - r] \,.
\end{split}
\end{align}

In particular, there is a diffeomorphism 
\beq
\label{E:RegPart}
\Phi : \Mab \backslash \tau^{-1}(\{-1, 1\}) \to (\sph^3 \x \sph^3) \bq \Delta Q \x (-1,1)
\eeq 
such that $\Phi|_{M_- \backslash \tau^{-1}(-1)} = \Phi_-$ and $\Phi|_{M_+ \backslash \tau^{-1}(1)} = \Phi_+$.  Given \eqref{E:RegPart}, points in $\Mab$ will often be conveniently represented as equivalence classes $[q_1, q_2, t]$, with $(q_1, q_2) \in \sph^3 \x \sph^3$ and $t \in [-1,1]$.

The manifold $\Mab$ can now be equipped with a smooth metric by pulling back via $\Phi$ a smooth metric $g_t + dt^2$ on $(\sph^3 \x \sph^3) \bq \Delta Q \x (-1,1)$ defined such that $g_t$ is a one-parameter family of smooth metrics on the biquotient $(\sph^3 \x \sph^3) \bq \Delta Q$ which deforms to smooth metrics on the singular biquotients $(\sph^3 \x \sph^3) \bq \Pina$ and $(\sph^3 \x \sph^3) \bq \Pjnb$ smoothly as $t$ tends to $-1$ and $1$, respectively.

By an abuse of notation in what follows, although the intended meaning should be clear, the symbols $\alpha$, $\beta$ and $\gamma$ will be used to denote both the indices $1,2,3$ and the imaginary unit quaternions $i,j,k$, where $1$ is identified with $i$, $2$ with $j$ and $3$ with $k$.  With this convention, define $\delta_{\alpha \beta}$ to be the Kronecker delta,
\beq
\label{E:signs}
\eps_{\alpha \beta} := 
\begin{cases}
\phantom{-} 1, & \alpha = \beta \,, \\
-1, & \alpha \neq \beta\,,
\end{cases}
\ \text{ and } \ 
\eps_{\alpha \beta \gamma} := 
\begin{cases}
\phantom{-}1, & \text{if } (\alpha, \beta, \gamma) = \,\circlearrowright \!(1,2,3),\\
-1, & \text{if } (\alpha, \beta, \gamma) = \,\circlearrowright \!(2,1,3),\\
\phantom{-}0, & \text{otherwise}.
\end{cases}
\eeq

Consider now the left-invariant vector fields $E_\alpha$ and $F_\alpha$ on $\sph^3 \x \sph^3$ defined via
\begin{align}
\begin{split}
\label{E:LIVF}
E_\alpha(q_1, q_2) &:= \frac{d}{ds} (q_1 \exp(s\alpha), q_2)|_{s = 0} = (q_1 \alpha, 0) \,, \\
F_\alpha(q_1, q_2) &:= \frac{d}{ds} (q_1, q_2 \exp(s\alpha))|_{s = 0} = (0, q_2 \alpha) \,,
\end{split}
\end{align} 
and let $X_\alpha$ denote the right-invariant vector field given by
\beq
\label{E:RIVF}
X_\alpha(q_1, q_2) := \frac{d}{ds} (q_1, \exp(s\alpha) q_2)|_{s = 0} = (0, \alpha q_2) \,.
\eeq

Equip $\sph^3 \x \sph^3$ with the standard, bi-invariant product metric $\<\,,\>_0$ so that the six vector fields $E_\alpha$ and $F_\beta$ describe a global orthonormal basis.  As the right-invariant vector fields $X_\alpha$ can be written in terms of the basis $F_\beta$, there are smooth coefficient functions
\beq
\label{E:Xcoeffs}
\vphi_{\alpha \beta} : \sph^3 \x \sph^3 \to \R \,;\, (q_1, q_2) \mapsto \<X_\alpha, F_\beta\>_0 = \<\Ad_{\bar q_2} \alpha, \beta\>_0 \,,
\eeq
such that the $(3\x 3)$-matrix $(\vphi_{\alpha \beta}(q_1, q_2))_{\alpha, \beta}$ is an element of $\SO(3)$.  The derivatives of the functions $\vphi_{\alpha \beta}$ are given by $E_\gamma (\vphi_{\alpha \beta}) = 0$ and
\begin{align}
\begin{split}
\label{E:phiDerivs}
(F_\gamma (\vphi_{\alpha \beta}))(q_1, q_2) &=  -\<[\gamma, \Ad_{\bar q_2} \alpha], \beta \>_0 \\
&= \<\Ad_{\bar q_2} \alpha, [\gamma, \beta] \>_0 \\
&= 2 \, \sum_{\delta = 1}^3 \eps_{\gamma \beta \delta} \, \vphi_{\alpha \delta}(q_1, q_2) \,.\\
\end{split}
\end{align}

Recall from \eqref{E:BiqDiff} that the free (right) action of $\Delta Q$ on $\sph^3 \x \sph^3$ is given by the anti-homomorphism
$$
\rho : Q \to \Diff(\sph^3 \x \sph^3) \,,
$$ 
where 
$$
\rho(\pm 1)(q_1. q_2) 
= (\pm q_1, q_2) \ \text{ and } \ \rho(\pm \alpha) 
= (\pm q_1 \alpha, \bar \alpha q_2 \alpha) \,.
$$

Although the vector fields $E_\alpha$, $F_\alpha$ and $X_\alpha$ are not $Q$ invariant, it is easy to describe their behaviour under the $Q$ action.

\begin{lem}
\label{L:Qinv}
The vectors fields $E_\alpha$, $F_\alpha$ and $X_\alpha$ satisfy the identities
$$
\rho(\pm 1)_* E_\alpha = E_\alpha, \qquad  
\rho(\pm 1)_* F_\alpha = F_\alpha, 
\qquad \rho(\pm 1)_* X_\alpha = X_\alpha,
$$
and
$$
\rho(\pm \beta)_* E_\alpha = \eps_{\alpha \beta} \, E_\alpha, \qquad
\rho(\pm \beta)_* F_\alpha = \eps_{\alpha \beta} \, F_\alpha, 
\qquad \rho(\pm \beta)_* X_\alpha = \eps_{\alpha \beta} \, X_\alpha.
$$
Furthermore, the functions $\vphi_{\alpha \beta}$ satisfy
$$
\vphi_{\alpha \beta} \circ \rho(\pm 1) = \vphi_{\alpha \beta} \ \text{ and } \ 
\vphi_{\alpha \beta} \circ \rho(\pm \gamma) = \eps_{\alpha \gamma} \eps_{\beta \gamma} \, \vphi_{\alpha \beta}.
$$
\end{lem}

\begin{proof}
Since $\beta^2 = -1$ and $\rho$ is an anti-homomorphism, all non-trivial cases follow from the case $\rho(\beta)_*$.  For $E_\alpha$ one has
\begin{align*}
(\rho(\beta)_* E_\alpha)(\rho(\beta) (q_1, q_2)) &= \frac{d}{ds} \rho(\beta)(q_1 \exp(s \alpha), q_2)|_{s = 0} \\
&= \frac{d}{ds} (q_1 \exp(s \alpha) \beta, \bar \beta q_2 \beta)|_{s = 0}\\
&= \frac{d}{ds} (q_1 \beta (\bar \beta \exp(s \alpha) \beta), \bar \beta q_2 \beta)|_{s = 0}\\
&= \frac{d}{ds} (q_1 \beta \exp(s \Ad_{\bar \beta} \alpha), \bar \beta q_2 \beta)|_{s = 0}\\
&= \frac{d}{ds} (q_1 \beta \exp(s \, \eps_{\alpha \beta} \, \alpha), \bar \beta q_2 \beta)|_{s = 0}\\
&= \eps_{\alpha \beta} \, E_\alpha(q_1 \beta, \bar \beta q_2 \beta).
\end{align*}

The computations in the cases of $F_\alpha$ and $X_\alpha$ are analogous.  The identities for the functions $\vphi_{\alpha \beta}$ follow from those for the vector fields since $\< \,, \>_0$ is bi-invariant and, for example,
$$
\vphi_{\alpha \beta} \circ \rho(\gamma) = \<X_\alpha \circ \rho(\gamma), F_\beta \circ \rho(\gamma) \>_0 = \eps_{\alpha \gamma} \, \eps_{\beta \gamma} \<\rho(\gamma)_* X_\alpha , \rho(\gamma)_* F_\beta \>_0 \,.
$$
\end{proof}

Consider finally a local basis $e_\alpha$, $f_\alpha$ of vector fields on a neighbourhood of a point $[q_1, q_2] \in (\sph^3 \x \sph^3) \bq \Delta Q$ given by the projections under the quotient map of the restrictions of the vector fields $E_\alpha$ and $F_\alpha$ to a neighbourhood of a fixed representative $(q_1, q_2) \in \sph^3 \x \sph^3$.  In particular, 
\begin{align}
\begin{split}
\label{E:efVF}
e_\alpha([q_1, q_2]) &= \frac{d}{ds} [q_1 \exp(s\alpha), q_2]|_{s = 0} \,, \\[1mm]
f_\alpha([q_1, q_2]) &= \frac{d}{ds} [q_1, q_2 \exp(s\alpha)]|_{s = 0} \,,
\end{split}
\end{align} 
Similarly, let $x_\alpha$ denote the vector field given near $[q_1, q_2] \in (\sph^3 \x \sph^3) \bq \Delta Q$ by projection of $X_\alpha$ restricted to a neighbourhood of $(q_1, q_2)$, so that
\beq
\label{E:xVF}
x_\alpha([q_1, q_2]) = \frac{d}{ds} [q_1, \exp(s\alpha) q_2]|_{s = 0} \,.
\eeq
Observe that near a different representative $\rho(\ell)(q_1, q_2) \in \sph^3 \x \sph^3$, $\ell \in Q$, the pre-images of $e_\alpha$, $f_\alpha$, and $x_\alpha$ are given by the vector fields $(\rho(\ell)_* E_\alpha)$, $(\rho(\ell)_* F_\alpha)$ and $(\rho(\ell)_* X_\alpha)$ respectively.  That is, repeating the above construction of $e_\alpha$, $f_\alpha$ using instead the representative $\rho(\ell)(q_1, q_2)$ would produce a local basis $e_\alpha'$, $f_\alpha'$ near $[q_1, q_2] \in (\sph^3 \x \sph^3) \bq \Delta Q$ which differs from the previous one at most by a sign.

It is perhaps important to emphasise at this point that, when compared with the notation used in \cite{Gojems}, the roles of $e_\alpha$ and $f_\alpha$ have been switched.  In the present article, the notation $f_\alpha$ is intended to suggest that the vector field is tangent to the fibre of the Seifert fibration.

With the chosen representative $(q_1, q_2)$ of $[q_1, q_2]$ in mind, it is convenient to abuse the notation in \eqref{E:Xcoeffs} and define
\beq
\label{xcoeffs}
\vphi_{\alpha \beta} :  (\sph^3 \x \sph^3) \bq \Delta Q \to \R \,;\, [q_1, q_2] \mapsto \vphi_{\alpha \beta}(q_1, q_2) \,.
\eeq
Then $x_\alpha = \sum_{\beta = 1}^3 \vphi_{\alpha \beta} \, f_\beta$ and, moreover, the derivatives of the functions $\vphi_{\alpha \beta}$ on $(\sph^3 \x \sph^3) \bq \Delta Q$ are given by $e_\gamma (\vphi_{\alpha \beta}) = 0$ and
\beq
\label{E:phiderivs}
f_\gamma (\vphi_{\alpha \beta}) = 2 \, \sum_{\delta = 1}^3 \eps_{\gamma \beta \delta} \, \vphi_{\alpha \delta} \,.
\eeq

\begin{prop}
\label{P:metric}
For $\ve \in (0, \frac{1}{4})$, let $\sigma : \R \to \R$ denote a smooth function such that $\sigma|_{(-\infty, \ve-1)} \equiv 1$, $\sigma|_{(-\ve, \infty)} \equiv 0$ and $\sigma'(x) \leq 0$ for all $x \in \R$.  Furthermore, let
\begin{align*}
\lambda_-&: [-1,1] \to \R \,;\, t \mapsto 
 (1+t) \, \sigma(t) + \frac{|a_1|}{4}(1 - \sigma(t)) \,, \\[1mm]
\lambda_+&: [-1,1] \to \R \,;\, t \mapsto 
 \, (1-t) \, \sigma(-t) + \frac{|b_1|}{4}(1 - \sigma(-t)) \,,
\end{align*}
and let $h_\pm, u_\pm, v_\pm : (\sph^3 \x \sph^3) \bq \Delta Q \x (-1,1) \to \R$ be defined by
\begin{align*}
h_- &:= \frac{4}{|a_1|} \lambda_- \circ \tau \circ \Phi^{-1} \,, \ \ \ 
h_+ := \frac{4}{|b_1|} \lambda_+ \circ \tau \circ \Phi^{-1} \,, \\[1mm]
u_- &:= \frac{a_2}{a_1} \lambda_-' \circ \tau \circ \Phi^{-1} \,, \ \ \
u_+ := -\frac{b_2}{b_1} \lambda_+' \circ \tau \circ \Phi^{-1} \,, \\[1mm]
v_- &:= \frac{a_3}{a_1} \lambda_-' \circ \tau \circ \Phi^{-1} \,, \ \ \
v_+ := -\frac{b_3}{b_1} \lambda_+' \circ \tau \circ \Phi^{-1} \,,
\end{align*}
Then the metric $g_t + dt^2$ on $(\sph^3 \x \sph^3) \bq \Delta Q \x (-1,1)$ given by 
\begin{align*}
g_t(e_\alpha, e_\beta) &= \delta_{\alpha \beta} \left( 
\stack{1 + \delta_{1 \alpha} (h_-^2 + u_-^2 - 2 u_- v_- \,\vphi_{11} + v_-^2 - 1)}
{\phantom{1} + \delta_{2 \alpha} (h_+^2 + u_+^2 - 2 u_+ v_+ \,\vphi_{22} + v_+^2 - 1)} 
\right) , \\[1mm]
g_t(f_\alpha, f_\beta) &= \delta_{\alpha \beta} \,, \\[1mm]
g_t(e_\alpha, f_\beta) &= \delta_{1 \alpha} (u_- \, \vphi_{1 \beta} - v_- \, \delta_{1 \beta}) + \delta_{2 \alpha} (u_+ \, \vphi_{2 \beta} - v_+ \, \delta_{2 \beta}), 
\end{align*}
with respect to the local basis $e_\alpha$, $f_\alpha$ of vector fields on a neighbourhood of a point $[q_1, q_2] \in (\sph^3 \x \sph^3) \bq \Delta Q$ constructed above, is well defined and smooth.  Moreover, $\Phi^*(g_t + dt^2)$ extends to a smooth metric $g_M$ on $\Mab$. 
\end{prop}

Some remarks on the metric $g_M$ are in order.  By an abuse of notation, define vector fields $e_0 := \Phi^{-1}_*(\pd{t})$, $e_\alpha := \Phi^{-1}_*(e_\alpha)$, $f_\alpha := \Phi^{-1}_*(f_\alpha)$ and $x_\alpha := \Phi^{-1}_*(x_\alpha)$ on $\Mab$\,.   By a further abuse of notation, let $h_\pm$, $u_\pm$ and $v_\pm$ denote the smooth functions on $\Mab$ given by $h_\pm \circ \Phi$,  $u_\pm \circ \Phi$ and $v_\pm \circ \Phi$ respectively.  For $r \in (0, \ve)$, the vector fields tangent to the fibres of the radius-$r$ circle bundles over the zero sections of $M_-$ and $M_+$ are given by
\begin{align*}
\tpd{\theta}_- &= \frac{1}{4}(a_1 e_1 - a_2 x_1 + a_3 f_1) ,\\[1mm]
\tpd{\theta}_+ &= \frac{1}{4}(b_1 e_2 - b_2 x_2 + b_3 f_2) ,
\end{align*}
respectively, with lengths $g_M (\pd{\theta}_\pm, \pd{\theta}_\pm) = (1 \pm t)^2 = r^2$.  The vanishing of these vector fields for $r$ going to $0$ corresponds to the  roles of $\Pina$ and $\Pjnb$ at the singular biquotients. 

A (local) orthonormal frame of vector fields on $(\Mab, g_M)$ is described by 
\begin{align}
\begin{split}
\label{E:ONB}
\bar e_0 &:= e_0 , \qquad \bar e_1 := \frac{1}{h_-}(e_1 - u_- \, x_1 + v_- \, f_1), \\[1mm]
\bar e_2 &:= \frac{1}{h_+}(e_2 - u_+ \, x_2 + v_+ \, f_2), \qquad \bar e_3 := e_3,  \\[1mm]
\bar f_ \alpha &:= f_\alpha \,
\end{split}
\end{align}
away from the singular leaves $\tau^{-1}\{\pm 1\}$.  Notice that 
\beq
\label{E:compatibility}
\bar e_1|_{\tau^{-1}(-1, \ve -1)} = \frac{|a_1|}{ a_1 (1+t)} \, \pd{\theta}_- \,, \ \quad \ 
\bar e_2|_{\tau^{-1}(1-\ve, 1)} = \frac{|b_1|}{b_1 (1-t)} \, \pd{\theta}_+ \,,
\eeq
so the singularity in $\bar e_1$ and $\bar e_2$ along $\tau^{-1}\{-1\}$ and $\tau^{-1}\{1\}$, respectively, comes only from the normalisation by $\frac{1}{h_\pm}$.  Moreover, note that $\bar e_1|_{M_+} = e_1$ and $\bar e_2|_{M_-} = e_2$.  For convenience, define also $\bar x_\alpha := x_\alpha$.  

In addition, since $h_\pm |_{M_\mp} \equiv 1$, $u_\pm|_{M_\mp} \equiv 0$ and $v_\pm|_{M_\mp} \equiv 0$,  the metric on $M_- \cap M_+ = \tau^{-1}(-\ve, \ve)$ is quite simple: it is the product of a normal biquotient and an interval, where the vector fields $e_0$, $e_\alpha$, and $f_\alpha$ are orthonormal.  

Finally, the minus sign in the definition of $u_+$ and $v_+$ is to ensure that the isometry $\Psi$ in Remark \ref{R:Isom} below is compatible with Lie brackets.

\begin{proof}[Proof of Proposition \ref{P:metric}]
The strategy of the proof is to define a smooth $\Delta Q$-invariant metric $\hat g = \hat g_t + dt^2$ on $\sph^3 \x \sph^3 \x (-1,1)$ and equip $(\sph^3 \x \sph^3) \bq \Delta Q \x (-1,1)$ with the induced smooth submersion metric $g = g_t + dt^2$.  The smoothness as $t \to \pm 1$ will be obtained by defining a smooth quotient metric on $M_\pm$ which coincides with $\Phi_\pm^*(g|_{M_\pm})$ near the zero section.  

Consider $\sph^3 \x \sph^3 \x (-1,1)$ equipped with the metric $\hat g := \hat g_t + dt^2$, where
\begin{align*}
\hat g_t(E_\alpha, E_\beta) &= \<E_\alpha, E_\beta \>_0 \\
&\hspace*{10mm} + \left(h_-^2 + u_-^2 - 2 u_- v_- \, \vphi_{11} + v_-^2 - 1\right) \! \<E_1, E_\alpha \>_0 \, \<E_1, E_\beta \>_0 \\[1mm]
&\hspace*{10mm} + \left(h_+^2 + u_+^2 - 2 u_+ v_+ \, \vphi_{22} + v_+^2 - 1\right) \! \<E_2, E_\alpha \>_0 \, \<E_2, E_\beta \>_0 \,, \\[1mm]
\hat g_t(F_\alpha, F_\beta) &= \<F_\alpha, F_\beta\>_0\,, \\[1mm]
\hat g_t(E_\alpha, F_\beta) &= \hat g_r(F_\beta, E_\alpha) = \<E_1, E_\alpha\>_0 \,  \< u_- X_1 - v_- F_1, F_\beta \>_0 \\[1mm]
&\hspace*{35mm} + \<E_2, E_\alpha\>_0 \,  \< u_+ X_2 - v_+ F_2, F_\beta \>_0 \,.
\end{align*}
It is clear that $\hat g$ is a smooth metric on $\sph^3 \x \sph^3 \x (-1,1)$, since it is a positive-definite, symmetric, bi-linear form and all terms used to define $\hat g_t$ are smooth functions on $\sph^3 \x \sph^3 \x (-1,1)$.

Observe now that $\Delta Q$ acts on $(\sph^3 \x \sph^3 \x (-1, 1), \hat g)$ by isometries.  This follows from Lemma \ref{L:Qinv} and the bi-invariance of $\< \,, \>_0$.  For example, from
\begin{align*}
\<E_1, \rho(\gamma)_* E_\alpha\>_0  \circ \rho(\gamma) 
&= \eps_{1 \gamma} \<\rho(\gamma)_* E_1, \rho(\gamma)_* E_\alpha\>_0  \circ \rho(\gamma) \\
&= \eps_{1 \gamma} \<E_1, E_\alpha\>_0 \, , \\
\intertext{together with}
\< u X_1 - v F_1, \rho(\gamma)_* F_\beta \>_0 \! \circ\! \rho(\gamma) 
&=  \eps_{1 \gamma} \<\rho(\gamma)_*(u X_1 - v F_1), \rho(\gamma)_* F_\beta \>_0 \! \circ\! \rho(\gamma) \\
&=  \eps_{1 \gamma} \<u X_1 - v F_1, F_\beta \>_0 \,,
\end{align*} 
follows $\hat g_t(\rho(\gamma)_* E_\alpha, \rho(\gamma)_* F_\beta) \circ \rho(\gamma) = \hat g_t(E_\alpha, F_\beta)$.  Consequently, the metric $\hat g$ induces a smooth metric $g := g_t + dt^2$ on $(\sph^3 \x \sph^3) \bq \Delta Q \x (-1,1)$ such that the quotient map is a Riemannian submersion.

It remains, therefore, to show that the metric $\Phi^* g$ extends to a smooth metric at $t = \pm 1$.  It is sufficient to concentrate on $t = -1$, since the arguments for $t = 1$ are completely analogous and involve only replacing the triple $\ul a$ with $\ul b$\,, and the functions $h_-$, $u_-$ and $v_-$ with $h_+$, $u_+$ and $v_+$ respectively.  

Let $E_\alpha$, $F_\alpha$ and $X_\alpha$ be the vector fields on the $\sph^3 \x \sph^3$ factor of $\sph^3 \x \sph^3 \x \DD^2_\ve$ defined as in \eqref{E:LIVF} and \eqref{E:RIVF}, and let $\pd{r}$ and $\pd{\theta}$ denote the polar-coordinate radial and angular vector fields on the $\DD^2_\ve$ factor respectively, that is, 
\begin{align*}
\tpd{r}(q_1, q_2, r e^{i\theta}) &:= \frac{d}{ds} (q_1, q_2, (r + s) e^{i\theta}) |_{s=0} = (0,0, e^{i \theta}) \,, \\[1mm]
\tpd{\theta}(q_1, q_2, r  e^{i\theta}) &:= \frac{d}{ds} (q_1, q_2, r e^{i(\theta + s)}) |_{s=0} = (0,0, r i e^{i \theta}) \,.
\end{align*}

Choose the same representative $(q_1, q_2) \in \sph^3 \x \sph^3$ of $[q_1, q_2] \in (\sph^3 \x \sph^3) \bq \Delta Q$ as in \eqref{E:efVF} and \eqref{E:xVF}.  Notice that $E_\alpha$, $F_\alpha$ and $X_\alpha$ are tangent to the set $\mc Z$ of points in $\sph^3 \x \sph^3 \x \DD^2_\ve$ with $\theta = 0$.   Define vector fields in a neighbourhood of $[q_1, q_2, r] \in M_-$ as the projections of the restrictions of $E_\alpha$, $F_\alpha$ and $X_\alpha$ to the intersection of $\mc Z$ with a neighbourhood of $(q_1, q_2, r) \in \sph^3 \x \sph^3 \x \DD^2_\ve$.  By an abuse of notation, denote these projections again by $e_\alpha$, $f_\alpha$ and $x_\alpha$ respectively, since they are $\Phi_-$-related to the previously defined vector fields $e_\alpha$, $f_\alpha$ and $x_\alpha$ on $(\sph^3 \x \sph^3) \bq \Delta Q \x (-1,\ve)$. 

If $g_\DD := dr^2 + r^2 d\theta^2$ denotes the standard metric on $\DD^2_\ve$ in polar coordinates, let $\tilde g_r + dr^2$ be the bi-linear form on $\sph^3 \x \sph^3 \x \DD^2_\ve$ defined by
\begin{align*}
\tilde g_r(E_\alpha, E_\beta) &= \<E_\alpha, E_\beta \>_0 \\
&\hspace*{4mm}
+ \frac{1}{a_1^2} \!\left(16 r^2 + a_2^2 - 2 a_2 a_3 \, \vphi_{11} + a_3^2 - a_1^2 + 1 \right) \!\! \<E_1, E_\alpha \>_0 \<E_1, E_\beta \>_0 , \\[1mm]
\tilde g_r(F_\alpha, F_\beta) &= \<F_\alpha, F_\beta\>_0\,, \\[1mm]
\tilde g_r(E_\alpha, F_\beta) &= \tilde g_r(F_\beta, E_\alpha) = \<E_1, E_\alpha\>_0 \,  \< \frac{a_2}{a_1} X_1 -  \frac{a_3}{a_1} F_1, F_\beta \>_0 \,, \\[1mm]
\tilde g_r(E_\alpha, \tpd{\theta}) &= \tilde g_r(\tpd{\theta}, E_\alpha) = \frac{4}{a_1} \<E_1, E_\alpha\>_0 \, g_\DD (\tpd{\theta}, \tpd{\theta}) \,, \\[1mm]
\tilde g_r(F_\alpha, \tpd{\theta}) &= \tilde g_r(\tpd{\theta}, F_\alpha) = 0 \,,  \\[1mm]
\tilde g_r(\tpd{\theta}, \tpd{\theta}) &= g_\DD (\tpd{\theta}, \tpd{\theta}) \,.
\end{align*}

The bi-linear form $\tilde g_r + dr^2$ is symmetric and positive definite.  By reverting to Cartesian coordinates on $\DD^2_\ve$, it is easily verified that $\tilde g_r + dr^2$ describes a smooth metric on $\sph^3 \x \sph^3 \x \DD^2_\ve$ for all $r \in[0,1+\ve)$.    

Let $V$ be the vector field on $\sph^3 \x \sph^3 \x \DD^2_\ve$ tangent to the free $\Pina$ action and given by
\beq
V = -a_1 E_1 + a_2 X_1 - a_3 F_1 + 4 \tpd{\theta} \,. 
\eeq 

Notice that $\tilde g_r (V, V) \equiv 1$ and that,  near $(q_1, q_2, r) \in \sph^3 \x \sph^3 \x \DD^2_\ve$, the vector fields $E_2$, $E_3$, $F_\alpha$, $\pd r$ and $\pd{\theta}$ are all orthogonal to $V$ with respect to $\tilde g_r + dr^2$, hence horizontal.  Observe, however, that the horizontal vector field which projects to $e_1$ is given by
$$
E_1 + \frac{1}{a_1} V = \frac{a_2}{a_1} X_1 - \frac{a_3}{a_1} F_1 + \frac{4}{a_1} \tpd{\theta} \,
$$
and is of length $\frac{1}{a_1^2}(16 r^2 + a_2^2 - 2 a_2 a_3 \, \vphi_{11} + a_3^2)$.  Therefore, as $\lambda_-|_{(-1,\ve-1)} = r$, if the action of $\Pina$ on $(\sph^3 \x \sph^3 \x \DD^2_\ve, \tilde g_r + dr^2)$ is by isometries, the induced smooth submersion metric on $M_-$ will coincide with the metric $\Phi_-^* (g)|_{M_-}$ for $r \in (0, \ve)$, as desired.  (For the $t = 1$ case, note in addition that $(\Phi_+^{-1})_*(\pd{t}) = - \pd{r}$.)

From \eqref{E:BiqDiff} and \eqref{E:SliceAct}, the $\Pina$ action on $\sph^3 \x \sph^3 \x \DD^2_\ve$ is given by 
$$
\tilde \rho : \Pin(2) \to \Diff(\sph^3 \x \sph^3 \x \DD^2_\ve) \,,
$$
which is completely determined by
\begin{align*}
\tilde \rho(e^{it})(q_1, q_2, z) &= (q_1 e^{a_1 i t}, e^{-a_2 i t} q_2 e^{a_3 i t}, e^{4it} z) \,, \\
\tilde \rho(j)(q_1, q_2, z) &= ( q_1 j, - j q_2 j, \bar z) \,.
\end{align*}

The invariance under $\tilde \rho(j)$ of the first three expressions in the definition of $\tilde g_r$ follows via Lemma \ref{L:Qinv} precisely as for the corresponding terms in the case of $\hat g_t$.  As $j$ acts by conjugation on $\DD^2_\ve$, it follows that $\tilde \rho (j)_* \pd{r} = \pd{r}$ and $\tilde \rho (j)_* \pd{\theta} = - \pd{\theta}$, which, together with the identities in Lemma \ref{L:Qinv}, yield that the remaining terms in the definition of $\tilde g_r$ are invariant under $\tilde \rho(j)$.

On the other hand, the transformation rules under $\tilde \rho(e^{it})$ are given by
\begin{align*}
\tilde \rho(e^{it})_* E_\alpha 
&= 
\begin{cases}
E_1 \,,& \alpha = 1,\\
\cos(2 a_1 t) E_2 - \sin(2 a_1 t) E_3 \,, & \alpha = 2, \\
\sin(2 a_1 t) E_2 + \cos(2 a_1 t) E_3 \,, & \alpha = 3,
\end{cases} \\
\tilde \rho(e^{it})_* F_\alpha &= 
\begin{cases}
F_1 \,,& \alpha = 1,\\ 
\cos(2 a_3 t) F_2 - \sin(2 a_3 t) F_3 \,, & \alpha = 2, \\
\sin(2 a_3 t) F_2 + \cos(2 a_3 t) F_3 \,, & \alpha = 3,
\end{cases}
\end{align*}
as well as $\tilde \rho(e^{it})_* \tpd{\theta} = \tpd{\theta}$, $\tilde \rho(e^{it})_* \tpd{r} = \tpd{r}$ and, in particular, $\tilde \rho(e^{it})_* X_1 = X_1$.  It then follows that $\vphi_{11} \circ \tilde \rho(e^{it}) = \vphi_{11}$ and, moreover, that 
\begin{align*}
\<X_1 , \tilde \rho(e^{it})_* F_\beta\>_0  \circ \tilde \rho(e^{it}) 
&= \<\tilde \rho(e^{it})_* X_1, \tilde \rho(e^{it})_* F_\beta \>_0  \circ \tilde \rho(e^{it}) \\
&= \< X_1 , F_\beta \>_0,
\end{align*}
Similarly, $\<F_1 , \tilde \rho(e^{it})_* F_\beta\>_0  \circ \tilde \rho(e^{it}) = \< F_1 , F_\beta \>_0$.  The $\Pina$ invariance of $\tilde g_r + dr^2$ is now a simple consequence of these identities.
\end{proof}

\begin{rem}
\label{R:Isom}
Consider the manifold $M^7_{\ul b, \ul a}$ given by swapping $\ul a$ and $\ul b$.  Equip $M^7_{\ul b, \ul a}$ with a metric $g_M'$ defined in the same way as $g_M$ by simply switching the roles of $\ul a$ and $\ul b$ in $\lambda_\pm$, $h_\pm$, $u_\pm$ and $v_\pm$.  Let $\ell := \frac{1}{\sqrt 2}(i+j) \in \sph^3$.  Then, since the diffeomorphism
$$
\sph^3 \x \sph^3 \to \sph^3 \x \sph^3 \,;\, (q_1, q_2) \mapsto (\ell q_1 \bar \ell, \ell q_2 \bar \ell)
$$
respects $\Delta Q$ fibres and intertwines the $\Pina$ action on $\sph^3 \x \sph^3$ with the action of $\Pjn_{\ul a}(2)$, there is an induced orientation-reversing isometry 
\begin{align}
\begin{split}
\label{E:Mirror}
\Psi: (\Mab, g_M) &\to (M^7_{\ul b, \ul a}, g_M') \\[1mm]
[q_1, q_2, t] &\mapsto [\ell q_1 \bar \ell, \ell q_2 \bar \ell, -t]
\end{split}
\end{align}
mapping $\tau^{-1}(t) \In \Mab$ to $\tau^{-1}(-t) \In M^7_{\ul b, \ul a}$ for each $t \in [-1,1]$.  In particular, $\Psi_*$ maps $e_0$ to $-e_0$ and
\[
\begin{aligned}
e_1 \mapsto e_2 , &&&&&&& e_2 \mapsto e_1 , &&&&& e_3 \mapsto -e_3 , \\
f_1 \mapsto f_2 , &&&&&&& f_2 \mapsto f_1 , &&&&& f_3 \mapsto -f_3 , \\
x_1 \mapsto x_2 , &&&&&&& x_2 \mapsto x_1 , &&&&& x_3 \mapsto -x_3 .
\end{aligned}
\] 
Furthermore, notice that $\lambda_\pm(-t) = \lambda_\mp(t)$ and $\lambda'_\pm(-t) = - \lambda'_\mp(t)$, so that, for example, $u_-$, $v_-$ on $M^7_{\ul b, \ul a}$ correspond to $u_+$, $v_+$ on $\Mab$.  The isometry $\Psi$ hence ensures that the computations to follow need only be performed on $M_-$ (and that the expressions involving parentheses containing terms with `$+$' subscripts may be ignored).
\end{rem}

Given that the orbifold $\Bab$ admits a (cohomogeneity-one) $\sph^3$ action, hence is foliated by $\sph^3$-orbits, one can define vector fields $e_0$, $e_1$, $e_2$ and $e_3$ locally on $\Bab$ exactly as in the case of $\Mab$, that is, as the projections of local left-invariant vector fields.  In the same way as in Proposition \ref{P:metric}, one can obtain a metric on $\Bab$.

\begin{cor}
\label{C:metricB}
With the analogous notation as in Proposition \ref{P:metric},  the metric $\check g_t + dt^2$ on $\sph^3 / Q \x (-1,1)$ given by
$$
\check g_t(e_\alpha, e_\beta) = \delta_{\alpha \beta} \left( 
1 + \delta_{1 \alpha} (h_-^2 - 1) + \delta_{2 \alpha} (h_+^2 - 1) 
\right) , 
$$
is smooth and pulls back to a (globally) smooth metric $g_B$ on $\Bab$.  In particular, away from the singular orbits of $\Bab$, a local orthonormal frame of vector fields on $(\Bab, g_B)$ is described by
$$
\check e_0 := e_0, \ \ \ \check e_1 := \frac{1}{h_-} \, e_1,  \ \ \ \check e_2 := \frac{1}{h_+} \, e_2, \ \ \ \check e_3 := e_3 \,.
$$
Moreover,  the Seifert fibration $\pi : (\Mab, g_M) \to (\Bab, g_B)$ is a Riemannian submersion and the vector fields $\bar e_0, \dots, \bar e_3$ on $\Mab$ are the horizontal lifts of the orthonormal vector fields $\check e_0, \dots, \check e_3$.
\end{cor}

\medskip

\subsection{Chern-Weil forms for $\pi : (\Mab, g_M) \to (\Bab, g_B)$} \hspace*{1mm}\\
\label{SS:ChernWeil}

The basic tool used to determine the various invariants involved in the computation of the Eells-Kuiper invariant is Chern-Weil theory.  The necessary ingredients are gathered together in this section.  By Remark \ref{R:Isom}, only the computations for $M_-$ need to be carried out explicitly, and all expressions in Proposition \ref{P:metric} involving parentheses containing terms with `$+$' subscripts may be ignored.  

With this in mind, and to simplify the expressions to follow, it is convenient to define smooth functions $h, u, v : \Mab \to \R$ such that 
\beq
\begin{aligned}
h|_{M_-} &= h_-|_{M_-} \,, \ \ & 
u|_{M_-} &= u_-|_{M_-} \,, \ \ &
v|_{M_-} &= v_-|_{M_-} \,, \\
h|_{M_+} &= h_+|_{M_+} \,, \ \ &
u|_{M_+} &= u_+|_{M_+} \,, \ \ &
v|_{M_+} &= v_+|_{M_+} \,.
\end{aligned}
\eeq
For the sake of notation, the shorthand $h'$, $u'$ and $v'$ will be used to denote $\bar e_0(h) = \PD{h}{t}$, $\bar e_0(u) = \PD{u}{t}$ and $\bar e_0(v) = \PD{v}{t}$ respectively.  Notice, in particular, that
\beq
\begin{aligned}
\label{E:uvh}
u|_{M_-} &= \frac{a_2 |a_1|}{4 a_1} h'|_{M_-} \,, \ \ &
v|_{M_-} &= \frac{a_3 |a_1|}{4 a_1} h'|_{M_-} \,, \\
u|_{M_+} &= - \frac{b_2 |b_1|}{4 b_1} h'|_{M_+} \,, \ \ &
v|_{M_+} &= - \frac{b_3 |b_1|}{4 b_1} h'|_{M_+} \,. 
\end{aligned}
\eeq

\begin{lem}
\label{L:LieBracketsM}
The vector fields $\bar e_0, \dots, \bar e_3$ and $\bar f_1, \dots, \bar f_3$ on $\Mab$ satisfy the following Lie bracket identities:
\begin{equation*}
  \begin{aligned}\relax
    [\bar e_0,\bar e_1]|_{M_-}
    &\rlap{$\displaystyle\mathord{}=-\frac{h'}{h}\,\bar e_1
	-\frac{u'}{h}\,\bar x_1+\frac{v'}{h}\,\bar f_1\;,$}&&&
    [\bar e_1,\bar e_2]|_{M_-}
    &=\frac2{h}\,\bar e_3\;,\\
    [\bar e_2,\bar e_3]|_{M_-}
    &\rlap{$\displaystyle\mathord{}=2 \, h \, \bar e_1
	+2 \, u \, \bar x_1-2 \, v \, \bar f_1\;,$}&&&
    [\bar e_3,\bar e_1]|_{M_-}
    &=\frac2{h}\,\bar e_2\;,\\
    [\bar f_1,\bar f_2]|_{M_-}
    &=2 \, \bar f_3\;,&
    [\bar f_2,\bar f_3]|_{M_-}
    &=2 \, \bar f_1\;,&
    [\bar f_3,\bar f_1]|_{M_-}
    &=2 \, \bar f_2\;,\\
    [\bar e_1,\bar f_1]|_{M_-}
    &=0\;,& 
    [\bar e_1,\bar f_2]|_{M_-}
    &=2 \, \frac{v}{h}\,\bar f_3\;,&
    [\bar e_1,\bar f_3]|_{M_-}
    &=-2 \, \frac{v}{h}\,\bar f_2\;,\\
    [\bar x_1,\bar x_2]|_{M_-}
    &=-2 \, \bar x_3\;,&
    [\bar x_2, \bar x_3]|_{M_-}
    &=-2 \, \bar x_1\;,&
    [\bar x_3,\bar x_1]|_{M_-}
    &=-2 \, \bar x_2\;,\\
    [\bar e_1,\bar x_1]|_{M_-}
    &=0\;,& 
    [\bar e_1,\bar x_2]|_{M_-}
    &=2 \, \frac {u}{h}\,\bar x_3\;,&
    [\bar e_1,\bar x_3]|_{M_-}
    &=-2 \, \frac {u}{h}\,\bar x_2\;.
  \end{aligned}
\end{equation*}
All other Lie brackets of these vector fields vanish.
\end{lem}

\begin{proof}
These identities are similar to those obtained in \cite[(4.13)]{Gojems}.

The undecorated vector fields $e_\alpha$ and $f_\alpha$ are the projections of (local) left-invariant vector fields on different factors of $\sph^3 \x \sph^3$, hence satisfy $[e_\alpha, e_\beta] = 2 \, e_\gamma$ and $[f_\alpha, f_\beta] = 2 \, f_\gamma$, for cyclic permutations $(\alpha, \beta, \gamma)$ of $(1,2,3)$, as well as $[e_\alpha, f_\beta] = 0$.  On the other hand, the vector fields $x_\alpha$ are the projections of right-invariant vector fields and thus satisfy $[x_\alpha, x_\beta] = - 2 \, x_\gamma$, for cyclic permutations $(\alpha, \beta, \gamma)$ of $(1,2,3)$.

As their flows are projections of commuting left and right actions respectively, all Lie brackets of $x_\alpha$ with $e_\beta$ or $f_\beta$ will vanish.  Moreover, since $e_0$ commutes with all $e_\alpha$, $f_\alpha$ and $x_\alpha$ on $(\sph^3 \x \sph^3) \bq \Delta Q \x (-1, 1)$, the same is true on $\Mab$ via \eqref{E:RegPart}.

The Lie bracket identities in the lemma now follow from \eqref{E:ONB}.
\end{proof}

Observe that the Lie brackets in Lemma \ref{L:LieBracketsM} are compatible with the isometry $\Psi$ of Remark \ref{R:Isom}.  This was the reason for the `$-$' sign in the definitions of $u_+$ and $v_+$.  

\begin{lem}
\label{L:totgeo}
The Seifert fibration $\pi : (\Mab, g_M) \to (\Bab, g_B)$ has totally geodesic fibres.
\end{lem}

\begin{proof}
It is sufficient to show that $\nabla^{TM}_{\bar f_\alpha} \bar e_\beta$ is always orthogonal to the vector fields $\bar f_\gamma$\,, that is, orthogonal to the fibres, since this implies that the second fundamental form of the fibres vanishes.  From the Koszul formula one has
$$
2 \, g_M(\nabla^{TM}_{\bar f_\alpha} \bar e_\beta, \bar f_\gamma) = g_M([\bar f_\alpha, \bar e_\beta], \bar f_\gamma) - g_M([\bar e_\beta, \bar f_\gamma], \bar f_\alpha) + g_M([\bar f_\gamma, \bar f_\alpha], \bar e_\beta)
$$
and the result now follows from Lemma \ref{L:LieBracketsM}.
\end{proof}

Since the vector fields $\check e_0, \dots, \check e_3$ on $\Bab$ are $\pi$-related to the vector fields $\bar e_0, \dots, \bar e_3$ on $\Mab$, the corresponding identities on $\Bab$ follow immediately.

\begin{lem}
\label{L:LieBracketsB}
Let $B_- := \pi(M_-) \In \Bab$.  Then the vector fields $\check e_0, \dots, \check e_3$ on $\Bab$ satisfy the following Lie bracket identities:
\begin{equation*}
  \begin{aligned}
    [\check e_0,\check e_1]|_{B_-}
    &=-\frac{h'}{h}\,\check e_1 \;,&&&
    [\check e_0,\check e_2]|_{B_-}
    &= 0 \;,\\
    [\check e_0,\check e_3]|_{B_-}
    &= 0 \;,&&&
    [\check e_1,\check e_2]|_{B_-}
    &=\frac2{h}\,\check e_3\;,\\
    [\check e_2,\check e_3]|_{B_-}
    &=2 \, h \, \check e_1 \;,&&&
    [\check e_3,\check e_1]|_{B_-}
    &=\frac2{h}\,\check e_2 \;.
  \end{aligned}
\end{equation*}
\end{lem}

Let $\bar e^0, \dots, \bar e^3, \bar f^1, \dots, \bar f^3$ be the local frame of the cotangent bundle $T^*M$ of $\Mab$ which is dual to $\bar e_0, \dots, \bar e_3, \bar f_1, \dots, \bar f_3$.  In the computations to follow, it will be necessary to understand the exterior differentials of these $1$-forms.  Given $1$-forms $v^{\alpha_1}, \dots, v^{\alpha_k}$, the shorthand $v^{\, \alpha_1\dots \alpha_k}$ will be used to denote $v^{\alpha_1} \wedge \cdots \wedge v^{\alpha_k}$.  Finally, for $(\alpha, \beta, \gamma)$ a cyclic permutation of $(1,2,3)$, let 
\begin{align}
\begin{split}
\label{E:xDual}
\bar x^\alpha &:= \vphi_{\alpha 1} \bar f^1 + \vphi_{\alpha 2} \bar f^2 + \vphi_{\alpha 3} \bar f^3 \,,\\
\bar x^{\alpha \beta} &:= \vphi_{\gamma 1} \bar f^{23} - \vphi_{\gamma 2} \bar f^{13} + \vphi_{\gamma 3} \bar f^{12} \,, 
\end{split}
\end{align}
be the forms dual to $\bar x_\alpha$ and $\bar x_\alpha \wedge \bar x_\beta$ respectively.

\begin{lem}
\label{L:ExtDiff}
The exterior differentials of a function $y := y \circ \tau : M_- \to \R$ and of the $1$-forms $\bar e^0, \dots, \bar e^3, \bar f^1, \dots, \bar f^3$ are given on $M_-$ by 
\begin{equation*}
  \begin{aligned}
    d y &= y' \bar e^0 \,, &   
    d\bar e^0 &=0\,,\\
    d\bar e^1 &= \frac{h'}{h}\,\bar e^{01}-2h\,\bar e^{23}\,,&
     d\bar f^1 &= \frac{u'\vphi_{11}-v'}{h} \bar e^{01}
	-2(u \, \vphi_{11}-v)\,\bar e^{23} - 2 \bar f^{23}\,,  \\
    d\bar e^2 &= \frac{2}{h}\,\bar e^{13}\,,&
    d\bar f^2 &= \frac{u'\vphi_{12}}{h}\,\bar e^{01}-2u \, \vphi_{12}\,\bar e^{23}
	+2\frac{v}{h}\,\bar e^1\bar f^3+2\bar f^{13}\,,  \\
    d\bar e^3 &= -\frac{2}{h}\,\bar e^{12}\,,&
    d\bar f^3 &= \frac{u'\vphi_{13}}{h}\,\bar e^{01}-2u \, \vphi_{13}\bar e^{23}
	-2\frac {v}{h}\,\bar e^1\bar f^2 - 2\bar f^{12} \,.  
  \end{aligned}
\end{equation*}
\end{lem}

\begin{proof}
The expressions for the exterior differentials follow from the Cartan formula together with Lemma \ref{L:LieBracketsM}, the relation $g_M(\bar x_1, \bar f_\beta) = \vphi_{1 \beta}$ and the derivatives \eqref{E:phiderivs}. 
\end{proof}

For the computation of the adiabatic limit of the $\eta$-invariants of the spin-Dirac operator $\D$ and the odd signature operator $\B$, it is not necessary to determine the curvature of the full Levi-Civita connection $\nabla^{TM}$ of $(\Mab, g_M)$.  Indeed, one need only compute the Chern-Weil forms \eqref{E:CWforms} of the Levi-Civita connection $\nabla^{TB}$ of $(\Bab, g_B)$, and of two connections $\nabla^W$ and $\nabla^{\V}$ related to the fibres of the Seifert fibration.

\begin{lem}
\label{L:TBforms}
The Pontrjagin and Euler forms of $TB$ with respect to the Levi-Civita connection $\nabla^{TB}$ of $(\Bab,g_B)$ are given by
\begin{align*}
    p_1\bigl(TB,\nabla^{TB}\bigr)
    &=\frac{1}{\pi^2}\,\biggl(\frac{h' h''}{h}+4h' h^2 - 4 h' \biggr)\,\check e^{0123}\;,\\
    e\bigl(TB,\nabla^{TB}\bigr)
    &=\frac1{4\pi^2}\biggl(6h^{\prime 2}+3h'' h - \frac{4h''}{h} \biggr)\,\check e^{0123}\;.
\end{align*}
Moreover, it follows that
$$
    \int_B p_1\bigl(TB,\nabla^{TB}\bigr) = \frac2{b_1^2} - \frac2{a_1^2} \ \ \text{ and } \
    \int_B e\bigl(TB,\nabla^{TB}\bigr) = \frac1{|a_1|}+\frac1{|b_1|}\,.
$$
\end{lem}

\begin{proof}
The orbifold $(\Bab, g_B)$ is, up to a slightly different choice of the function $h$, isometric to the one considered in \cite[Section 4.c]{Gojems}, with $p_-=a_1$ and $p_+=b_1$.  Therefore, via \cite[(4.16)]{Gojems} the curvature $2$-form $\Omega^{TB}$ of $TB$ is given, with respect to the orthonormal basis $\bar e_0, \dots, \bar e_3$, by
\beq
\label{E:TBcurv}
\Omega^{TB}=
\scalebox{0.795}{$
\bpm
0 & -\frac{h''}{h} \check e^{01} + 2 h' \check e^{23} & h' \check e^{13} & - h' \check e^{12} \\
\frac{h''}{h} \check e^{01} - 2 h' \check e^{23} & 0 & - h' \check e^{03} + h^2 \check e^{12} &   h' \check e^{02} + h^2 \check e^{13} \\
- h' \check e^{13} & h' \check e^{03} - h^2 \check e^{12} & 0 & 2 h' \check e^{01} + (4 - 3 h^2) \check e^{23} \\
h' \check e^{12} & - h' \check e^{02} - h^2 \check e^{13} & - 2 h' \check e^{01} - (4 - 3 h^2) \check e^{23} & 0
\epm
$} .
\eeq
Together with the isometry (analogous to) $\Psi$ in \eqref{E:Mirror}, it follows that the Euler and Pontrjagin forms have been determined in \cite[(4.17)]{Gojems}.  The calculation of the integrals now follows as in \cite[(4.18)]{Gojems}.
\end{proof}

Consider now the Seifert fibration $\pi : (\Mab, g_M) \to (\Bab, g_B)$ as an orbi-bundle with structure group $\SO(4)$.  Associated to the vertical bundle $\V = \ker (d \pi)$ there is a fibre-bundle connection $1$-form $\omega^\pi \in \Hom(TM, \V)$ which acts as the identity on $\V$ and is uniquely determined by the horizontal bundle $\H = \ker (\omega^\pi)$.  Recall that $\H = \Span\{\bar e_0, \bar e_1, \bar e_2, \bar e_3\}$.  The following lemma will prove useful when computing the contribution of the twisted sectors $\Lambda B \backslash B$ to the adiabatic-limit formulae in Theorem \ref{T:AdiaLimD}.

\begin{lem}
\label{L:Omegapi}
The curvature $2$-form $\Omega^\pi$ associated to $\omega^\pi$ is given on $M_-$ by
$$
  \Omega^\pi|_{M_-}
  =\biggl(\frac{u'}h\,\bar e^{01}-2u\,\bar e^{23}\biggr)\,\bar x_1
  -\biggl(\frac{v'}h\,\bar e^{01}-2v\,\bar e^{23}\biggr)\,\bar f_1 \,.
$$
In particular, $\Omega^\pi|_{M_-}$ is smooth at $\tau^{-1}\{-1\}$ and the two summands of $\Omega^\pi|_{M_-}$ correspond to elements of the two summands of the Lie algebra $\so(4) \cong \so(3) \oplus \so(3)$.
\end{lem}

\begin{proof}
Let $\pr_\H : TM \to \H$ denote orthogonal projection.  The curvature $2$-form $\Omega^\pi = (d \omega^\pi) \circ \pr_\H$ is given by twice the O'Neill tensor of $\pi$, namely, if $X$ and $Y$ are vector fields on $\Mab$, then $\Omega^\pi (X, Y) = -[X^\H, Y^\H]^\V$.  The desired expression for $\Omega^\pi|_{M_-}$ now follows from Lemma \ref{L:LieBracketsM}, while the smoothness at $\tau^{-1}\{-1\}$ is a result of the vanishing of $u'$ and $v'$ on $\tau^{-1}(-1, \ve - 1)$.
\end{proof}

By taking the cone over the fibres of $\pi : \Mab \to \Bab$, one obtains a vector orbi-bundle $W\to B$ of rank~$4$, that is, $W = \Pab \x_{\sph^3 \x \sph^3} \HH$, where the action on the fibre is the usual one (cf.\ Lemma \ref{L:orbi}): 
$$
(\sph^3 \x \sph^3) \x \HH \to \HH \,:\, ((y_1, y_2), q) \mapsto y_1 \, q \, \bar y_2 \,.
$$ 
This action is effectively an $\SO(4)$ action and is defined in such a way that the vector fields $\bar f_\alpha$ (respectively, $\bar x_\alpha$) on $\Mab$ correspond to right (respectively, left) multiplication by $\alpha$ on $\HH$. 

In particular, for $\alpha = i$ and the identification of $\HH$ with $\C^2$ via $q = z + j w \mapsto (z,w)$, left and right multiplication on $\HH$ are given by the elements
\beq
\label{E:SO4elts}
L_{i} = \bpm
0 & -1 & 0 & 0 \\
1 & 0 & 0 & 0\\
0 & 0 & 0 & -1\\
0 & 0 & 1 & 0
\epm \,, 
\quad
R_{i} = \bpm
0 & -1 & 0 & 0 \\
1 & 0 & 0 & 0\\
0 & 0 & 0 & 1\\
0 & 0 & -1 & 0
\epm 
\in \so(4) \,.
\eeq

This orientation for $\HH$ has been chosen to agree that of \cite[(4.20), (4.21)]{Gojems}, where the case~$a_3 = q_-$, $b_3 = q_+$ and~$a_2=b_2=0$ was considered.  

\begin{lem}
\label{L:EulerPontW}
If $\nabla^W$ denotes the connection on $W$ induced by $\omega^\pi$, then the smooth Euler and Pontrjagin forms of $(W, \nabla^W)$ are given by
\begin{align*}
e(W,\nabla^W)
    &= \frac{u'u-v'v}{\pi^2h}\,\check e^{0123} \,, \\[1mm]
\frac{p_1}2 \bigl( W,\nabla^W \bigr)
    &=\frac{uu'+vv'}{\pi^2 h}\,\check e^{0123}\,.
\end{align*}
The corresponding Euler and Pontrjagin numbers are
\begin{align*}
\int_B e(W,\nabla^W)
    &=  \frac{1}{8 a_1^2 b_1^2} \det \bpm a_1^2 & b_1^2 \\ a_2^2 - a_3^2 & b_2^2 - b_3^2 \epm ,\\
\int_B \frac{p_1}2 \bigl( W,\nabla^W \bigr)
    &= \frac{1}{8 a_1^2 b_1^2} \det \bpm a_1^2 & b_1^2 \\ a_2^2 + a_3^2 & b_2^2 + b_3^2 \epm .
\end{align*}
\end{lem}

\begin{proof}
By Lemma \ref{L:Omegapi} and \eqref{E:SO4elts}, the curvature $R^W \in \Omega^2(B; \End(W))$ of $\nabla^W$ is given by
$$
R^W|_{B_-} = \biggl(\frac{u'}h\,\check e^{01}-2u\,\check e^{23}\biggr)\, L_i
  -\biggl(\frac{v'}h\,\check e^{01}-2v\,\check e^{23}\biggr)\, R_i \,.
$$
This is clearly smooth at $\tau^{-1}\{-1\}$, since $u'$ and $v'$ vanish on $\tau^{-1}(-1, \ve - 1)$.

Given the isometry $\Psi$ of \eqref{E:Mirror}, the definition $e(W, \nabla^W) = \frac{1}{4 \pi^2} \Pf(R^W)$ yields the Euler form.  The expression for the Euler number follows directly from 
$$
\int_B e(W,\nabla^W) =  \int_{-1}^1\frac{uu'-vv'}4 \, dt \,,
$$  
which derives from the fact that the leaf $\sph^3/Q \cong \tau^{-1}\{t\} \In \Bab$ has volume $h \, \Vol(\sph^3)/8 = \frac{\pi^2 h}4$ with respect to $g_B$, see~\cite[Section~4.c]{Gojems}.

To compute the half-Pontrjagin form and number of~$(W, \nabla^W)$, recall that the elements~$\bar x_1$ and~$\bar f_1$ act on $\HH$ via $L_i$ and $R_i$ respectively.  As both square to $- \Id$ while, on the other hand, the product of the two is trace free, one obtains the desired expressions from $\frac{p_1}2 \bigl( W,\nabla^W \bigr) =\frac1{16\pi^2} \tr \bigl((R^W)^2\bigr)$ and
$$
\int_B \frac{p_1}2 \bigl( W,\nabla^W \bigr)
    =\int_{-1}^1\frac{uu'+vv'}4 \, dt \,.
$$
\end{proof}

The Pontrjagin form $p_1(TB, \nabla^{TB})$ has been computed in Lemma \ref{L:TBforms}.  In the adiabatic limit \eqref{E:AdLimP1Int} there is a second Pontrjagin form which must also be computed, namely, that of the vertical bundle $\V \to \Mab$.  The compression of the Levi-Civita connection $\nabla^{TM}$ on $(\Mab, g_M)$ to $\V$ yields a connection $\nabla^\V$ on $\V$ defined by $\nabla^\V_X V = (\nabla^{TM}_X V)^\V$ for $V \in \Gamma(\V)$ and $X \in TM$.

\begin{lem}
\label{L:PontV}
The smooth Pontrjagin form $p_1(\V, \nabla^\V)$ is given on $M_-$ by
\begin{align*}
p_1(\V, \nabla^\V)|_{M_-}  &=\frac{uu'+(uv)' \vphi_{11}+vv'}{\pi^2 h}\,\bar e^{0123}\\
    &\qquad
	+\frac{1}{2 \pi^2} \left( \biggl(\frac{u'}{h}\,\bar e^{01} - 2u \,\bar e^{23}\biggr)\,\bar x^{23}
        +\biggl(\frac{v'}{h}\,\bar e^{01} - 2 v \,\bar e^{23} \biggr)  \,\bar f^{23} \right) 
\end{align*}
On $M_+$, $p_1(\V, \nabla^\V)$ is given by replacing $\ul a$ with $\ul b$ and $\vphi_{11}$ with $\vphi_{22}$, and by pulling back via the isometry $\Psi$ of \eqref{E:Mirror}.
\end{lem}

\begin{proof}
With respect to the orthonormal basis $\bar f_1, \bar f_2, \bar f_3$ of $\V$, the connection $1$-form is given by 
$$
\omega^\V = (g_M(\bar f_\alpha, \nabla^\V_{\cdot} \bar f_\beta))_{\alpha, \beta}.
$$  
On $M_- \backslash \tau^{-1}\{-1\}$ it then follows via the Koszul formula and Lemma \ref{L:LieBracketsM} that
$$
\omega^\V = 
\bpm
0& -\bar f^3 & \bar f^2\\
    \bar f^3 & 0 & - \frac{2v}h\,\bar e^1 - \bar f^1\\
    -\bar f^2 &  \frac{2v}h\,\bar e^1 + \bar f^1 & 0
\epm .
$$
By applying Lemma \ref{L:ExtDiff} one derives the curvature $2$-form to be
\begin{align*}
\Omega^\V &= d \omega^\V + \omega^\V \wedge \omega^\V \\
&= \bpm
0 \phantom{\Bigg|} & \stack{\bar f^{12} - \smash{\frac{u' \vphi_{13}}h} \, \bar e^{01}}{+ 2u \vphi_{13} \bar e^{23}}
    &\stack{ \bar f^{13} + \smash{\frac{u' \vphi_{12}}h} \, \bar e^{01}}{- 2u \vphi_{12} \bar e^{23}}\\
    \stack{-\bar f^{12} + \smash{\frac{u' \vphi_{13}}h} \, \bar e^{01}}{- 2u \vphi_{13} \bar e^{23}} & 0 \phantom{\Bigg|} &
    \stack{ \bar f^{23} - \smash{\frac{u' \vphi_{11} + v'}h} \, \bar e^{01}}{+ 2(u \vphi_{11} + v) \, \bar e^{23}}\\
    \stack{-\bar f^{13} - \smash{\frac{u' \vphi_{12}}h} \, \bar e^{01}}{+ 2u \vphi_{12} \bar e^{23}} &
    \stack{-\bar f^{23} + \smash{\frac{u' \vphi_{11} + v'}h} \, \bar e^{01}}{- 2(u \vphi_{11} + v) \, \bar e^{23}} & 0 \phantom{\Bigg|}
\epm 
\end{align*}

Since $u'$ and $v'$ vanish on $\tau^{-1}(-1, \ve -1)$, it is clear that $\Omega^\V$ can be extended smoothly to $\tau^{-1}\{-1\}$, hence that the Pontrjagin form $p_1(\V, \nabla^\V)$ is given on $M_-$ by

\begin{multline*}
  p_1\bigl(\V, \nabla^\V \bigr)|_{M_-}
  =\frac1{8\pi^2}\,\tr\bigl((\Omega^\V)^2\bigr)\\
  \begin{aligned}
    &=\frac1{4\pi^2}\,\Biggl(\biggl(\frac{4uu'}h(\vphi_{13}^2+\vphi_{12}^2+\vphi_{11}^2)+\frac{4(uv)'}h\vphi_{11}+\frac{4vv'}h\biggr)\,\bar e^{0123}\\
    &\kern4em
    +\bar e^{01}\,\biggl(\frac{2u'}h
	\,\bigl(\vphi_{11} \bar f^{23} - \vphi_{12}\bar f^{13} + \vphi_{13} \bar f^{12}\bigr)
	+\frac{2v'}h \, \bar f^{23}\biggr)\\
    &\kern4em\
	- \bar e^{23} \, \Bigl(4u\,\bigl(\vphi_{11} \bar f^{23} - \vphi_{12} \bar f^{13} + \vphi_{13} \bar f^{12} \bigr) + 4v \, \bar f^{23} \Bigr)\Biggr)\\
    &=\frac{uu' + (uv)' \vphi_{11} + vv'}{\pi^2 h} \, \bar e^{0123}\\
    &\qquad
	+\frac{1}{2 \pi^2} \left( \biggl(\frac{u'}{h}\,\bar e^{01} - 2u \,\bar e^{23}\biggr)\,\bar x^{23}
        +\biggl(\frac{v'}{h}\,\bar e^{01} - 2 v \,\bar e^{23} \biggr)  \,\bar f^{23} \right) 
  \end{aligned}
\end{multline*}
as claimed.
\end{proof}

\begin{rem}
\label{R:nmInvs}
Note that $\pi^*W$ is stably isomorphic to the vertical bundle $\V$, since $\pi : \Mab \to \Bab$ is the unit-sphere orbi-bundle associated to the vector orbi-bundle $W \to \Bab$.  Moreover, both the order $|n|$ of $H^4(\Mab; \Z)$ and the number $m$ appearing in the expression for the Eells-Kuiper invariant given in Theorem \ref{T:thmC} can be written in terms of orbifold characteristic numbers.  Indeed, from Lemma \ref{L:EulerPontW} one has 
$$
\int_B e(W, \nabla^W) = \frac{n}{a_1^2 b_1^2} \,,
$$
while, on the other hand,  Lemmas \ref{L:TBforms} and \ref{L:EulerPontW} yield
$$
\int_B \frac{p_1}{2}(TB \oplus W, \nabla^{TB} \oplus \nabla^W) 
= \frac{1}{8 a_1^2 b_1^2} \det \bpm a_1^2 & b_1^2 \\ a_2^2 + a_3^2 + 8 & b_2^2 + b_3^2 + 8 \epm = m \,.
$$ 
Given that both $TB$ and $W$ are orbi-bundles, there is no reason to expect that $\frac{n}{a_1^2 b_1^2}$ and $m$ should be integers.  However, whenever $a_1 = b_1 = 1$, that is, whenever $\pi$ is a classical $\sph^3$-bundle over $\sph^4$, these are integers.
\end{rem}

\medskip
\subsection{Evaluation of the Pontrjagin term}  \hspace*{1mm}\\
\label{SS:AdLimPont}

Recall from \eqref{E:AdLimP1Int} that the adiabatic limit of the Pontrjagin term in \eqref{E:EKeta} is given by
\begin{align*}
\lim_{\ve \to 0} \int_M p_1(TM, \nabla^{TM, \ve}) \wedge \, & \hat p_1(TM, \nabla^{TM, \ve})  \\
\nonumber 
&= \int_M \left(p_1(\V, \nabla^\V) + \pi^* p_1(TB, \nabla^{TB})\right) \\
&\hspace*{20mm} \wedge \left(\hat p_1(\V, \nabla^\V) + \hat p_1(\pi^* TB, \nabla^{TB}) \right) ,
\end{align*}
where
\begin{align*}
d \hat p_1(\V, \nabla^\V) &= p_1(\V, \nabla^\V), \\
d \hat p_1( \pi^* TB, \nabla^{TB}) &= \pi^* p_1(TB, \nabla^{TB}).
\end{align*} 
By Lemmas \ref{L:TBforms} and \ref{L:PontV}, only the $3$-form 
$$
\hat p_1 := \hat p_1(\V, \nabla^\V) + \hat p_1( \pi^* TB, \nabla^{TB})
$$ 
in the integrand remains to be determined.  Given an exact form $\zeta$, a form $\xi$ with $d \xi = \zeta$ will be called a \emph{primitive} of $\zeta$.  

\begin{lem}
\label{L:StabIsom}
On $M_-$ one has the identity
$$
p_1 \bigl(\V, \nabla^\V \bigr)|_{M_-} = \pi^* p_1\bigl(W, \nabla^W \bigr)|_{M_-} + d\xi_- \,,
$$
where
\begin{align*}
\xi_- &:= \frac{1}{4\pi^2} \left(\biggl(\frac{u'}h\,\bar e^{01}-2u\,\bar e^{23}\biggr)\,\bar x^1
  -\biggl(\frac{v'}h\,\bar e^{01}-2v\,\bar e^{23}\biggr) \,\bar f^1 \right) \,. 
\end{align*}
In particular, $\xi_-$ is smooth at $t = -1$ and $\xi_-|_{\tau^{-1}(-\ve, \ve)} \equiv 0$.  After switching the roles of $\ul a$ and $\ul b$, and pulling back by $\Psi$ \eqref{E:Mirror}, one obtains on $M_+$ a similar smooth primitive $\xi_+$ of $p_1 \bigl(\V, \nabla^\V \bigr)|_{M_+} - \pi^* p_1\bigl(W, \nabla^W \bigr)|_{M_+}$.
\end{lem} 

\begin{proof}
The existence of such a form $\xi_-$ follows from Remark \ref{R:nmInvs}, since the Pontrjagin classes of stably isomorphic bundles must agree, hence their representatives differ by an exact form.

In order to compute $d \xi_-$, some further exterior differentials are needed.  Given $d \vphi_{\alpha \beta} (v) = v(\vphi_{\alpha \beta})$, one derives from  \eqref{E:phiderivs} that
 \begin{align*}
    d \vphi_{11}
    &= 2 \vphi_{12} \bar f^3 - 2 \vphi_{13} \bar f^2\,,\\
    d \vphi_{12}
    &= 2 \frac vh \vphi_{13} \bar e^1 + 2 \vphi_{13} \bar f^1 - 2 \vphi_{11} \bar f^3\,,\\
    d \vphi_{13}
    &= - 2 \frac vh \vphi_{12} \bar e^1 + 2 \vphi_{11} \bar f^2 - 2 \vphi_{12} \bar f^1\,,   
\intertext{which, together with Lemma \ref{L:ExtDiff} and \eqref{E:xDual}, yield}
     d \bar x^1
    &= \frac{u'-v'\vphi_{11}} h\, \bar e^{01}
    - 2 (u - v \vphi_{11}) \, \bar e^{23} + 2 \bar x^{23}\,.
  \end{align*}
Lemma \ref{L:ExtDiff} now gives
\begin{align}
    \begin{split}
    \label{E:xiDiff}
      d \biggl(\Bigl( \frac{u'}h \, \bar e^{01} - 2u \, \bar e^{23} \Bigr)
      \, \bar x^1\biggr)
      &=\frac{2(u v)' \vphi_{11} - 4 u u'} h \, \bar e^{0123}
      + 2\Bigl(\frac{u'} h \bar e^{01} - 2u \, \bar e^{23} \Bigr) \bar x^{23}\,,\\
      d \biggl(\Bigl( \frac{v'}h \, \bar e^{01} - 2v \, \bar e^{23} \Bigr)
      \, \bar f^1\biggr)
      &=\frac{4 v v' - 2 (u v)' \vphi_{11}} h \, \bar e^{0123}
      - 2 \Bigl(\frac{v'}h \bar e^{01} - 2v \, \bar e^{23} \Bigr) \bar f^{23}\,,
    \end{split}
\end{align}
which, with Lemmas \ref{L:EulerPontW} and \ref{L:PontV}, yields
\begin{align*}
d \xi_-  &= - \left(\frac{uu' - (uv)' \vphi_{11} + vv'}{\pi^2 h}\right) \bar e^{0123}\\
    &\hspace*{20mm}
	+\frac{1}{2 \pi^2} \left( \biggl(\frac{u'}{h}\,\bar e^{01} - 2u \,\bar e^{23}\biggr)\,\bar x^{23}
        +\biggl(\frac{v'}{h}\,\bar e^{01} - 2 v \,\bar e^{23} \biggr)  \,\bar f^{23} \right)  \\
&= p_1 \bigl(\V, \nabla^\V \bigr)|_{M_-} - \pi^* p_1\bigl(W, \nabla^W \bigr)|_{M_-} 
\end{align*}
as desired.  The smoothness of $\xi_-$ at $t = -1$ now follows from the vanishing of $h''$, $u'$ and $v'$ on $\tau^{-1}(-1, \ve -1)$.
\end{proof}

\begin{lem}
\label{L:P1prim}
The $3$-form
\begin{multline*}
\kappa_- := \xi_- +  \frac{1}{\pi^2}(h^3 - 2h) \bar e^{123} \\
+ \frac{1}{2 \pi^2 h} \left( (h')^2 + 2(u^2 + v^2) - 2\left( \frac{a_2^2 + a_3^2 + 8}{a_1^2} \right) \right) \bar e^{123}
\end{multline*}
on $M_-$ is a smooth primitive of $p_1(\V, \nabla^\V)|_{M_-} + \pi^* p_1(TB, \nabla^{TB})|_{M_-}$, that is,
$$
d \kappa_- = p_1(\V, \nabla^\V)|_{M_-} + \pi^* p_1(TB, \nabla^{TB})|_{M_-} \,.
$$
By swapping $\ul a$ with $\ul b$ and pulling back via the isometry $\Psi$ of \eqref{E:Mirror}, one obtains an analogous $3$-form $\kappa_+$ on $M_+$ which is a smooth primitive of $p_1(\V, \nabla^\V)|_{M_+} + \pi^* p_1(TB, \nabla^{TB})|_{M_+}$.

In particular, on $M_- \cap M_+ = \tau^{-1}(-\ve, \ve)$ one has
$$
\kappa_-|_{\tau^{-1}(-\ve, \ve)} - \kappa_+|_{\tau^{-1}(-\ve, \ve)} = \frac{8 m}{\pi^2} \, \bar e^{123} \,.
$$
\end{lem}

\begin{proof}
The smoothness of $\kappa_-$ along $\tau^{-1} \{-1\}$ is a consequence of the smoothness of the forms $h \, \bar e^1 = g_M(h \, \bar e_1, \cdot )$ and $\bar e^{23}$ at the singular leaf, together with Lemma \ref{L:StabIsom} and the vanishing of the second $\bar e^{123}$ term on $\tau^{-1}(-1, \ve -1)$.

By Lemmas \ref{L:TBforms}, \ref{L:EulerPontW} and \ref{L:StabIsom}, as well as the definitions of $u$ and $v$, one has
\begin{align*}
p_1(\V, \nabla^\V)|_{M_-} &+ \pi^* p_1(TB, \nabla^{TB})|_{M_-}  \\
&= d \xi_- + \pi^*p_1(W, \nabla^W)|_{M_-} + \pi^* p_1(TB, \nabla^{TB})|_{M_-} \\
&= d \xi_- + \left( \frac{(a_2^2 + a_3^2 + 8) h' h''}{8 \pi^2 h} + \frac{4h'(h^2 - 1)}{\pi^2} \right) \bar e^{0123} \\
&= d \xi_- + \frac{1}{\pi^2} \, d \left( (h^3 - 2h) \bar e^{123} \right) \\
&\hspace*{20mm} 
+ d \left(\left(\frac{(a_2^2 + a_3^2 + 8)}{8 \pi^2}\right) \left( \frac{(h')^2}{2h} - \frac{8}{a_1^2 h} \right) \bar e^{123} \right) \\
&= d \kappa_- \,,
\end{align*}
where the second-last equality follows by applying Lemma \ref{L:ExtDiff} to obtain $d \bar e^{123} = \frac{h'}{h} \bar e^{0123}$.
\end{proof}

As a consequence of Lemma \ref{L:P1prim}, to obtain a smooth, global primitive $\hat p_1$ for $p_1(\V, \nabla^\V) + \pi^* p_1(TB, \nabla^{TB})$ it suffices to find closed $3$-forms $\nu_-$ and $\nu_+$ on $M_-$ and $M_+$ respectively, such that $(\kappa_- + \nu_-) - (\kappa_+ + \nu_+ ) = 0$ on $\tau^{-1}(-\ve, \ve)$.

\begin{lem}
\label{L:closedform}
The $3$-form 
\begin{align*}
\nu_- &:= \frac{a_1^2 b_1^2 m}{\pi^2 n} \left(
\bar f^{123} - \frac{1}{h} \left( (u^2 - v^2) - \frac{a_2^2 - a_3^2}{a_1^2 }\right) \bar e^{123} \right. \\
&\hspace*{25mm}
\left. - \frac{1}{2} \left(\biggl(\frac{u'}h\,\bar e^{01}-2u\,\bar e^{23}\biggr)\,\bar x^1
  + \biggl(\frac{v'}h\,\bar e^{01}-2v\,\bar e^{23}\biggr) \,\bar f^1 \right)
\right)
\end{align*}
on $M_-$ is smooth and closed.  Moreover, if $\nu_+$ is the corresponding closed $3$-form on $M_+$ obtained by swapping $\ul a$ and $\ul b$ and pulling back via the isometry $\Psi$ of \eqref{E:Mirror}, then
$$
\nu_-|_{\tau^{-1}(-\ve, \ve)} - \nu_+|_{\tau^{-1}(-\ve, \ve)} = - \frac{8m}{\pi^2} \bar e^{123} \,.
$$
\end{lem}

\begin{proof}
The smoothness of $\nu_-$ at $\tau^{-1}\{-1\}$ is clear, since $u'$, $v'$ and the coefficient of $\bar e^{123}$ all vanish identically on $\tau^{-1}(-1, \ve - 1)$.

From Lemma \ref{L:ExtDiff} it can be shown that
\beq
\label{E:df123}
d \bar f^{123} = \biggl(\frac{u'}h\,\bar e^{01}-2u\,\bar e^{23}\biggr)\,\bar x^{23}
  -\biggl(\frac{v'}h\,\bar e^{01}-2v\,\bar e^{23}\biggr) \,\bar f^{23} \,.
\eeq
Together with \eqref{E:xiDiff} and the identity $d \bar e^{123} = \frac{h'}{h} \bar e^{0123}$, it is now easy to confirm that $d \nu_- = 0$.
\end{proof}

\begin{prop}
\label{P:P1prim}
The $3$-form 
$$
\hat p_1 := 
\begin{cases}
\kappa_- + \nu_- \,, & \text{ on } M_- \,,\\
\kappa_+ + \nu_+ \,, & \text{ on } M_+ \,.
\end{cases}
$$
is a smooth, global primitive for $p_1(\V, \nabla^\V) + \pi^* p_1(TB, \nabla^{TB})$.
\end{prop}

\begin{proof}
The result follows immediately from Lemmas \ref{L:P1prim} and \ref{L:closedform}.  In particular, on the intersection $M_- \cap M_+ = \tau^{-1}(-\ve, \ve)$ one has 
$$
(\kappa_- + \nu_-)|_{\tau^{-1}(-\ve, \ve)} - (\kappa_+ + \nu_+ )|_{\tau^{-1}(-\ve, \ve)} = 0 \,.
$$
\end{proof}

It is finally possible to evaluate the Pontrjagin term in the formula for the Eells-Kuiper invariant given in Corollary \ref{C:AdLimEK}.

\begin{thm}
\label{T:PontTerm}
If $n \neq 0$ then, with respect to the metric $g_M$ on $\Mab$ given in Proposition \ref{P:metric}, the adiabatic limit of $\int_M p_1(TM, \nabla^{TM}) \wedge \hat p_1(TM, \nabla^{TM})$ is given by
\begin{align*}
&\frac{1}{2^7 \! \cdot \! 7}  \int_M \left(p_1(\V, \nabla^\V) + \pi^* p_1(TB, \nabla^{TB})\right) 
\wedge \left(\hat p_1(\V, \nabla^\V) + \hat p_1(\pi^* TB, \nabla^{TB}) \right) \\[1mm]
&\hspace*{25mm} 
= \frac{1}{2^7 \! \cdot \! 7}\left( \frac{4 a_1^2 b_1^2 m^2}{n} - \frac{n}{a_1^2 b_1^2} \right) .
\end{align*}
\end{thm}

\begin{proof}
By Lemmas \ref{L:TBforms} and \ref{L:PontV}, together with Proposition \ref{P:P1prim}, the integrand is given on $\tau^{-1}[-1,0] \In M_-$ by
\begin{align*}
d \kappa_- \wedge (\kappa_- + \nu_-) &= d ((\kappa_- - \xi_-) + \xi_-) \wedge ((\kappa_- - \xi_-) + \xi_- + \nu_-) \\
&= d (\kappa_- - \xi_-) \wedge (\kappa_- + \nu_-) + d \xi_- \wedge \xi_-  \\
&\hspace*{13mm}
+ d \xi_- \wedge (\kappa_- - \xi_- + \nu_-)  \\
&= d (\kappa_- - \xi_-) \wedge (\kappa_- + \nu_-) + d \xi_- \wedge \xi_-  \\
&\hspace*{13mm}
+ d(\xi_- \wedge (\kappa_- - \xi_- + \nu_-)) + \xi_- \wedge d (\kappa_- - \xi_- + \nu_-) .
\end{align*}
Given that $\nu_-$ is closed and
$$
d (\kappa_- - \xi_-) = d (\kappa_- - \xi_- + \nu_-) = \pi^*p_1(W, \nabla^W)|_{M_-} + \pi^* p_1(TB, \nabla^{TB})|_{M_-}
$$ 
involves only $\bar e^{0123}$ terms, it follows from Lemma \ref{L:StabIsom} that 
\beq
\label{E:lastterm}
\xi_- \wedge d (\kappa_- - \xi_- + \nu_-) = \xi_- \wedge d (\kappa_- - \xi_-) = 0\,,
\eeq
and from Lemmas \ref{L:TBforms}, \ref{L:EulerPontW}, \ref{L:P1prim} and \ref{L:closedform} that
\begin{align}
\nonumber
 d (\kappa_- - \xi_-) &\wedge (\kappa_- + \nu_-) \\
\nonumber
 &= \frac{a_1^2 b_1^2 m}{\pi^2 n} \left( \pi^*p_1(W, \nabla^W)|_{M_-} + \pi^* p_1(TB, \nabla^{TB})|_{M_-} \right) \wedge \bar f^{123}\\
\label{E:1stterm}
 &= \frac{a_1^2 b_1^2 m}{\pi^4 n} \left(\frac{h' h''}{h}+4h' h^2 - 4 h' + \frac{2(uu' + vv')}{h} \right) \bar e^{0123} \bar f^{123} \,.
\end{align}
On the other hand, from \eqref{E:xDual} it follows that $\bar x^{123} = \bar f^{123}$ and $\bar f^1 \bar x^{23} = \bar x^1 \bar f^{23} = \vphi_{11} \bar f^{123}$.  Therefore, from Lemma \ref{L:StabIsom} one derives
\beq
\label{E:2ndterm}
d \xi_- \wedge \xi_- = - \frac{u u' - v v'}{2 \pi^4 h} \bar e^{0123} \bar f^{123} \,.
\eeq
Finally, since $\xi_-|_{\tau^{-1}(-\ve, \ve)} \equiv 0$, it follows from Stokes' Theorem that
\beq
\label{E:3rdterm}
\int_{\tau^{-1}[-1,0]}  d(\xi_- \wedge (\kappa_- - \xi_- + \nu_-)) = 0 \,.
\eeq
Together with the fact that, with respect to the metric $g_M$, a leaf $\tau^{-1}\{t\} \In \Mab$ has volume $\left(\frac{\pi^2 h}{4}\right) (2 \pi^2) = \frac{\pi^4 h}{2}$ \cite[(4.33)]{Gojems}, equations \eqref{E:lastterm}, \eqref{E:1stterm}, \eqref{E:2ndterm} and \eqref{E:3rdterm} yield
\begin{align*}
&\int_{\tau^{-1}[-1,0]} \left(p_1(\V, \nabla^\V) + \pi^* p_1(TB, \nabla^{TB})\right) 
\wedge \left(\hat p_1(\V, \nabla^\V) + \hat p_1(\pi^* TB, \nabla^{TB}) \right) \\[1mm]
&\hspace*{7mm} 
= \int_{-1}^0 \frac{a_1^2 b_1^2 m}{2 n} \left(\frac{1}{2}((h')^2)' + (h^4)' - 2 (h^2)' + (u^2 + v^2)' \right) - \frac{(u^2 - v^2)'}{8} \, dt \\
&\hspace*{7mm} 
= \left( \frac{a_1^2 b_1^2 m}{2 n} \left(\frac{1}{2}(h')^2 + (h^4)' - 2 h^2 + u^2 + v^2 \right) - \frac{u^2 - v^2}{8} \right) \bigg|_{-1}^0 \\
&\hspace*{7mm} 
= C_0 - \frac{a_1^2 b_1^2 m}{2 n} \left(\frac{a_2^2 + a_3^2 + 8}{a_1^2} \right) + \frac{a_2^2 - a_3^2}{8 a_1^2}  \,,
\end{align*}
where $C_0$ denotes the $t = 0$ boundary term.  Similarly, bearing in mind that the isometry $\Psi$ \eqref{E:Mirror} is orientation reversing, on $\tau^{-1}(0,1] \In M_+$ one obtains
\begin{align*}
&\int_{\tau^{-1}(0, 1]} \left(p_1(\V, \nabla^\V) + \pi^* p_1(TB, \nabla^{TB})\right) 
\wedge \left(\hat p_1(\V, \nabla^\V) + \hat p_1(\pi^* TB, \nabla^{TB}) \right) \\[1mm]
&\hspace*{20mm} 
= \frac{a_1^2 b_1^2 m}{2 n} \left(\frac{b_2^2 + b_3^2 + 8}{b_1^2} \right) - \frac{b_2^2 - b_3^2}{8 b_1^2} - C_0  \,.
\end{align*}
The result now follows from the definitions of $m$ and $n$ by combining the integrals over $\tau^{-1}[-1,0]$ and $\tau^{-1}(0,1]$.
\end{proof}

\medskip
\subsection{The contribution of the $\eta$-forms}\hspace*{1mm}\\
\label{SS:etaForm}

Given the computations in Subsection \ref{SS:ChernWeil},some further terms can be computed in the expression for the Eells-Kuiper invariant given by the adiabatic-limit formula of Corollary \ref{C:AdLimEK}.

\begin{thm}
\label{T:etaFormTerm}
If $n \neq 0$ then, with respect to the metric $g_M$ on $\Mab$ given in Proposition \ref{P:metric}, it follows that
\begin{align*}
\frac{1}{2} \int_{\Lambda B}& \hat A_{\Lambda B}(TB, \nabla^{TB})\, 2 \, \eta_{\Lambda B}(\D_{\sph^3}) 
 + \frac{1}{2^5 \! \cdot \! 7} \int_{\Lambda B} \hat L_{\Lambda B}(TB, \nabla^{TB})\, 2 \, \eta_{\Lambda B}(\B_{\sph^3}) \\
&= - \frac{1}{2^7 \cdot 7} \left(\frac{n}{a_1^2 b_1^2}\right) - D(\ul a) + D(\ul b) \,.
\end{align*}
\end{thm}

\begin{proof}
Recall from Corollary \ref{C:InerOrb} that the inertia orbifold $\Lambda B$ associated to $\Bab$ is described by the disjoint union
$$
\Lambda B = \Bab \sqcup \left(\sph^2_- \x \left\{1, \dots, \frac{|a_1|-1}{2}\right\} \right) \sqcup \left(\sph^2_+ \x \left\{1, \dots, \frac{|b_1|-1}{2}\right\} \right) .
$$
As such, the integrals in the statement can be performed over each of the connected components separately.  On $\Bab \In \Lambda B$ the integrals can be computed using Theorem 3.9 of \cite{Go} (as was done in \cite[Prop.\ 4.2]{Gojems}), which, following Lemma \ref{L:EulerPontW} and Remark \ref{R:nmInvs}, yields
\begin{align*}
\frac{1}{2} \int_{B}& \hat A_{\Lambda B}(TB, \nabla^{TB})\, 2 \, \eta_{\Lambda B}(\D_{\sph^3}) 
 + \frac{1}{2^5 \! \cdot \! 7} \int_{B} \hat L_{\Lambda B}(TB, \nabla^{TB})\, 2 \, \eta_{\Lambda B}(\B_{\sph^3}) \\
&= - \frac{1}{2^7 \cdot 7} \int_B e(W, \nabla^W) = - \frac{1}{2^7 \cdot 7} \left(\frac{n}{a_1^2 b_1^2}\right) .
\end{align*}

It remains, therefore, only to show that the contribution of $\Lambda B \backslash \Bab$ consists of the generalised Dedekind sums $-D(\ul a)$ and $+D(\ul b)$.  In order to do this, it is necessary to determine some equivariant characteristic numbers and the equivariant $\eta$-forms for the pullback of the Seifert fibration $\pi : \Mab \to \Bab$ to the double covers $\sph^2_\pm$ of the components $\RP^2_\pm$ of the singular locus of $\Bab$.  As there is an analogous orientation-reversing isometry on $\Bab$ to that given on $\Mab$ by $\Psi$ in \eqref{E:Mirror}, only the computations for $\RP^2_-$ need to be carried out explicitly.  Observe first that Lemma \ref{L:LieBracketsB} yields $\nabla_{\check e_2} \check e_2 = \nabla_{\check e_3} \check e_3 = 0$ and $\nabla_{\check e_2} \check e_3 = - \nabla_{\check e_3} \check e_2 = h \check e_1$, from which it follows that $\RP^2_-$ is totally geodesic in $\Bab$.

Following the notation of Corollary \ref{C:InerOrb} and recalling the discussion preceding \eqref{E:ChChar}, let $(b, [\gamma_-^s])$ be a point in $\Lambda B \backslash \Bab$, let $\mc N_- \to \RP^2_-$ be the normal bundle of $\RP^2_- \In \Bab$, and let $\tilde{\mc N}_- \to \sph^2_-$ denote the pullback of $\mc N_-$ to $\sph^2_-$. Since $\Bab$ is oriented by $\check e^{0123}$ and the twisted sector $\sph^2_-$ is locally oriented by $\check e^{23}$, the orientation on $\tilde{\mc N}_-$ is given (in a limiting sense) by $\check e^{01}$.  The bundle $\tilde{\mc N}_-$ carries a natural spin structure with an associated spinor bundle $\mc S(\tilde{\mc N}_-)$.  

By Lemma \ref{L:orbi}, in an orbifold chart $V$ the elements $\gamma_-^s$, $s \in \{1, \dots, \frac{|a_1|-1}2\}$, of the isotropy group $\Z_{|a_1|}$ act on $\tilde{\mc N}_-$ via multiplication by $e^{8\pi i s/a_1}\in \sph^1 \cong \SO(2)$.  As $\Z_{|a_1|}$ is an odd cyclic group, this action has a unique lift to $\Spin(2)$, represented by
	\beq
	\label{E:SpinAct}
	\tilde\gamma_-^s = e^{4\pi i s/ a_1}\in \sph^1 \cong \Spin(2)\,.
	\eeq

Similarly to the arguments employed for \cite[(4.22), (4.23)]{Gojems}, the curvature $2$-forms for $\tilde{\mc N}_-$ and $T \sph^2_-$ can be computed in an orbifold chart by considering the upper and lower $(2 \x2)$-blocks of the curvature \eqref{E:TBcurv} and taking limits as $t \to -1$.  It then follows that the corresponding curvatures are given by 
$$
R^{\tilde{\mc N}_-} = -\frac{8 i}{|a_1|} \, \check e^{23} \ \  \text{ and } \ \ R^{T\sph^2_-} = - 4 i \, \check e^{23} \,.
$$ 
In particular, using the (non-standard) convention from \cite{Gojems} that the Clifford actions of $c(\check e_0) c(\check e_1)$ on $\mc S^\pm (\tilde{\mc N}_-)$ and of $c(\check e_2) c(\check e_3)$ on $\mc S^\pm (T \sph^2_-)$ are both given by $\pm i$, one derives (from, for example, \cite[Sec.\ II, Thm.\ 4.15]{LM}) that the curvature of the summands of the spinor bundle is given by
$$
R^{\mc S^\pm (\tilde{\mc N}_-)} = \mp \frac{4i}{|a_1|} \, \check e^{23} 
\ \ \text{ and } \ \ 
R^{\mc S^\pm (T \sph^2_-)} = \mp 2i \, \check e^{23} \,.
$$
On the other hand, by \eqref{E:compatibility} and \eqref{E:SpinAct} the action of $\tilde\gamma_-^s$ on $\mc  S^\pm (\tilde{\mc N}_-)$ is given by 
$$
\gamma_-^s|_{\mc S^\pm(\tilde{\mc N_-})} 
= \exp \left( \frac{a_1}{|a_1|} \frac{4\pi s}{a_1}c(\check e_0) c(\check e_1) \right) \bigg|_{\mc S^\pm(\tilde{\mc N_-})} 
= \exp \left(\pm \frac{a_1}{|a_1|} \frac{4\pi i s}{a_1} \right)
$$ 
respectively.  Therefore, one deduces that
$$
\tilde \gamma_-^s \exp \left(- \frac{R^{\mc S^\pm (\tilde{\mc N}_-)}}{2 \pi i} \right) 
= \exp \left( \pm \frac{a_1}{|a_1|} \frac{4 i}{a_1} \left( \pi s +  \frac{\check e^{23}}{2 \pi i} \right) \right) ,
$$
which in turn, via \eqref{E:ChChar}, yields the equivariant Chern character 
\begin{align}
\nonumber
\ch_{\tilde \gamma_-^s}(\mc S^+ &(\tilde{\mc N}_-) - \mc S^- (\tilde{\mc N}_-), \nabla^{\mc S(\tilde{\mc N}_-)}) \\
\label{E:ChCharSN}
&=  \exp \left(\frac{a_1}{|a_1|} \frac{4 i}{a_1} \left( \pi s + \frac{\check e^{23}}{2 \pi i} \right) \right) 
- \exp \left(- \frac{a_1}{|a_1|} \frac{4 i}{a_1} \left( \pi s + \frac{\check e^{23}}{2 \pi i} \right) \right) \\
\nonumber
&= 2 i \frac{a_1}{|a_1|}  \sin \left( \frac{4}{a_1} \left( \pi s + \frac{\check e^{23}}{2 \pi i} \right) \right).  
\end{align}
From \eqref{E:eqAhat} and \eqref{E:orbAhat}, and given that $\hat A(T \sph^2_-, \nabla^{\sph^2_-}) = 1$ since it has degree $\equiv 0$ mod $4$, it can now be concluded that the orbifold $\hat A$-form on $\sph^2_- \x \{s\} \In \Lambda B$ is given by
\beq
  \label{E:AhatCont}
    \hat A_{\Lambda B} \bigl( TB, \nabla^{TB} \bigr)
    =  - \frac{1}{a_1 \cdot 2i \sin \bigl(\frac {4}{a_1}
		\bigl(\pi s +  \frac{\check e^{23}}{2 \pi i} \bigr)\bigr)} \,.
\eeq

Similarly, since the action of $\gamma_-^s$ tangential to $\sph^2_-$ is trivial, one derives
\begin{align*}
\ch_{\tilde \gamma_-^s}(\mc S^+ &(TB) + \mc S^- (TB), \nabla^{\mc S(TB)}) \\ 
&=  2 \left( \exp \left( \frac{a_1}{|a_1|}  \frac{4 i}{a_1} \left( \pi s + \frac{\check e^{23}}{2 \pi i} \right) \right) 
+ \exp \left(- \frac{a_1}{|a_1|}  \frac{4 i}{a_1} \left( \pi s + \frac{\check e^{23}}{2 \pi i} \right) \right) \right) \\
&= 4 \cos \left(\frac {4}{a_1}
		\left(\pi s + \frac{\check e^{23}}{2 \pi i} \right)\right) ,  
\end{align*}
where the additional factor of $2$ is a consequence of $T \sph^2_-$ being a rank-$2$ bundle.  From \eqref{E:orbLhat} one concludes that the orbifold $\hat L$-form on $\sph^2_- \x \{s\} \In \Lambda B$  is given by 
\begin{align}
\begin{split}
\label{E:LhatCont}
\hat L_{\Lambda B}(TB,\nabla^{TB}) 
&= \hat A_{\Lambda B}(TB,\nabla^{TB}) \ch_{\Lambda B}(\mc S^+ (TB) + \mc S^- (TB), \nabla^{\mc S(TB)}) \\
    &=  \frac{2i}{a_1}\, \cot \biggl( \frac{4}{a_1} \biggl( \pi s
		+ \frac{\check e^{23}}{2 \pi i} \biggr)\biggr) \,.
\end{split}
\end{align}

To compute the equivariant $\eta$-forms of $\Mab|_{\RP^2_-}\to \RP^2_-$, recall from Lemma \ref{L:orbi} that $\gamma_-^s$ acts on the fibre $\sph^3$ via 
\beq
\label{E:FibAct}
(\gamma_-^s, q) \mapsto \gamma_-^{s a_2} q \, \bar \gamma_-^{s a_3} = e^{2\pi i(a_2 - a_3) s/a_1} z + j e^{- 2\pi i(a_2 + a_3) s/a_1}  w \,,
\eeq
where $q = z + j w \in \sph^3$.  This action clearly extends to the fibres of the associated rank-$4$ vector orbi-bundle $W \to B$.  Furthermore, by the proof of Lemma \ref{L:EulerPontW}, the curvature of $W$ at $\RP^2_-$ is given by 
$$
R^W_- := R^W|_{\RP^2_-} = - \frac{2 a_2}{a_1} \, \check e^{23} \, L_i + \frac{2 a_3}{a_1} \, \check e^{23} \, R_i \,,
$$
which then acts on the fibres of $W$ via
\begin{align}
\begin{split}
\label{E:FibCurv}
\exp \left(- \frac{R^W_-}{2 \pi i} \right) & \cdot (z + j w) \\
&= \exp \left(\frac{-2(a_3 - a_2)}{a_1} \frac{\check e^{23}}{2 \pi i} \right)\! z 
+ \exp \left(\frac{-2(a_2 + a_3)}{a_1} \frac{\check e^{23}}{2 \pi i} \right)\! j w \,.
\end{split}
\end{align}

On the other hand, given that the fibres of $\pi$ have positive scalar curvature, hence that the kernel of $\D_{\sph^3}$ is trivial, explicit formulae for the equivariant $\eta$-invariants $\eta_{\gamma_-^s \exp(-R^W_-/2\pi i)}(\D_{\sph^3})$ and $\eta_{\gamma_-^s \exp(-R^W_-/2\pi i)}(\B_{\sph^3})$ of the (untwisted) spin-Dirac operator $\D_{\sph^3}$ and the odd signature operator $\B_{\sph^3}$ can be found in \cite[Eqns.\ (5), (11), (14)]{HR} and \cite[proof of Prop.~2.12]{APS} respectively, as well as in \cite{Go1}.  Therefore, in analogy with the result in \cite[(4.24)]{Gojems},  on the component $\sph^2_-\times\{s\} \In \Lambda B \backslash \Bab$ these formulae, together with \eqref{E:orbetaD}, \eqref{E:orbetaB}, \eqref{E:FibAct} and \eqref{E:FibCurv}, yield 
\begin{align}
  \begin{split}
  \label{E:EquivEta}
    2 \eta_{\Lambda B}(\D_{\sph^3})
    &= - \frac{1}{2 \sin \bigl( \frac{a_2 + a_3}{a_1} \bigl(\pi s + \frac{\check e^{23}}{2\pi i} \bigr)\bigr)
	\sin \bigl( \frac{a_3 - a_2}{a_1} \bigl(\pi s + \frac{ \check e^{23}}{2\pi i} \bigr)\bigr)}\,,\\[1mm]
    2 \eta_{\Lambda B}(\B_{\sph^3})
    &= - \cot \biggl( \frac{a_2 + a_3}{a_1} \biggl(\pi s + \frac{\check e^{23}}{2\pi i} \biggr)\biggr)
	\cot \biggl( \frac{a_3 - a_2}{a_1} \biggl(\pi s + \frac{\check e^{23}}{2\pi i} \biggr)\biggr) \,.
  \end{split}
\end{align}
For the sake of notation below, let 
$$
q = a_1,  \quad  p_1 = 4,  \quad  p_2 = a_2 + a_3 ,  \quad p_3 = a_3 - a_2 
$$
Now, by combining the expressions obtained in \eqref{E:AhatCont}, \eqref{E:LhatCont} and \eqref{E:EquivEta}, one obtains that the integrand on $\sph^2_- \x \{s\}$ is given by
\begin{align*}
&\frac{1}{2} \hat A_{\Lambda B}(TB, \nabla^{TB})\, 2 \, \eta_{\Lambda B}(\D_{\sph^3}) 
 + \frac{1}{2^5 \! \cdot \! 7} \hat L_{\Lambda B}(TB, \nabla^{TB})\, 2 \, \eta_{\Lambda B}(\B_{\sph^3}) \\
&\hspace*{15mm}= - \frac{i}{q \cdot 2^4 \cdot 7} \left(
\frac{14 + \prod_{\ell = 1}^3 \cos \left(\frac{p_\ell \pi s}{q} + \frac{ p_\ell}{q} \frac{\check e^{23}}{2 \pi i} \right)}
{\prod_{\ell = 1}^3 \sin \left(\frac{p_\ell \pi s}{q} + \frac{ p_\ell}{q}  \frac{\check e^{23}}{2 \pi i} \right)} \right) .
\end{align*}
The goal now is to extract the degree-two term from this expression, that is, the term involving the volume form $\check e^{23}$.  By expanding the respective formal power series and noting that $(\check e^{23})^k = 0$ for $k > 1$, one obtains
\begin{align*}
\sin \left(\frac{p_\ell \pi s}{q} + \frac{ p_\ell}{q}  \frac{\check e^{23}}{2 \pi i} \right) 
&= \sin \left(\frac{p_\ell \pi s}{q} \right) + \frac{ p_\ell}{q}  \cos \left(\frac{p_\ell \pi s}{q} \right) \frac{\check e^{23}}{2 \pi i} \,, \\[1mm]
\cos \left(\frac{p_\ell \pi s}{q} + \frac{ p_\ell}{q}  \frac{\check e^{23}}{2 \pi i} \right) 
&= \cos \left(\frac{p_\ell \pi s}{q} \right) - \frac{ p_\ell}{q} \sin \left(\frac{p_\ell \pi s}{q} \right) \frac{\check e^{23}}{2 \pi i} \,.
\end{align*}
This observation yields, in particular, that
\begin{align*}
\frac{\cos \left(\frac{p_\ell \pi s}{q} + \frac{ p_\ell}{q} \frac{\check e^{23}}{2 \pi i} \right)}
{\sin \left(\frac{p_\ell \pi s}{q} + \frac{ p_\ell}{q} \frac{\check e^{23}}{2 \pi i} \right) } \! 
&= \frac{\cos \left(\frac{p_\ell \pi s}{q} + \frac{ p_\ell}{q} \frac{\check e^{23}}{2 \pi i} \right)}
{ \sin \left(\frac{p_\ell \pi s}{q} + \frac{ p_\ell}{q} \frac{\check e^{23}}{2 \pi i} \right) } 
\cdot 
\frac{\sin \left(\frac{p_\ell \pi s}{q} \right) \! - \frac{ p_\ell}{q} \cos \left(\frac{p_\ell \pi s}{q} \right) \! \frac{\check e^{23}}{2 \pi i}}
{\sin \left(\frac{p_\ell \pi s}{q} \right) \! - \frac{ p_\ell}{q} \cos \left(\frac{p_\ell \pi s}{q} \right) \! \frac{\check e^{23}}{2 \pi i}} \\
&= \frac{\sin \left(\frac{p_\ell \pi s}{q} \right) \cos \left(\frac{p_\ell \pi s}{q} \right) - \frac{ p_\ell}{q} \frac{\check e^{23}}{2 \pi i}}
{\sin^2 \left(\frac{p_\ell \pi s}{q} \right)} \;.
\end{align*}

From this one deduces that
\begin{align*}
&\frac{14 + \prod_{\ell = 1}^3 \cos \left(\frac{p_\ell \pi s}{q} +  \frac{p_\ell}{q} \frac{\check e^{23}}{2 \pi i} \right)}
{\prod_{\ell = 1}^3 \sin \left(\frac{p_\ell \pi s}{q} +  \frac{p_\ell}{q} \frac{\check e^{23}}{2 \pi i} \right)} \\
&\hspace*{17mm}= 
\frac{14 \prod_{\ell=0}^{3} \left(\sin \left(\frac{p_\ell \pi s}{q} \right) -  \frac{p_\ell}{q} \cos \left(\frac{p_\ell \pi s}{q} \right) \frac{\check e^{23}}{2 \pi i} \right) }
{\prod_{\ell = 1}^3 \sin^2 \left(\frac{p_\ell \pi s}{q} \right)} \\
&\hspace*{27mm}
+  \frac{\prod_{\ell = 1}^3 \left(\sin \left(\frac{p_\ell \pi s}{q} \right) \cos \left(\frac{p_\ell \pi s}{q} \right) -  \frac{p_\ell}{q} \frac{\check e^{23}}{2 \pi i} \right)}
{\prod_{\ell = 1}^3 \sin^2 \left(\frac{p_\ell \pi s}{q} \right)} \\
&\hspace*{17mm}= 
\left(\ \! - \sum_{\mathclap{\substack{(i,j,k) = \\ \circlearrowright (1,2,3)}}} \; 
\frac{\frac{p_i}{q}  \left( 14 \cos \left(\frac{p_i \pi s}{q} \right) + \cos \left(\frac{p_j \pi s}{q} \right) \cos \left(\frac{p_k \pi s}{q} \right) \right)}
{\sin^2 \left(\frac{p_i \pi s}{q} \right) \sin \left(\frac{p_j \pi s}{q} \right) \sin \left(\frac{p_k \pi s}{q} \right) }
\right) \frac{\check e^{23}}{2 \pi i} \\
&\hspace*{27mm}
+ \frac{ 14 + \prod_{\ell = 1}^3 \cos \left(\frac{p_\ell \pi s}{q} \right) }
{\prod_{\ell = 1}^3 \sin \left(\frac{p_\ell \pi s}{q} \right)} \,.
\end{align*}

Since $\sph^2_-$ has constant curvature $4$, hence volume $\pi$, the contribution of the twisted sectors $\sph^2_- \x \{1, \dots, \frac{|q| - 1}{2}\}$ to the Eells-Kuiper invariant is, therefore, given by
\begin{align*}
&\frac{1}{2} \int_{\Lambda B \backslash \Bab} \hat A_{\Lambda B}(TB, \nabla^{TB})\, 2 \, \eta_{\Lambda B}(\D_{\sph^3}) \\[-1mm]
&\hspace*{50mm}+ \frac{1}{2^5 \! \cdot \! 7}  \int_{\Lambda B \backslash \Bab} \hat L_{\Lambda B}(TB, \nabla^{TB})\, 2 \, \eta_{\Lambda B}(\B_{\sph^3}) \\[2mm]
&\hspace*{15mm}= \frac{1}{q^2 \cdot 2^5 \cdot 7}
\sum_{s = 1}^{\frac{|q| - 1}{2}}
\ \sum_{\mathclap{\substack{(i,j,k) = \\ \circlearrowright (1,2,3)}}} \; 
\frac{ p_i \left( 14 \cos \left(\frac{p_i \pi s}{q} \right) + \cos \left(\frac{p_j \pi s}{q} \right) \cos \left(\frac{p_k \pi s}{q} \right) \right)}
{\sin^2 \left(\frac{p_i \pi s}{q} \right) \sin \left(\frac{p_j \pi s}{q} \right) \sin \left(\frac{p_k \pi s}{q} \right) } \\
&\hspace*{15mm}= \frac{1}{q^2 \cdot 2^6 \cdot 7}
\sum_{s = 1}^{|q| - 1}
\ \sum_{\mathclap{\substack{(i,j,k) = \\ \circlearrowright (1,2,3)}}} \; 
\frac{ p_i \left( 14 \cos \left(\frac{p_i \pi s}{q} \right) + \cos \left(\frac{p_j \pi s}{q} \right) \cos \left(\frac{p_k \pi s}{q} \right) \right)}
{\sin^2 \left(\frac{p_i \pi s}{q} \right) \sin \left(\frac{p_j \pi s}{q} \right) \sin \left(\frac{p_k \pi s}{q} \right) } \\
&\hspace*{15mm}=
\mc D(q; p_1, p_2, p_3) \\
&\hspace*{15mm}=
\mc D \left(a_1; 4, a_2 + a_3, a_3 - a_2 \right) 
= -D(\ul a) ,
\end{align*}
where the second equality follows from the invariance of the summands under the map $s \mapsto q - s$ and the final equality from the remarks preceding Theorem \ref{T:thmC}.  Replacing $\ul a$ with $\ul b$ and applying the isometry $\Psi$ of \eqref{E:Mirror} yields the analogous contribution $+D(\ul b)$ of the twisted sectors $\sph^2_+ \x \{1, \dots, \frac{|b_1| - 1}{2}\}$.
\end{proof}

Despite their complicated appearance, it is sometimes straightforward to compute the generalised Dedekind sums $\mc D(q; p_1, p_2, p_3)$, where $\gcd(q, p_i) = 1$ for $i = 1,2,3$.  For example, in the case $q =1$, it is clear that $\mc D(1; p_1, p_2, p_3)= 0$.  A non-trivial situation which arises in Corollary \ref{C:nonMilnor} is detailed below.

\begin{example}
\label{Ex:DedSum}
Consider the case $q = -3$ and $p_i = 2 x_i$, $i = 1,2,3$, where $x_i \in \Z$ satisfies $\gcd(3, x_i) = 1$ for all $i \in \{ 1,2,3 \}$.  Then, for all $i \in \{ 1,2,3 \}$ and $\ell \in \{1,2\}$, one has $\cos \left(-\frac{2 x_i \pi \ell}{3} \right) = - \frac{1}{2}$ and  
$
\sin \left(-\frac{2 x_i \pi \ell}{3} \right) = - \varrho(x_i \ell) \, \frac{\sqrt3}{2}
$, 
 where $\varrho : \Z \to \{0, \pm 1\}$ is defined by
$$
\varrho(x) = 
\begin{cases}
\phantom{-}0, & \text{if } x \equiv 0 \mod 3,\\
\phantom{-}1, & \text{if } x \equiv 1 \mod 3,\\
-1, & \text{if } x \equiv 2 \mod 3.
\end{cases}
$$

Therefore, for cyclic permutations $(i,j,k)$ of $(1,2,3)$ and $\ell \in \{1,2\}$, one has
$$
14 \cos \left(-\frac{2 x_i \pi \ell}{3} \right) + \cos \left(-\frac{2 x_j \pi \ell}{3} \right) \cos \left(-\frac{2 x_k \pi \ell}{3} \right) = - \frac{27}{4} \,,
$$
whereas
$$
\sin^2 \left(-\frac{2 x_i \pi \ell}{3} \right) \sin \left(-\frac{2 x_j \pi \ell}{3} \right) \sin \left(-\frac{2 x_k \pi \ell}{3} \right) = \varrho(x_j \ell) \varrho(x_k \ell) \, \frac{9}{16} \,.
$$
However, as the sign of $\varrho(x_i \ell)$ changes depending on the choice of $\ell \in \{1,2\}$, it follows that both expressions are independent of $\ell$ and, hence, that
\begin{align*}
\mc D(-3; 2x_1, 2x_2, 2x_3) &= - \frac{2^4 \cdot 3}{2^6 \cdot 3^2 \cdot 7} \ \ \sum_{\mathclap{\substack{(i,j,k) = \\ \circlearrowright (1,2,3)}}} \,  \varrho(x_j) \varrho(x_k) \, x_i  \\
&= - \frac{1}{2^2 \cdot 3 \cdot 7} \ \ \sum_{\mathclap{\substack{(i,j,k) = \\ \circlearrowright (1,2,3)}}} \,  \varrho(x_j) \varrho(x_k) \, x_i \,.
\end{align*}

Notice, in particular, that $\mc D(-3; 2x_1, 2x_2, 2x_3)  \in \frac{1}{28} \Z$\,, since $x_i \not\equiv 0$ mod $3$ for every $i \in \{1,2,3\}$ ensures that the numerator of $\mc D(-3; 2x_1, 2x_2, 2x_3)$ is always divisible by $3$.

As an application of this formula to the situation in Corollary \ref{C:nonMilnor}, consider $D(\ul a) = \mc D(a_1; 4, a_2 + a_3, a_2 - a_3)$ for $\ul a = (-3, 12k - 3, 12 l + 1)$, $k, l \in \Z$.  In this case, $x_1 = 2$, $x_2 = 6(k+l)-1$ and $x_3 = 6(k - l) - 2$, which ensures that $\varrho(x_1) = \varrho(x_2) = -1$ and $\varrho(x_3) = 1$.  It now easily follows that $D(\ul a) =  \frac{4 l +1}{28}$.
\end{example}


\medskip
\subsection{The contribution of the very small eigenvalues} \hspace*{1mm}\\
\label{SS:VerySmallEVs}

Recall that the term $\frac{1}{2^5 \cdot 7} \lim_{\ve \to 0} \tau_\ve$ in the formula for the Eells-Kuiper invariant given in Corollary \ref{C:AdLimEK} is the signature of the quadratic form \eqref{E:QuadForm} coming from the $E_4$-page of a Leray-Serre spectral sequence for the Seifert fibration $\pi : (\Mab, g_M) \to (\Bab, g_B)$.

\begin{thm}
\label{T:VerySmallEVs}
If $n \neq 0$ then, with respect to the metric $g_M$ on $\Mab$ given in Proposition \ref{P:metric}, the contribution of the very small eigenvalues of the odd signature operator $\B$ to the adiabatic-limit formula for the Eells-Kuiper invariant $\mu(\Mab)$ is given by
$$
\frac{1}{2^5 \cdot 7} \lim_{\ve \to 0} \tau_\ve = \frac{|n|}{2^5 \cdot 7 \cdot n} \,.
$$ 
\end{thm}

\begin{proof}
As in \cite[Section 4.g.]{Gojems}, given that the entries on the $E_4$-page are trivial except for $E_4^{ij} = \R$ whenever $i \in \{0,4\}$, $j \in \{0, 3\}$,  it suffices to determine the sign of the integral $\int_M \xi \, d\xi$\,, where $\xi \in \Omega^3(\Mab)$ is a $3$-form such that the fibrewise integral is nowhere zero, and such that $d\xi \in \pi^* \Omega^4(\Bab)$ is basic.  Consider the $3$-form
$$
\xi = \begin{cases}
2 \bar f^{123} - \Bigl( \frac{u'}{h} \, \bar e^{01} - 2 u \, \bar e^{23} \Bigr) \, \bar x^1
 - \Bigl( \frac{v'}{h} \, \bar e^{01} - 2 v \, \bar e^{23} \Bigr) \bar f^1 & \text{ on } M_- \text{\,, and}\\[2mm]
2 \bar f^{123} - \Bigl( \frac{u'}{h} \, \bar e^{02} + 2 u \, \bar e^{13} \Bigr) \, \bar x^2
 - \Bigl( \frac{v'}{h} \, \bar e^{02} + 2 v \, \bar e^{13} \Bigr) \bar f^2 & \text{ on } M_+ \,.
\end{cases}
$$
It is clear that $\xi|_{M_- \cap M_+} = 2 \bar f^{123}$ and that the fibrewise integral is nowhere zero, as desired.  Furthermore, from \eqref{E:xiDiff} and \eqref{E:df123} it follows that
$$
d \xi = \frac{4 u u' - 4 v v'}{h}\,\bar e^{0123} = \frac{2(u^2 - v^2)'}{h} \, \bar e^{0123} \in \pi^*\Omega^4(\Bab) \,.
$$
Since the leaves $\tau^{-1}(t) \in \Mab$ have volume $\frac{\pi^4 h}{2}$, the result now follows from
\begin{align*}
\int_M \xi \, d\xi &= \int_M \frac{4(u^2 - v^2)'}{h} \, \bar e^{0123} \bar f^{123} \\
&= 2 \pi^4 \int_{-1}^1  (u^2 - v^2)' \, dt \\
&= \frac{16 \pi^4 n}{a_1^2 b_1^2} \,.
\end{align*}
\end{proof}


\medskip
\subsection{The Eells-Kuiper invariant}\hspace*{1mm}\\
\label{SS:EKinv}

Combining the results of the previous sections, it is finally possible to compute the Eells-Kuiper invariant of $\Mab$.

\begin{thm}
\label{T:EKinv}
If $n \neq 0$, then the Eells-Kuiper invariant of $\Mab$ is given by
$$
\mu(\Mab) = \frac{|n| - a_1^2 b_1^2 m^2}{2^5 \cdot 7 \cdot n} - D(\ul a) + D(\ul b) \mod 1 \ \in \Q/\Z \,.
$$
\end{thm}

\begin{proof}
Equip $\Mab$ with the metric $g_M$ given in Proposition \ref{P:metric}.  Using the adiabatic-limit formula in Corollary \ref{C:AdLimEK}, the claimed expression for $\mu(\Mab)$ now follows immediately from Theorems \ref{T:PontTerm}, \ref{T:etaFormTerm} and \ref{T:VerySmallEVs}.
\end{proof}


\bibliographystyle{alpha}

\end{document}